\RequirePackage{fix-cm}
\documentclass{svjour3}                     
\smartqed  

\usepackage[utf8]{inputenc}
\usepackage{graphics,graphicx,epstopdf}
\usepackage[leqno]{amsmath}
\usepackage{amsxtra,amsfonts,amscd,amssymb,bm,epsf}
\usepackage{cases}
\usepackage{subfig}
\usepackage{xcolor}
\usepackage{stmaryrd}
\usepackage{dsfont}
\usepackage{enumitem}
\usepackage{mathtools}
\usepackage{setspace,url}
\usepackage{tabularx}
\usepackage{longtable}
\usepackage{comment}
\usepackage[algo2e,linesnumbered,vlined,ruled]{algorithm2e}
\usepackage[normalem]{ulem}
\usepackage{times}
\usepackage{amsopn}

\numberwithin{definition}{section}
\numberwithin{corollary}{section}
\numberwithin{lemma}{section}
\numberwithin{theorem}{section}
\numberwithin{equation}{section}

\newtheorem{assumption}[theorem]{Assumption}

\definecolor{tumb}{RGB}{0,101,189}

\newcommand{\R}{\mathbb{R}}
\newcommand{\Rn}{\R^n}

\newcommand{\Sn}{\mathbb{S}^{n^2-1}}
\newcommand{\Smn}{\mathbb{S}^{nm-1}}
\newcommand{\Snn}{\mathbb{S}^{n-1}}
\newcommand{\Smm}{\mathbb{S}^{m-1}}

\newcommand{\cA}{\mathcal{A}}

\newcommand{\cD}{\mathcal{D}}

\newcommand{\cI}{\mathcal{I}}  
 
\newcommand{\cL}{\mathcal{L}} 
\newcommand{\cM}{\mathcal{M}}
\newcommand{\cN}{\mathcal{N}}  
\newcommand{\cS}{{\mathcal{S}}}

\newcommand{\cMC}{\mathcal{MC}}
\newcommand{\cSD}{\mathcal{SD}}

\newcommand{\dist}{\mathrm{dist}} 
\newcommand{\rank}{\mathrm{rank}}

\newcommand{\st}{\mathrm{s.\,t.}}

\DeclareMathOperator*{\argmax}{arg\,max}

\newcommand{\revise}[1]{{\color{black}#1}}
\newcommand{\revisee}[1]{{\color{black}#1}}

\usepackage{ulem}

\colorlet{color1}{blue}
\colorlet{color2}{red!50!black}

\usepackage{hyperref}[6.83]
\hypersetup{
  colorlinks=true,
  frenchlinks=false,
  pdfborder={0 0 0},
  naturalnames=false,
  hypertexnames=false,
  breaklinks,
  allcolors = color1,
  urlcolor = color2,
}
\begin{document}

\title{A New Complexity Metric for Nonconvex Rank-one Generalized Matrix Completion
\thanks{We note that a similar complexity metric based on a special case of instances in Section \ref{sec:aistats} was proposed in our conference paper \cite{yalcin2021factorization}. However, the complexity metric in this work has a different form and is proved to work on a broader set of applications. In addition, we prove several theoretical properties of the metric in this work, which are not included in \cite{yalcin2021factorization}.\\
Corresponding Author: Javad Lavaei\\
Affiliation: Department of Industrial Engineering and Operations Research, University of California, Berkeley\\
\email{lavaei@berkeley.edu}
}
}

\titlerunning{New Complexity Metric for Generalized Matrix Completion}        

\author{    Haixiang Zhang  \and 
            Baturalp Yalcin \and 
            Javad Lavaei    \and 
            Somayeh Sojoudi
}


\institute{Haixiang Zhang  \at
              Department of Mathematics, University of California, Berkeley, CA \\
              \email{haixiang\_zhang@berkeley.edu}            
        \and
           Baturalp Yalcin \at 
           Department of Industrial Engineering and Operations Research, University of California, Berkeley, CA \\
           \email{baturalp\_yalcin@berkeley.edu}  
         \and 
           Javad Lavaei \at 
           Department of Industrial Engineering and Operations Research, University of California, Berkeley, CA \\
           \email{lavaei@berkeley.edu} 
        \and
           Somayeh Sojoudi \at 
           Department of Electrical Engineering and Computer Science, University of California, Berkeley, CA \\
           \email{sojoudi@berkeley.edu}
}

\date{Received: date / Accepted: date}

\maketitle

\begin{abstract}

In this work, we develop a new complexity metric for an important class of low-rank matrix optimization problems in both symmetric and asymmetric cases, where the metric aims to quantify the complexity of the nonconvex optimization landscape of each problem and the success of local search methods in solving the problem. The existing literature has focused on two \revisee{recovery guarantees}. The RIP constant is commonly used to characterize the complexity of matrix sensing problems. On the other hand, the incoherence and the sampling rate are used when analyzing matrix completion problems. The proposed complexity metric has the potential to \revisee{generalize} these two notions and also applies to a much larger class of problems. To mathematically study the properties of this metric, we focus on the rank-$1$ generalized matrix completion problem and illustrate the usefulness of the new complexity metric on three types of instances, namely, instances with the RIP condition, instances obeying the Bernoulli sampling model, and a synthetic example. We show that the complexity metric exhibits a consistent behavior in the three cases, even when other existing conditions fail to provide theoretical guarantees. These observations provide a strong implication that the new complexity metric has the potential to \revisee{generalize} various conditions of optimization complexity proposed for different applications. Furthermore, we establish theoretical results to provide sufficient and necessary conditions for the existence of spurious solutions in terms of the proposed complexity metric. This contrasts with the RIP and incoherence conditions that fail to provide any necessary condition.

\end{abstract}

\section{Introduction}
\label{sec:intro}

A variety of modern signal processing and machine learning applications require solving optimization problems that involve a low-rank matrix variable. More specifically, given measurements to some unknown ground truth matrix $M^*\in\R^{n\times n}$ of rank $r \ll n$, the \textit{low-rank matrix optimization} problem can be formulated as
\begin{align}\label{eqn:obj}
    \min_{M\in\R^{n\times n}} f(M;M^*)\quad \st\quad M\succeq 0,\quad \rank(M) \leq r,
\end{align}
where $f(\cdot;M^*)$ is the loss function that penalizes the mismatch between the input matrix and $M^*$. The goal is to recover the matrix $M^*$ via \eqref{eqn:obj}. Examples of this problem include matrix sensing \cite{recht2010guaranteed,zhang2019sharp,zhang2021general}, matrix completion \cite{candes2009exact,candes2010power,ge2017no}, phase retrieval \cite{candes2015phase,sun2018geometric,chen2019gradient} and robust principle component analysis \cite{candes2011robust,fattahi2020exact}; see the review papers \cite{chen2020nonconvex,chi2019nonconvex} for more applications. \revise{The asymmetric version of problem \eqref{eqn:obj} eliminates the condition $M\succeq0$ and allows $M$ to be a non-square matrix.} To deal with the nonconvex rank constraint, there have been several works on the convex relaxations of problem \eqref{eqn:obj}. More concretely, one may replace the rank constraint with a nuclear norm regularizer \cite{candes2009exact,recht2010guaranteed,candes2010power,candes2011robust,levin2022effect}. The convex relaxation approach is proven to achieve the optimal sampling complexity for various statistical models. \revise{In the special case when $f(\cdot;M^*)$ is a linear function, the sketching method \cite{yurtsever2021scalable} can be applied to accelerate the computation.} However, \revise{for most applications of problem \eqref{eqn:obj},} the convex relaxation approach needs to update a matrix variable in each iteration, which relies on the Singular Value Decomposition (SVD) of the matrix variable. This will lead to an $O(n^3)$ computational complexity in each iteration and an $O(n^2)$ space complexity, which are prohibitively high for large-scale problems; see the numerical comparison in \cite{zheng2015convergent}. 

To improve the computational efficiency, an alternative approach was proposed by Burer and Monteiro \cite{burer2003nonlinear}, which is named as the Burer-Monteiro factorization approach. The factorization approach is based on the fact that the mapping $U \mapsto UU^T$ is surjective onto the manifold of positive semi-definite matrices of rank at most $r$, where $U\in\R^{n\times r}$. Therefore, problem \eqref{eqn:obj} is equivalent to
\begin{align}
\label{eqn:obj-bm}
    \min_{U\in\R^{n\times r}} f(UU^T;M^*),
\end{align}
which is an unconstrained nonconvex problem. A major difficulty about nonconvex optimization problems is the existence of spurious local minima\footnote{A point $U^0$ is called a spurious local minimum if it is a local minimum of problem \eqref{eqn:obj-bm} and $U^0(U^0)^T \neq M^*$.}. In general, common local search methods are only able to guarantee a point approximately satisfying the first-order and the second-order necessary optimality conditions. Therefore, local search methods with a random initialization will likely be stuck at spurious local minima and unable to converge to the global solution.
However, despite the aforementioned issue of nonconvex optimization problems, simple iterative algorithms such as gradient descent and alternating minimization have achieved empirical success in a wide range of applications. In recent years, substantial progress has been made on the theoretical understandings of these algorithms, which generally focused on proving the \revise{absence} of spurious local minima. For example, the alternating minimization algorithm was first studied in \cite{jain2013low,netrapalli2013phase,netrapalli2014non}. The (stochastic) gradient descent algorithm, which is in general easier to implement than the alternating minimization algorithm, was analyzed in \cite{candes2015phase,tu2016low,yi2016fast,chen2019gradient,chen2020nonconvex}. Besides algorithmic analysis, a critical geometric property named the strict-saddle property \cite{sun2018geometric} was established in \cite{ge2017no,sun2018geometric,zhu2018global,zhang2021general}, which can guarantee the polynomial-time global convergence of various saddle-escaping algorithms \cite{cartis2011adaptive,jin2018accelerated,allen2018neon2}. 

\revisee{\textit{Complexity metrics} are useful to characterize the behavior of local search methods for problem \eqref{eqn:obj-bm}.} A small complexity metric implies that the landscape of problem \eqref{eqn:obj-bm} is benign and thus, local search methods with random initialization converge to global solutions with high probability. Otherwise, if the complexity metric takes a large value, problem \eqref{eqn:obj-bm} may have spurious local minima, which will imply the failure of most local search methods. \revisee{However, the existing so-called ``complexity metrics'' for problem \eqref{eqn:obj-bm} are only able to guarantee a benign landscape when the complexity is small and fail to prove the existence of spurious local minima when the complexity is large. To differentiate with true complexity metrics, we use the term \textit{recovery guarantees} to reflect such weaker properties. In addition, the existing recovery guarantees were designed separately for different applications.} As a result, several different \revisee{bounds} were proposed to characterize the optimization complexity of problem \eqref{eqn:obj-bm}. For example, in the context of matrix sensing problems, the following Restrict Isometry Property (RIP) is usually assumed:
\begin{definition}[\cite{recht2010guaranteed,zhu2018global}]
Given natural numbers $r$ and $s$, the function $f(\cdot;M^*)$ is said to satisfy the \textbf{Restricted Isometry Property} (RIP) of rank $(2r,2s)$ for a constant $\delta\in[0,1)$, denoted as $\delta$-RIP$_{2r,2s}$, if 
\begin{align}\label{eqn:rip} 
(1-\delta)\| K\|_F^2 \leq \left[\nabla^2 f(M;M^*)\right] (K,K) \leq (1+\delta)\| K\|_F^2 
\end{align}
holds for all matrices $M,K\in\R^{n\times n}$ such that $\rank(M)\leq 2r,\rank(K) \leq 2s$, where $\left[\nabla^2 f(M;M^*)\right](\cdot,\cdot)$ is the curvature of the Hessian at point $M$.
\end{definition}
One important class of matrix sensing problems is the \textit{linear matrix sensing problem}, which is induced by linear measurements of the ground truth matrix $M^*$. If the $\ell_2$-loss is used, the linear matrix sensing problem can be formulated as
\begin{align}
\label{eqn:ms}
    \min_{U\in\R^{n\times r}} \frac{1}{m} \sum_{i=1}^m \langle A_i, UU^T - M^* \rangle^2,
\end{align}
where $m\in\mathbb{N}$ is the number of measurements modeled by the known measurement matrices $A_i \in \R^{n\times n}$ for all $i\in[m]$. In the special case when each matrix $A_i$ is an independently identically distributed Gaussian random matrix, the $\delta$-RIP$_{2r,2s}$ condition holds with high probability if $m=O(nr\delta^{-2})$ \cite{candes2011tight}. The RIP constant $\delta$ plays a critical role in bounding the optimization complexity of problem \eqref{eqn:obj-bm}. In \cite{bi2021local}, the authors showed that the strict-saddle property holds for problem \eqref{eqn:obj-bm} if the $\delta$-RIP$_{2r,2r}$ condition holds with $\delta < 1/2$ and the ground truth matrix satisfies $\rank(M^*)=r$. On the other hand, counterexamples have been constructed in \cite{zhang2019sharp,zhang2021general} to illustrate that the strict-saddle property can fail under the $\delta$-RIP$_{2r,2r}$ condition with $\delta \geq 1/2$. 

Despite these strong theoretical results under the RIP assumption, there exists a large number of applications that do not satisfy the RIP condition. One of those applications without the RIP condition is the \textit{matrix completion problem}. Given a set of indices $\Omega\subset[n]\times [n]$, the matrix completion problem aims at recovering the low-rank matrix $M^*$ from the available entries $M_{ij}^*$ for $(i,j)\in\Omega$. With the least squares loss function, the matrix completion problem can be formulated as
\begin{align}
\label{eqn:mc}
    \min_{U\in\R^{n\times r}} {\sum}_{(i,j)\in\Omega} \left[ (UU^T)_{ij} - M^*_{ij} \right]^2.
\end{align}
The matrix completion problem \eqref{eqn:mc} is a special case of the matrix sensing problem \eqref{eqn:ms}, where each measurement matrix $A_i$ has exactly one nonzero entry. However, the RIP$_{2r,2r}$ condition does not hold for problem \eqref{eqn:mc} unless all entries of $M^*$ are observed, namely, when $\Omega=[n]\times[n]$. As an alternative to the RIP condition, the optimization complexity of problem \eqref{eqn:mc} is closely related to the incoherence of $M^*$.
\begin{definition}[\cite{candes2009exact}]
Given a constant $\mu\in[1,n]$, the ground truth matrix $M^*$ is said to be \textbf{$\mu$-incoherent} if
\begin{align}\label{eqn:incoh} 
\| e_i ^T V^* \|_F \leq \sqrt{ {\mu r}/{n} },\quad \forall i\in[n], 
\end{align}
where $V^*\Lambda^* (V^*)^T$ is the truncated SVD of $M^*$ and $e_i$ is the $i$-th standard basis of $\Rn$.
\end{definition}
Intuitively, if the ground truth $M^*$ is highly sparse, it is likely that only zero entries of $M^*$ are observed and there is no chance to learn the other entries of the matrix $M^*$. A relatively small incoherence of $M^*$ avoids this extreme case. The most popular statistical model of the measurements for problem \eqref{eqn:mc} is the Bernoulli model, where each entry of $M^*$ is observed independently with probability $p\in(0,1]$. Assuming the Bernoulli model, the incoherence of $M^*$ and the sampling probability $p$ can jointly characterize the complexity of the matrix completion problem. For example, the scaled gradient descent algorithm with a spectral initialization \cite{tong2021accelerating} converges linearly given the condition $p\geq O(\mu r^2\kappa^2 \max(\mu\kappa^2,\log{n})/n )$, where $\kappa:=\sigma_1(M^*)/\sigma_r(M^*)$ is the condition number of $M^*$. In addition, under the assumption that $p\geq O(\mu^4 r^6\kappa^6 \log{n}/n)$, the global convergence was established in \cite{ge2017no} through the strict-saddle property of a regularized version of problem \eqref{eqn:mc}. We note that the dependence on the condition number $\kappa$ may be unnecessary as shown in \cite{hardt2014fast} and that the condition number is equal to $1$ in the rank-$1$ case. On the other hand, the information-theoretical lower bound in \cite{candes2009exact} shows that $p\geq \Theta( \mu r\log(n/\delta) / n )$ is necessary for the exact completion with probability at least $1-\delta$. Therefore, the complexity of problem \eqref{eqn:mc} is closely related to the incoherence of $M^*$ and the sampling probability $p$. In the remainder of this work, we refer to the conditions on the incoherence of $M^*$ and sampling rate $p$ as \textit{incoherence conditions} when there is no confusion in the context. 

To be more rigorous, the RIP condition and the incoherence condition may have a subtle difference in their nature. As a counterpart of the incoherence condition in other low-rank matrix optimization problems, one should consider conditions in terms of the sampling complexity. On the other hand, the RIP condition is a deterministic condition on the loss function and is not related to the underlying random model. However, there is a wide range of problems that satisfy the RIP condition when the sample complexity is sufficiently large. By considering the properties of the RIP condition, we are able to analyze a large number of low-rank matrix optimization problems simultaneously. Therefore, we use the RIP condition instead of conditions based on the sample complexity as a notion of the computational complexity for those problems.

The main issue with the notions of RIP and incoherence is that they require stringent conditions to guarantee the success of local search methods for recovering $M^*$. Whenever these conditions are violated, local search methods may still work successfully, which questions whether these customized notions designed for special cases of the problem truly capture the complexity of the problem in general. Hence, it is natural to ask: 
\begin{quote}
    \textit{Does there exist a complexity metric with two properties: (i) it \revisee{is consistent with existing recovery guarantees} designed for different applications, e.g., the RIP constant $\delta$ and the incoherence $\mu$ combined with the sampling rate $p$, (ii) even when the customized conditions for different applications are violated, it still quantifies the optimization complexity of the problem in the sense that the smaller the value of this metric is, the higher the success of local search methods with random initialization is in finding the ground truth $M^*$?}
\end{quote}
In this work, we provide a partial answer to the question by developing a powerful complexity metric. To analyze the usefulness of this new metric, we focus on the rank-$1$ \revise{generalized matrix completion} problem
\begin{align}
\label{eqn:gmc}
    \min_{u\in\R^{n}} {\sum}_{i,j\in[n]} C_{ij}(u_iu_j - M^*_{ij})^2,
\end{align}
where the ground truth $M^*$ is symmetric and has rank at most $1$. The weights are $C_{ij} \geq0$ for all $i,j\in[n]$. Without loss of generality, we can assume that the matrix $C:=(C_{ij})_{i,j\in[n]}$ is symmetric since otherwise one can replace $C$ with $(C+C^T)/2$, which will not change the optimization landscape. We use $\cMC(C,u^*)$ to denote the instance of problem \eqref{eqn:gmc} with the weight matrix $C$ and the ground truth $M^*=u^*(u^*)^T$, for all $C\in\R^{n\times n}$ and $u^*\in\Rn$. The matrix completion problem \eqref{eqn:mc} is a special case of the generalized matrix completion problem \eqref{eqn:gmc}, where $C_{ij} = 1$ if $(i,j)\in\Omega$ and $C_{ij}=0$ otherwise. 

Moreover, problem \eqref{eqn:gmc} is a special case of the matrix sensing problem \eqref{eqn:ms}, where each measurement only captures one entry of $M^*$. \revise{However, the problem \eqref{eqn:gmc} still contains difficult instances of the matrix sensing problem from the perspective of the RIP condition.} In Section \ref{sec:aistats}, we show that there exists an instance of problem \eqref{eqn:gmc} that satisfies the $1/2$-RIP$_{2,2}$ condition but has spurious local minima. This counterexample implies that the optimal RIP bound in \cite{zhang2019sharp,zhang2021general} still holds for problem \eqref{eqn:gmc} and thus, problem \eqref{eqn:gmc} contains difficult instances of the matrix sensing problem. \revisee{Moreover, we show in Section \ref{sec:rip} that some of the results developed for problem \eqref{eqn:gmc} can be extended to general problem \eqref{eqn:obj-bm}.}

Now, we provide an intuition into the design of our complexity metric for problem \eqref{eqn:gmc}. For a given problem instance of \eqref{eqn:gmc}, if there exist global solutions $u^1,u^2$ such that $u^1(u^1)^T \neq u^2(u^2)^T$, it is impossible to decide which global solution corresponds to $M^*$ from the observations. Intuitively, no matter what optimization algorithm we choose and how much computational effort is exerted, there is a chance that we could not recover $M^*$ by solving problem \eqref{eqn:gmc}.
This observation motivates us to define the complexity metric to be the inverse of the infimum of the distance between any given instance and the set of instances with multiple global solutions. Since problem \eqref{eqn:gmc} is parameterized by the weight matrix $C$ and the global solution $M^*$, we are able to define the metric through norms in Euclidean spaces and their Cartesian products. In addition, in the rank-$1$ case, (random) graph theory serves as an important tool in characterizing the solvability of problem \eqref{eqn:gmc}. These two advantages enable a more thorough analysis of the new complexity metric. The formal definition of the metric is provided in Section \ref{sec:metric}. In this work, we exhibit several pieces of evidence to show that the proposed metric can serve as an alternative to the RIP constant and the incoherence, which are summarized below:
\begin{enumerate}
    \item For problem instances that satisfy the $\delta$-RIP$_{2,2}$ condition, we provide an upper bound on the complexity metric. The upper bound is tightened with extra information about the incoherence of $M^*$. 
    Similarly, for matrix completion problems obeying the Bernoulli sampling model, an upper bound on the complexity metric in terms of the incoherence of $M^*$ is derived. 
    \item We then construct a class of parameterized instances of problem \eqref{eqn:gmc}, \revise{where the RIP condition fails to provide useful guarantees}. A lower bound on the complexity metric is developed to prove that instances whose complexity metric is larger than the lower bound have an exponential number of spurious local minima. \revise{In addition, an upper bound that is consistent with the aforementioned two upper bounds is established to guarantee the \revise{absence} of spurious local minima if the complexity metric is below this bound. The consistency of the upper bounds between different types of models provides strong evidence that the new complexity metric is able to provide theoretical guarantees for different applications, even when the RIP condition or the incoherence condition fails.}
    \item We prove the existence of a non-trivial upper bound on the complexity metric. For all problem instances whose complexity metric is below this upper bound, problem \eqref{eqn:gmc} has no spurious local minima and $M^*$ can be successfully found via local search methods with random initialization. In addition, under a standard bounded-away-from-zero assumption, we show that all instances with a larger complexity metric will possess spurious local minima.
    \item \revise{We extend all results for the symmetric generalized matrix completion problem to the asymmetric case, where low-rank matrices is decomposed in to $UV^T$ for some $U\in\R^{m\times r}$ and $V\in\R^{n\times r}$ in problem \eqref{eqn:obj-bm}.}
\end{enumerate}
Based on the aforementioned results, we make some key conjectures and discuss the potential extensions of the proposed metric to more general cases of the low-rank matrix optimization problem \eqref{eqn:obj}.

\subsection{Related works}

Following the famous \textit{Netflix prize}, the theoretical analysis of problem \eqref{eqn:obj} has attracted a lot of attention in recent years; see the review papers \cite{chen2018harnessing,chi2019nonconvex}. Early attempts mainly focused on the construction of convex relaxations to rank-constrained problems \cite{candes2009exact,candes2010power,recht2010guaranteed,candes2011robust}, where the RIP condition and the incoherence condition were introduced. \revise{Recently, several modified RIP conditions were proposed to better characterize the landscapes of other classes of problems, e.g., the $\ell_1$/$\ell_2$-RIP condition \cite{li2020nonconvex}, the sign-RIP condition \cite{ma2021sign}, and the approximation and sharpness condition \cite{charisopoulos2021low}.}

Although the convex relaxation is usually guaranteed to recover the exact ground truth with almost the optimal sample complexity, the associated algorithms operate in the space of matrix variables and, thus, are computationally inefficient for large-scale problems \cite{zheng2015convergent}. Similar issues are observed for algorithms based on the Singular Value Projection \cite{jain2010guaranteed} and Riemannian optimization algorithms \cite{wei2016guarantees,wei2020guarantees,hou2020fast,ahn2021riemannian,luo2021nonconvex}. The analysis of the convex relaxation approach in the noisy case is recently conducted by bridging the convex and the nonconvex approaches \cite{chen2020noisy,chen2021bridging}. 

To deal with the difficulties in solving large-scale problems, an efficient alternative model \eqref{eqn:obj-bm} using the Burer-Monteiro factorization is considered. Despite the nonconvexity, a growing number of works demonstrated that problem \eqref{eqn:obj-bm} has benign landscapes and, therefore, is amenable for efficient optimization. Theoretical analysis stems from the alternating minimization method \cite{jain2013low,netrapalli2013phase,hardt2014understanding,hardt2014fast,netrapalli2014non,agarwal2016learning}. The alternating minimization method has the advantage that the number of iterations has only logarithmic dependence on the condition number of the ground truth \cite{hardt2014fast}. More recently, this advantage is also achieved by the scaled (sub)gradient descent algorithm \cite{tong2021accelerating,tong2021low,tong2021scaling,zhang2021preconditioned}. 

The gradient descent algorithm has also gained a significant attention due to its simplicity in implementation. In general, there are two ways to apply the gradient descent algorithm. First, the gradient descent algorithm can serve as the local refinement method after a suitable initialization \cite{candes2015phase,tu2016low,sun2016guaranteed,yi2016fast,ajayi2018provably,chen2020nonconvex}. On the other hand, the gradient descent algorithm is proved to converge globally for the phase retrieval problem \cite{chen2019gradient}. More generally, under the strict-saddle property, a number of saddle-escaping algorithms \cite{jin2018accelerated,cartis2011adaptive,allen2018neon2} converge to the global solution in polynomial time; see e.g., \cite{sun2016complete,ge2016matrix,ge2017no,zhu2018global,sun2018geometric,zhang2019sharp,chen2019model,zhang2021general,bi2020global,bi2021local,ma2021sharp}. Moreover, the gradient descent algorithm is proved to have the implicit regularization phenomenon in the over-parameterization case \cite{li2018algorithmic,chou2020gradient,stoger2021small}. 


\subsection{Notation}

The number of elements in a finite set $\cS$ is denoted as $|\cS|$. We use $\overline{\cS}$ to denote the closure of a set $\cS\subset\Rn$. The index set $\{1,\dots,n\}$ is denoted as $[n]$ for all $n\in\mathbb{N}$. The entry-wise $\ell_1$-norm and the Frobenius norm of a matrix $M$ are denoted as $\|M\|_1$ and $\|M\|_F$, respectively. The unit sphere of matrices with non-negative entries denoted as $\Sn_{+,1}$ is the set of all symmetric matrices $X\in\R^{n\times n}$ such that $\|X\|_1 = 1$ and $X_{ij}\geq0$ for all $i,j\in[n]$. Similarly, the unit sphere of vectors $\Snn_1$ is the set of all vectors $x\in\Rn$ such that $\|x\|_1 = 1$.
For every symmetric matrix $M\in\R^{n\times n}$, 
the minimum eigenvalue is denoted as $\lambda_{min}(M)$. 
The $n$-by-$n$ identity matrix is denoted as $\mathcal{I}_n$. 
The notation $M\succeq 0$ means that the matrix $M$ is symmetric and positive semi-definite. The sub-matrix $R_{i:j,k:\ell}$ consists of the $i$-th to the $j$-th rows and the $k$-th to the $\ell$-th columns of matrix $R$. For every vector $x\in\R^n$, the sets of indices corresponding to zero and nonzero components of $x$ are denoted as $\mathcal{I}_0(x)$ and $\mathcal{I}_1(x)$, respectively. For every instance $\cMC(C,u^*)$, we use $\mathbb{G}(C,u^*)=[\mathbb{V}(C,u^*),\mathbb{E}(C,u^*),\mathbb{W}(C,u^*)]$ to denote the associated weighted graph, which is defined in Section \ref{sec:metric}. The unweighted undirected graph $\mathbb{G}$ with node set $\mathbb{V}$ and edge set $\mathbb{E}$ is denoted as $\mathbb{G}=(\mathbb{V},\mathbb{E})$. The objective function of an instance $\cMC(C,u^*)$ is shown as $g(u; C,u^*) := \sum_{i,j\in[n]} C_{ij}(u_iu_j - u^*_{i}u^*_j)^2$. We use $[\nabla^2 g(M;C,u^*)](K,L) := \sum_{i,j,k,\ell} [\nabla^2 g(M;C,u^*)]_{i,j,k,\ell} K_{ij}L_{k,\ell}$ to denote the action of the Hessian $\nabla^2 g(M;C,u^*)$ on any two matrices $K$ and $L$. The notations $a_n = O(b_n)$ and $a_n = \Theta(b_n)$ mean that there exist constants $c_1,c_2 > 0$ such that $a_n\leq c_2 b_n$ and $c_1b_n\leq a_n\leq c_2b_n$ hold for all $n\in\mathbb{Z}$, respectively.

\revise{\subsection{Organization}}

In the remainder of this paper, we first define the proposed complexity metric and derive basic properties of the metric in Section \ref{sec:metric}. In Section \ref{sec:exm}, we analyze this metric under existing conditions, including the RIP condition and the incoherence condition. Section \ref{sec:theory} is devoted to the theoretical guarantees provided by the new complexity metric on the general instances of problem \eqref{eqn:gmc}. The results for the \revisee{rank-$1$} asymmetric generalized matrix completion problem are provided in Appendix \ref{sec:asym}. Finally, we conclude the paper in Section \ref{sec:cls}. Some of the proofs are provided in the appendix.

\section{New complexity metric and basic properties}
\label{sec:metric}

\revise{In this section, we first provide the formal definition of the new complexity and investigate the properties of the proposed metric. More specifically, we show that we are able to utilize the graph theory to estimate the complexity metric and calculate the minimum possible value of the proposed complexity metric in closed form.
} 
Before proceeding to the definitions, we note that the problem \eqref{eqn:gmc} is ``scale-free'' in the sense that the instance $\cMC(\eta_1 C, \eta_2 u^*)$ has the same landscape as $\cMC(C,u^*)$ up to a scaling, where $C\in\R^{n\times n}$, $u^*\in\Rn$ and $\eta_1,\eta_2 > 0$ are constants. Therefore, we may normalize the parameters $C$ and $u^*$ without loss of generality, as follows: 
\begin{assumption}\label{asp:norm}
Assume that $C\in\Sn_{+,1}$ and $u^*\in\Snn_1$, i.e., $\|C\|_1 = \|u^*\|_1 = 1$.
\end{assumption}
The above assumption excludes the degenerate cases when $C = 0$ or $M^*=0$. If $C = 0$, the objective function is always $0$ and it is impossible to recover the ground truth. For the case when $M^* = 0$, we can prove that either $u=0$ is the only stationary point or the instance $\cMC(C,0)$ has multiple different global solutions. In the first situation, the results in \cite{lee2016gradient} imply that randomly initialized gradient descent algorithm will converge to $0$ with probability $1$. In the second situation, the instance is information-theoretically unsolvable. We provide a more detailed analysis in the appendix and assume that Assumption \ref{asp:norm} holds in the remainder of the paper.

The definition of the complexity metric is closely related to the set of instances with multiple ``essentially different'' global solutions. More specifically, the set of degenerate instances is defined as
\begin{align*}
    \cD := \{ (C,u^*) ~|~ &C\in\Sn_{+,1},u^*\in\Snn_1,\\
    &\exists u\in\R^{n}\quad \st\quad g(u; C,u^*) = 0,~ uu^T \neq u^*(u^*)^T \}.
\end{align*}
Since there exist multiple global solutions to problem \eqref{eqn:gmc} if $(C,u^*)\in\cD$, it is information-theoretically impossible to find the ground truth for any instance in $\cD$. Intuitively, we say that the \textit{optimization complexity} of all instances in $\cD$ is infinity. Motivated by the above observation, we introduce the new complexity metric.
\begin{definition}[Complexity Metric]
Given arbitrary parameters $C\in\Sn_{+,1}$, $u^*\in \Snn_1$ and $\alpha \in [0,1]$, the complexity of the instance $\cMC(C,u^*)$ is defined as 
\begin{align}\label{eqn:metric}
    \mathbb{D}_\alpha(C,u^*) := \left[ \inf_{(\tilde{C},\tilde{u}^*)\in\cD} \alpha \|C - \tilde{C}\|_1 + (1-\alpha)\|u^* - \tilde{u}^*\|_1 \right]^{-1}.
\end{align}
\end{definition}
Since the set $\cD$ is bounded, the infimum in the definition is finite. The term inside the inverse operation can be viewed as a weighted distance between the point $(C,u^*)$ and the set $\cD$. In addition, we take the convention that $1/0 = +\infty$ and thus, $\mathbb{D}_\alpha(C,u^*) = +\infty$ for all $(C,u^*)\in\cD$. In this work, we choose the entry-wise $\ell_1$-norm in \eqref{eqn:metric} for the simplicity of calculations. We believe that similar theory can still be derived for other choices of the norm. \revise{We note that a similar complexity was proposed in \cite{renegar1995linear,renegar1996condition} for conic optimization and to the best of authors' knowledge, there is no similar complexity metric for nonconvex optimization problems.}

\revise{For the parameter $\alpha$, we will discuss two potential choices in this section, namely $\alpha^*$ and $\alpha^\diamond$. In the case when $\alpha = \alpha^*$, the range of the complexity metric has the largest size. Intuitively, by choosing $\alpha = \alpha^*$, the difference between the complexities of two instances will be maximized and thus, it is easier to compare the complexities of different instances. On the other hand, when we choose $\alpha = \alpha^\diamond$, the complexity metric attains its minimum possible value if and only if the $0$-RIP$_{2,2}$ condition holds. This is consistent with the intuition that instances with the RIP constant $0$ are the easiest to solve. We note that both $\alpha^*$ and $\alpha^\diamond$ satisfy $1-\alpha = \Theta(1/n)$. Moreover, in Section \ref{sec:exm}, we show that the parameter $\alpha$ strikes a balance between the RIP constant of the instance and the incoherence of the ground truth. It is still an open question what the optimal choice of parameter $\alpha$ is, which may depend on the class of problems under consideration. It may be needed to jointly consider the complexity metric with several different choices of $\alpha$ to determine the solvability of the instance.}

\subsection{Basic properties of the new complexity metric}

We first provide a more concrete characterization of the set $\cD$. In the rank-$1$ case, we are able to exactly describe the set $\cD$ using graph-theoretic notations. We introduce the associated graphs of any instance of the problem. Given an instance $\cMC(C,u^*)$, the weighted graph $\mathbb{G}(C,u^*)=[\mathbb{V}(C,u^*),\mathbb{E}(C,u^*),\mathbb{W}(C,u^*)]$ is defined by
\begin{align*}
    &\mathbb{V}(C,u^*) := [n],\quad \mathbb{E}(C,u^*) := \left\{ \{i,j\}~|~ C_{ij} > 0, i,j\in[n] \right\},\\
    &[\mathbb{W}(C,u^*)]_{ij} := C_{ij},\quad \forall i,j\in[n]\quad\st\quad\{i,j\}\in\mathbb{E}(C,u^*).
\end{align*}
%
To include the information of $u^*$, we define
\begin{align*} 
\mathcal{I}_1(C,u^*) &:= \{i\in[n] ~|~ u_{i}^* \neq 0\},\quad \mathcal{I}_0(C,u^*) := [n]\backslash \mathcal{I}_1(C,u^*),\\
\mathcal{I}_{00}(C,u^*) &:= \{ i\in\mathcal{I}_0(C,u^*) ~|~ \{ i,j \}\notin\mathbb{E}(C,u^*) ,~ \forall j\in\mathcal{I}_1(C,u^*) \}. 
\end{align*}
Intuitively, the sets $\cI_1(C,u^*)$ and $\cI_0(C,u^*)$ contain the locations of the nonzero and zero components of $u^*$.
The subset $\cI_{00}(C,u^*)$ corresponds to indices in $\cI_0(C,u^*)$ that are not connected to any index in $\cI_1(C,u^*)$. We denote the subgraph of $\mathbb{G}(C,u^*)$ induced by the index set $\cI_1(C,u^*)$ as $\mathbb{G}_1(C,u^*)=[\cI_1(C,u^*),\mathbb{E}_1(C,u^*),\mathbb{W}_1(C,u^*)]$, where $\mathbb{E}_1(C,u^*)$ and $\mathbb{W}_1(C,u^*)$ are the edge set and weight set of this subgraph. The following theorem provides an equivalent definition of $\cD$ in terms of $\cI_{00}(C,u^*)$ and $\mathbb{G}_1(C,u^*)$.
\begin{theorem}\label{thm:unique}
Given $C\in\Sn_{+,1}$ and $u^*\in\Snn_1$, it holds that $(C,u^*) \notin \cD$ if and only if
\begin{enumerate}
	\item $\mathbb{G}_1(C,u^*)$ is connected and \revise{not bipartite};
	\item $\{i,i\}\in\mathbb{E}(C,u^*)$ for all $i\in\cI_{00}(C,u^*)$.
\end{enumerate}
%
\end{theorem}
\begin{proof}
We first construct counterexamples for the necessity part and then prove the uniqueness of the global minimum (up to a sign flip) for the sufficiency part. For the notational simplicity, we fix the point $(C,u^*)$ and omit them in the notations.
\paragraph{Necessity.} In this part, our goal is to construct a solution $u\in\Rn$ such that
\[ u_iu_j = u^*_{i}u^*_j,\quad \forall \{i,j\}\in\mathbb{E};\quad uu^T \neq u^* (u^*)^T. \]
%
We denote $M^*:=u^* (u^*)^T$ and analyze three different cases below. 

\paragraph{Case I.} First, we consider the case when $\mathbb{G}_1$ is disconnected, which means that there exist two non-empty subsets $\mathcal{I}$ and $\mathcal{J}$ such that
\[ \mathcal{I} \cup \mathcal{J} = \mathcal{I}_1,\quad \mathcal{I} \cap \mathcal{J} = \emptyset;\quad \{i,j\} \notin \mathbb{E}_1 ,\quad \forall i\in\mathcal{I},\ \forall j\in\mathcal{J}. \]
We define the vector $u\in\Rn$ as
\[ u_i := 0,\quad \forall i\in\mathcal{I}_0;\quad u_i = u^*_i,\quad \forall i\in\mathcal{I};\quad u_i = -u^*_i,\quad \forall i\in\mathcal{J}. \]
The above definition leads to
\[ u_iu_j = \begin{cases} -M^*_{ij} & \text{if } i\in \mathcal{I} \text{ and } j\in\mathcal{J} \\ M^*_{ij} & \text{otherwise}.  \end{cases} \]
Since $u^*_i \neq 0$ for all $i\in\mathcal{I}_1$, it follows that $u_iu_j = - M^*_{ij} \neq M^*_{ij}$ for all $\{i,j\}$ such that $i\in\mathcal{I}$ and $j\in\mathcal{J}$. 

\paragraph{Case II.} Next, we consider the case when $\mathbb{G}_1$ is bipartite, which means that there exist two non-empty subsets $\mathcal{I}$ and $\mathcal{J}$ such that
\[ \mathcal{I} \cup \mathcal{J} = \mathcal{I}_1,\quad \mathcal{I} \cap \mathcal{J} = \emptyset;\quad \{i,j\} \notin \mathbb{E}_1 ,\quad \forall i,j\in\cI_1\quad \st\quad i,j\in \mathcal{I} \text{ or } i,j\in\mathcal{J}. \]
In this case, we define the vector $u\in\Rn$ as
\[ u_i := 0,\quad \forall i\in\mathcal{I}_0;\quad u_i := u^*_i / 2,\quad \forall i\in\mathcal{I};\quad u_i := 2u^*_i,\quad \forall i\in\mathcal{J}. \]
Now, we have
\[ u_iu_j = \begin{cases} M^*_{ij}/4 & \text{if } i,j\in \mathcal{I}\\ 4M^*_{ij} & \text{if } i,j\in \mathcal{J}\\ M^*_{ij} & \text{otherwise}.  \end{cases} \]
Since $M^*_{ij}\neq 0$ for all $i,j \in\mathcal{J}$, we have that $u_iu_j = 4M^*_{ij} \neq M^*_{ij}$ for all $i,j \in\mathcal{J}$. 

\paragraph{Case III.} Finally, we check the case when there exists a node $i_0\in\mathcal{I}_{00}$ such that $\{i_0,i_0\}\notin\mathbb{E}$. In this case, we define the vector $u\in\Rn$ as
\[ u_{i_0} := 1,\quad u_i := u^*_i,\quad \forall i\in[n] \backslash \{i_0\}. \]
Now, we have
\[ u_{i_0}u_{i_0} = 1 \neq 0 = M_{i_0i_0}^*,\quad u_iu_j = M^*_{ij},\quad \forall \{i,j\}\in\mathbb{E}. \]

Combining the above three cases completes the proof of the necessity part.

\paragraph{Sufficiency.} We prove that any global solution $u\in\Rn$ to problem \eqref{eqn:gmc} satisfies $uu^T = M^*$, where $M^*:=u^*(u^*)^T$. Since $u$ is a global solution, it follows that
\[ u_iu_j = M^*_{ij},\quad \forall i,j\in[n]\quad \st\quad \{i,j\} \in \mathbb{E}. \]
Since the graph $\mathbb{G}_1$ is \revise{not bipartite}, there exists a cycle with an odd number of edges in $\mathbb{G}_1$. We denote the length of the cycle as $2k+1$, where $k$ is a non-negative integer. Moreover, we denote the edges of the cycle as 
\[ \{i_0,i_1\},\{i_1,i_2\},\dots, \{i_{2k}, i_0\}. \]
Since $\{i_0,\dots,i_{2k}\} \subset \mathcal{I}_1$, we know that
\[ u_iu_j = M^*_{ij} \neq 0,\quad  \forall i,j\in[n]\quad \st\quad \{i,j\} \in \left\{\{i_\ell,i_{\ell+1}\},\ell\in\{0,\dots,2k\}\right\}, \]
where $i_{2k+1} := i_0$. Hence, we can calculate that
\begin{align*} 
u_0 ^2 &= \prod_{\ell=0}^{2k}(u_{i_\ell}u_{i_{\ell+1}})^{(-1)^\ell} = \prod_{\ell=0}^{2k+1}M_{i_\ell i_{\ell+1}}^{(-1)^\ell}
= (u_{i_0}^*)^2. 
\end{align*}
Without loss of generality, assume that $u_{i_0} = u_{i_0}^*$ since otherwise we can consider the solution $-u$ if $u_{i_0} = -u_{i_0}^*$. With the value of $u_{i_0}$ correctly recovered, it follows that
\[ u_{i_1} = \frac{u_{i_0}u_{i_1}}{u_{i_0}} = \frac{u^*_{i_0}u^*_{i_1}}{u^*_{i_0}} = u^*_{i_1}. \]
Similarly, we can utilize the connectivity of $\mathbb{G}_1$ to iteratively obtain $u_i = u_i^*$ for all $i\in\mathcal{I}_1$. 

The remaining part is to show that $u_i = 0$ for all $i\in\mathcal{I}_0$. For every node $i\in\mathcal{I}_0 \backslash \mathcal{I}_{00}$, there exists a node $j\in\mathcal{I}_1$ such that $\{i,j\}\in\mathbb{E}$. This implies that
\[ u_j = u^*_j \neq 0,\quad u_iu_j = M^*_{ij} = 0, \]
Hence, it holds that $u_i = 0$. For every node $i\in\mathcal{I}_{00}$, the assumption in the theorem requires that $\{i,i\}\in\mathbb{E}$, which leads to
\[ u_i^2 = M^*_{ii} = 0. \]
In this case, we also obtain $u_i =0$.
\qed\end{proof}

Since the set $\cD$ is bounded, the infimum in the definition \eqref{eqn:metric} can be attained by using the closure of $\cD$, namely
\begin{align}\label{eqn:metric-new}
    \mathbb{D}_\alpha(C,u^*) = \left[ \min_{(\tilde{C},\tilde{u}^*)\in\overline{\cD}} \alpha \|C - \tilde{C}\|_1 + (1-\alpha)\|u^* - \tilde{u}^*\|_1 \right]^{-1}.
\end{align}
The alternative definition \eqref{eqn:metric-new} simplifies the verification of parameters that attain the infimum. In addition, with the help of Theorem \ref{thm:unique}, we can exactly characterize the closure $\overline{\cD}$, which has a slightly simpler form than $\cD$.
\begin{theorem}\label{thm:unique-closure}
We have the following relation:
\begin{align*} 
\overline{\cD} &= \{ (C,u^*) ~|~C\in\Sn_{+,1},u^*\in\Snn_1, \mathbb{G}_1(C,u^*) \text{ is disconnected or bipartite} \}\\
&\hspace{7em} \cup \{ (C,u^*) ~|~ C\in\Sn_{+,1},u^*\in\Snn_1, \cI_{00}(C,u^*) \text{ is not empty} \}. 
\end{align*}
\end{theorem}
%


Let the set in the right-hand side of the above equation be called $\mathcal{D}'$. The proof of Theorem \ref{thm:unique-closure} is based on a standard technique that first shows $\bar{\mathcal{D}}\subset \mathcal{D}'$ and then shows $\mathcal{D}'\subset \bar{\mathcal{D}}$. The details can be found in Appendix \ref{adp:unique-closure}.
Using the results in Theorems \ref{thm:unique} and \ref{thm:unique-closure}, we provide an estimate on the scale of the new metric. Since $\cD$ is a bounded set, there exists an upper bound on the minimum possible value of the complexity metric, which is defined below:
\[ \mathbb{D}_\alpha^{min} := \min_{C\in\Sn_{+,1}, u^*\in\Snn_1 }~ \mathbb{D}_\alpha(C,u^*). \]
The next theorem provides the expression of $\mathbb{D}_\alpha^{min}$. 
\begin{theorem}\label{thm:max-dist}
Suppose that $n\geq 5$. Then, it holds that
\begin{align*}
    \mathbb{D}_\alpha^{min} = \begin{cases}
    \frac{n}{4\alpha} & \text{if } \alpha \leq \frac{n^2-3n-2}{n^2-5n+4}\\
    \frac{n^2}{2(1-\alpha)(n-2)n + 4\alpha} &\text{if } \frac{n}{n+2}\leq \alpha \leq \frac{n}{n+1}\\
    \frac{n(n+1)}{2(1-\alpha)(n-2)(n+1) + 4} &\text{if } \alpha \geq \frac{n}{n+1}.
    \end{cases}
\end{align*}
In the regime $(n^2-3n-2)/(n^2-5n+4) \leq \alpha \leq {n}/(n+2)$, we have the estimate
\[ \mathbb{D}_\alpha^{min} \in \left[ \frac{n}{4\alpha}, \frac{n^2}{4\alpha(n-1)} \right]. \]
%
%
%
\end{theorem}

\revisee{The proof of Theorem \ref{thm:max-dist} can be found in Appendix \ref{adp:max-dist}.} Now, we provide the proof of Theorem \ref{thm:max-dist}.
The results of Theorem \ref{thm:max-dist} imply that in the regime where $\alpha\geq\Theta(1)$ and $1-\alpha \geq \Theta(n^{-1})$, we have $\mathbb{D}_\alpha^{min} = O\left(n\right)$. This suggests that $n^{-1}\mathbb{D}_\alpha(C,u^*)$ may be a dimension-free complexity metric; see more examples supporting this claim in Section \ref{sec:exm}. In addition, the minimum possible value of the complexity is attained at
\[ \alpha^* := (n^2-5n+4)/(n^2-3n-2). \]
Hence, the set of possible values of the complexity metric attains the maximum size by choosing $\alpha=\alpha^*$. This observation hints that $\alpha^*$ may be the optimal choice of $\alpha$ since it may enable the metric to differentiate instances with different complexities to the maximum degree. Using the exact formulation of $g(\alpha,c)$ in Lemma \ref{lem:max_dist-1}, we plot the minimum possible value of the complexity metric both without scaling and after scaling by $n^{-1}$ in Figure \ref{fig:1}.
\begin{figure}[t]
\centering
\includegraphics[scale=0.43]{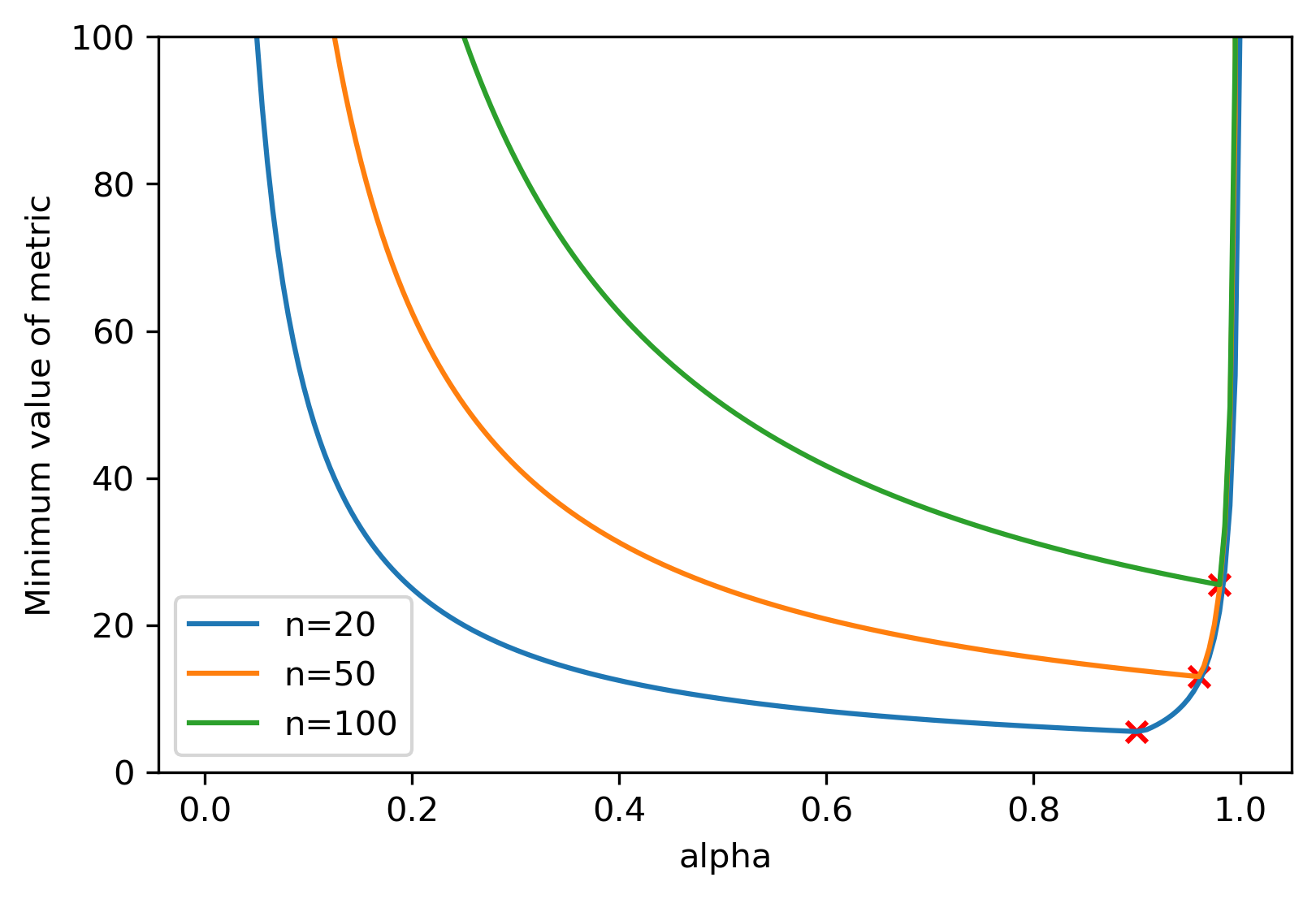}
\includegraphics[scale=0.43]{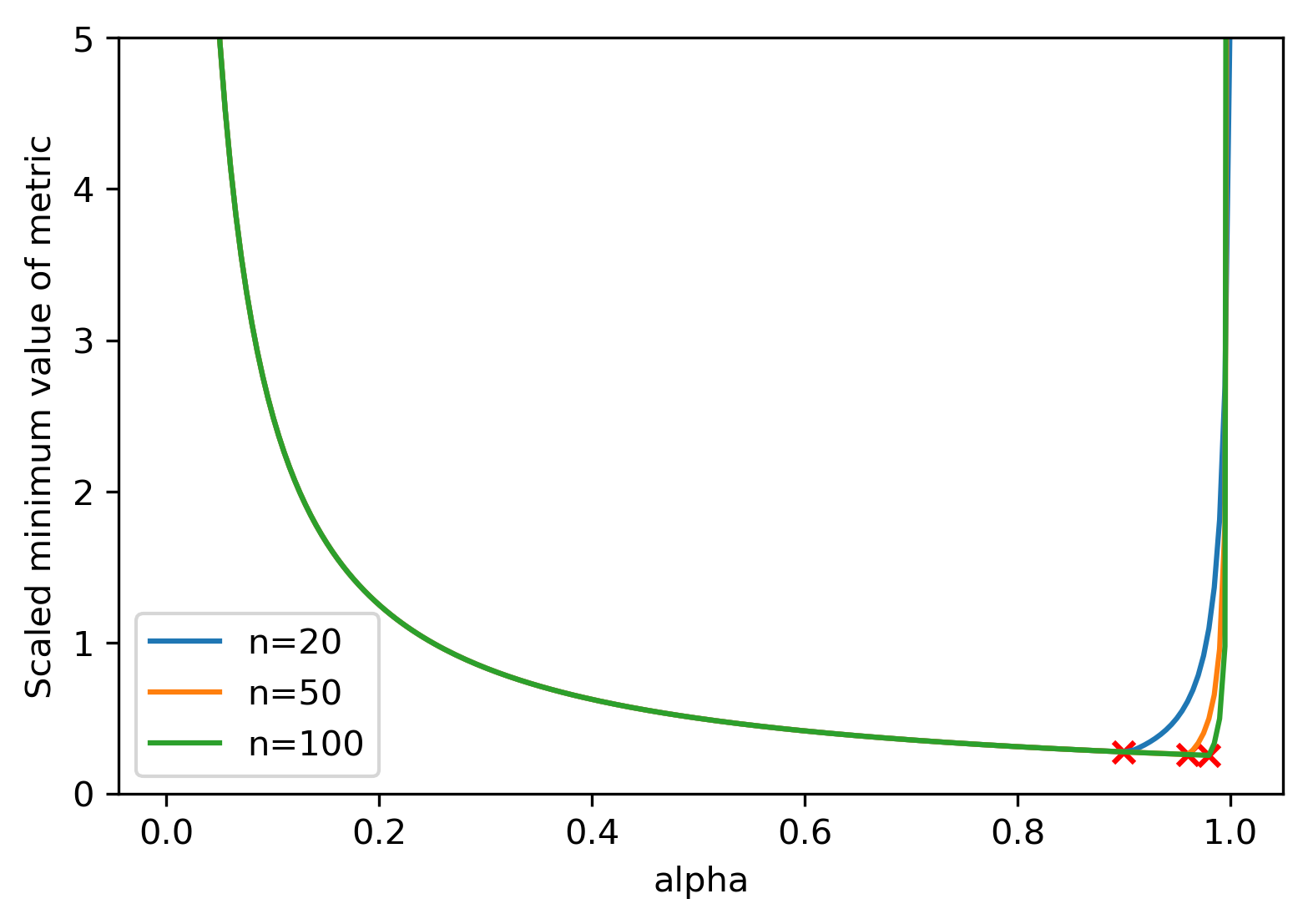}
\caption{Comparison of $\mathbb{D}_\alpha^{min}$ for $n=20,50,100$. The red ``$\times$'' sign refers to the value at $\alpha^*$. In the right plot, the complexity metric is scaled by $n^{-1}$.}
\label{fig:1}
\end{figure}
From the numerical results, we can see that the complexity scales with $n$ if $\alpha$ is smaller than $\alpha^*$, which is consistent with Theorem \ref{thm:max-dist}. If $\alpha$ is larger than $\alpha^*$, the complexity metric for different values of $n$ approximately lies on the same curve. 

In the following theorem, we show that if $\alpha=\alpha^*$, the instances that attain the minimum value of the complexity metric are unique up to sign flips to components of the global solution.
\begin{theorem}\label{thm:optimal}
Suppose that $n\geq 5$ and the instance $\cMC(C,u^*)$ satisfies
\[ \mathbb{D}_{\alpha^*}(C,u^*) = {n}/(4\alpha^*). \]
Then, it holds that
\[ |u^*_i| = 1/n ,\quad C_{ii}=0,\quad \forall i\in[n];\quad C_{ij}=1/[n(n-1)],\quad \forall i,j\in[n],~i\neq j. \]
\end{theorem}

\revisee{The proof of Theorem \ref{thm:optimal} can be found in Appendix \ref{adp:optimal}.}
The above theorem states that if we choose the weight $\alpha = \alpha^*$, the ``easiest'' instance is unique up to a change in the signs of the components of the global solution $u^*$. In the next theorem, we show that a similar property as $\alpha^*$ holds if we set $\alpha$ to be 
\[ \alpha^{\diamond} := n/(n+2). \]
%
%
\begin{theorem}\label{thm:optimal-2}
Suppose that $n\geq 5$ and the instance $\cMC(C,u^*)$ satisfies
\[ \mathbb{D}_{\alpha^\diamond}(C,u^*) = \mathbb{D}_\alpha^{min} = {n(n+2)}/[4(n-1)]. \]
Then, it holds that
\[ |u^*_i| = 1/n ,\quad \forall i\in[n];\quad C = n^{-2} I_n. \]
\end{theorem}

Since the proof is similar to that of Theorem \ref{thm:optimal}, we omit it for brevity. The above theorem implies that the weight matrix $C$ of the ``easiest'' instances is a constant multiple of the identity matrix $I_n$, which satisfies the $\delta$-RIP$_{2,2}$ condition with $\delta=0$. This is consistent with the common sense that the RIP constant $\delta$ being $0$ is the optimal situation. Hence, Theorem \ref{thm:optimal-2} suggests that the choice $\alpha^\diamond = n/(n+2)$ may potentially be the optimal choice of $\alpha$. On the other hand, we will prove in Section \ref{sec:small} that the ``easiest'' instances in Theorems \ref{thm:optimal} and \ref{thm:optimal-2} all have a benign landscape in the sense that they satisfy the strict-saddle property \cite{sun2018geometric}, which guarantees the polynomial-time global convergence of various algorithms. If the weight $\alpha$ is different from $\alpha^*$ and $\alpha^\diamond$, there may exist multiple ``essentially'' different instances attaining the minimum complexity.

\section{Connections to existing results}
\label{sec:exm}

In this section, we provide estimates of the proposed complexity metric on two well-studied problem instances and a synthetic problem. More specifically, we consider matrix sensing problems satisfying the RIP condition and matrix completion problems under the Bernoulli sampling model. In addition, we construct a class of instances parameterized by a single parameter. We estimate the threshold of the parameter that separates instances with a desirable optimization landscape from those with a bad landscape. The results in the synthetic example show that our proposed complexity metric has the potential to provide guarantees on the optimization landscape when the RIP condition fails.

\subsection{Matrix sensing problem: RIP condition}
\label{sec:rip}

We first consider instances of problem \eqref{eqn:gmc} that satisfy the $\delta$-RIP$_{2,2}$ condition, where $\delta\in[0,1)$ is the RIP constant. However, the constraint that $C\in\Sn_{+,1}$ is inconsistent with the RIP condition \eqref{eqn:rip} in the sense that the entries of $C$ are averagely on the scale of $n^{-2}$, but the RIP condition requires that the entries of $C$ be on the scale of $O(1)$. Therefore, we generalize the definition of the RIP condition to deal with the inconsistent scaling:
\begin{definition}
Given natural numbers $r$ and $s$, the function $f(\cdot;M^*)$ is said to satisfy the \textbf{Restricted Isometry Property} (RIP) of rank $(2r,2s)$ for a constant $\delta\in[0,1)$, denoted as $\delta$-RIP$_{2r,2s}$, if there exist constants $c_1,c_2 \geq 0$ such that $c_2/c_1=(1+\delta)/(1-\delta)$ and
\begin{align}\label{eqn:rip-1} 
    c_1\| K\|_F^2 \leq \left[\nabla^2 f(M;M^*)\right] (K,K) \leq c_2\| K\|_F^2 
\end{align}
holds for all matrices $M,K\in\R^{n\times n}$ such that $\rank(M)\leq 2r,\rank(K) \leq 2s$.
\end{definition}
The above definition of the RIP condition is scale-free in the sense that for any constant $c>0$, the function $cf(\cdot;M^*)$ satisfies the $\delta$-RIP$_{2r,2s}$ condition if and only if $f(\cdot;M^*)$ satisfies the same condition. 

\revisee{Since the instances satisfying the RIP condition have a benign optimization landscape, we expect that the complexity metric is upper-bounded for those instances. By suitably generalizing the definitions of $\mathbb{D}_\alpha(C,u^*)$ and $\cD$, we provide an upper bound for problem \eqref{eqn:obj-bm} under the RIP condition. Note that the ground truth $M^*$ is not necessarily rank-$1$ in this part. Instead, we assume that $M^* = U^*(U^*)^T$ is rank-$r$, where $U^*$ belongs to $\R^{n\times r}$. For problem \eqref{eqn:obj-bm}, each instance is defined by the loss function $f(\cdot;\cdot)$ and the ground truth $M^*$. We assume that the $M^*$ is a global optimum of the loss function, namely, 
\begin{align}\label{eqn:normalization-3}
f(M^*;M^*) &= \min_{K\in\mathbb{R}^{n\times n}} {f}(K;{M}^*),\quad\forall M^*\in\R^{n\times n}\quad \st \quad M^*\succeq0,~ \rank(M^*)=r. 
\end{align}
In the special case when $f(\cdot;\cdot)$ is the weighted $\ell_2$-loss function in \eqref{eqn:gmc}, the above condition implies that $C_{ij}\geq 0$ for all $i,j\in[n]$. Similar to the normalization constraint $C\in\Sn_{+,1}$, we assume that objective function $f(\cdot;M^*)$ is normalized in the sense that
\begin{align}\label{eqn:normalize-1} 
{\sum}_{i,j\in[n]} \left[f(M^* + E_{ij}; M^*) - f(M^*;M^*) \right] = 1.
\end{align}
For the normalization constraint $u^*\in\Snn_1$, we assume that the global truth $M^*$ satisfies
\begin{align}\label{eqn:normalization-2}
\|U^*\|_1 = 1.
\end{align}
%
The set of degenerate instances is given by
\begin{align*}
    \cD := \bigg\{ (f, M^*)& ~\bigg|~f(\cdot;\cdot) \text{ and } M^* \text{ satisfy \eqref{eqn:normalization-3}-\eqref{eqn:normalization-2}},\\
    &\exists M\neq M^*\quad \st\quad f(M;M^*) = f(M^*;M^*),~ M^*\succeq0,~\rank(M^*)=r \bigg\}.
\end{align*}
The ``entry-wise $\ell_1$-norm'' between two arbitrary functions $h^1(\cdot)$ and $h^2(\cdot)$ with the domain $\R^{n\times n}$ is defined as the restricted $\ell_\infty$-Lipschitz constant of $h^1-h^2$. Namely, we define $\|h^1-h^2\|_1$ to be
\begin{align*}
     \|h^1-h^2\|_1 := \sup_{K,L\in\R^{n\times n}} &\frac{\left| \left(h^1(K) - h^2(K) \right) - \left(h^1(L) - h^2(L) \right) \right|}{\max_{i,j\in[n]} (K_{ij} - L_{ij})^2}\\
     \st \quad &K\neq L,\quad \rank(K-L) \leq 2r.
\end{align*}
For every constant $\alpha\in[0,1]$, the distance between two instances $(f,M^*)$ and $(\tilde{f},\tilde{M}^*)$ is defined as
\begin{align*}
    &\dist_\alpha\left[ (f,M^*), (\tilde{f},\tilde{M}^*) \right] := \alpha\|f(\cdot;M^*) - \tilde{f}(\cdot;\tilde{M}^*)\|_1 + (1-\alpha)\|U^* - \tilde{U}^*\|_1,
\end{align*}
where $U^*,\tilde{U}^*\in\R^{n\times r}$ satisfy $U^*(U^*)^T=M^*$ and $\tilde{U}^*(\tilde{U}^*)^T=\tilde{M}^*$. Finally, the complexity metric is given by
\begin{align}\label{eqn:metric-general}
    \mathbb{D}_\alpha(f,M^*) := \left[\inf_{(\tilde{f},\tilde{M}^*)\in\cD} \dist_\alpha\left[ (f,M^*), (\tilde{f},\tilde{M}^*) \right]\right]^{-1}.
\end{align}
We note that the definitions of $\cD$ and $\mathbb{D}_\alpha(f,M^*)$ are consistent with those of instance \eqref{eqn:gmc}. The following theorem provides an upper bound on the complexity metric of any instance satisfying the RIP$_{2,2}$ condition.
\begin{theorem}\label{thm:rip-1}
Let $\alpha\in[0,1]$ and $\delta\in[0,1)$ be two constants. Suppose that the function $f(\cdot;M^*)$ satisfies the $\delta$-RIP$_{2r,2r}$ condition and the normalization constraint \eqref{eqn:normalize-1}, where $r$ is the rank of $M^*$. Then, it holds that
\[ \mathbb{D}_\alpha(f,M^*) \leq \frac{n^2(1+\delta)}{\alpha(1-\delta)}
\]
\end{theorem}
\begin{proof}
We fix the instance $(f,M^*)$ and assume that $(\tilde{f},\tilde{M}^*)\in\cD$. Suppose that the matrix $M \neq \tilde{M}^*$ satisfies
\[ \tilde{f}(M;\tilde{M}^*) = \tilde{f}(\tilde{M}^*;\tilde{M}^*). \]
We first consider the case when $M \neq M^*$. In this case, we can estimate that
\begin{align}\label{eqn:rip-general-1}
    &\|f(\cdot;M^*) - \tilde{f}(\cdot;\tilde{M}^*)\|_1\\
    \nonumber&\hspace{3em}\geq \frac{\left| \left[ f(M;M^*) - \tilde{f}(M;\tilde{M}^*) \right] - \left[ f(M^*;M^*) - \tilde{f}(M^*;\tilde{M}^*) \right] \right|}{\max_{i,j\in[n]} (M_{ij} - M_{ij}^*) ^2}\\
    \nonumber&\hspace{3em}= \frac{\left| \left[ f(M;M^*) - f(M^*;M^*) \right] + \left[ \tilde{f}(M^*;\tilde{M}^*) - \tilde{f}(M;\tilde{M}^*) \right] \right|}{\max_{i,j\in[n]} (M_{ij} - M_{ij}^*) ^2}\\
    \nonumber&\hspace{3em}= \frac{\left| \left[ f(M;M^*) - f(M^*;M^*) \right] + \left[ \tilde{f}(M^*;\tilde{M}^*) - \tilde{f}(\tilde{M}^*;\tilde{M}^*) \right] \right|}{\max_{i,j\in[n]} (M_{ij} - M_{ij}^*) ^2}\\
    \nonumber&\hspace{3em}\geq \frac{f(M;M^*) - f(M^*;M^*)}{\max_{i,j\in[n]} (M_{ij} - M_{ij}^*) ^2} \geq \frac{(c_1/2)\cdot \|M-M^*\|_F^2}{\max_{i,j\in[n]} (M_{ij} - M_{ij}^*) ^2} \geq \frac{c_1}{2},
\end{align}
where $c_1$ is the constant in the RIP condition of $f(\cdot;M^*)$. The second inequality is due to
\[ f(M;M^*) - f(M^*;M^*) \geq 0,\quad \tilde{f}(M^*;\tilde{M}^*) - \tilde{f}(\tilde{M}^*;\tilde{M}^*) \geq 0. \]
The second last inequality follows from the global optimality of $M^*$ and the second inequality after inequality (12) in \cite{zhang2021general}, namely,
\[ f(M;M^*) \geq f(M^*;M^*) + \frac{c_1}{2} \| M - M^*\|_F^2,\quad \forall M\in\R^{n\times n},~ \rank(M) \leq r. \]
Now, we provide a lower bound on $c_1$. Using the normalization constraint \eqref{eqn:normalize-1} and the stationarity of $M^*$, it holds that
\begin{align*}
    1 = \sum_{i,j\in[n]} \left[f(M^* + E_{ij}; M^*) - f(M^*;M^*) \right] \leq \frac{c_2}{2} \cdot \sum_{i,j\in[n]} \|E_{ij}\|_F^2 = \frac{c_2 n^2}{2},
\end{align*}
which implies that $c_2 \geq 2 n^{-2}$. Using the relation $c_2/c_1 = (1+\delta)/(1-\delta)$, we obtain that
\[ c_1 \geq \frac{2(1-\delta)}{n^2(1+\delta)}. \]
By substituting into inequality \eqref{eqn:rip-general-1}, it follows that
\[ \|f(\cdot;M^*) - \tilde{f}(\cdot;\tilde{M}^*)\|_1 \geq \frac{1-\delta}{n^2(1+\delta)}. \]
which leads to $\dist_\alpha[ (f,M^*), (\tilde{f},\tilde{M}^*) ] \geq \alpha(1-\delta)/[n^2(1+\delta)]$. Now, the desired bound on $\mathbb{D}_\alpha(f,M^*)$ follows from taking the inverse. In the case when $M = M^*$, we can replace $M$ with $\tilde{M}^*$ and the proof can be done in the same way.
\qed\end{proof}
}

We note that the upper bound on $\mathbb{D}_\alpha(C,u^*)$ is increasing in $\delta$, which is consistent with the intuition that a smaller $\delta$ will lead to a better optimization landscape. 
Moreover, in the case when $\alpha(1-\delta) = \Theta(1)$, the upper bound is on the order of $O(n^2)$, which is $O(n)$ larger than the minimum possible complexity metric in Theorem \ref{thm:max-dist}. Now, we provide a remedy to the aforementioned issue for problem \eqref{eqn:gmc}. With the knowledge about the incoherence of the global solution, we can improve the upper bound on the complexity metric.
%
\begin{theorem}\label{thm:rip-2}
Suppose that the instance $\cMC(C,u^*)$ satisfies the $\delta$-RIP$_{2,2}$ condition and $u^*$ has incoherence $\mu$. Then, it holds that
\[ \mathbb{D}_\alpha(C,u^*) \leq \max\left\{\frac{n(1+\delta)}{4\alpha(1-\delta)}, \frac{1}{2(1-\alpha)\mu} \right\} \times \min\left\{ \left(\frac{1}{\mu} - \frac{1}{n}\right)^{-1}, 3\mu \right\}. \]
\end{theorem}

\revisee{The proof of Theorem \ref{thm:rip-2} can be found in Appendix \ref{adp:rip-2}.}
From the above theorem, we can use the weight $\alpha$ to control the balance between the RIP constant $\delta$ and the incoherence $\mu$. If we choose $1-\alpha = \Theta (n^{-1})$, then the complexity can be upper-bounded by
\[ \mathbb{D}_\alpha(C,u^*) = \mu n \cdot \max\left\{ O\left(  \frac{1+\delta}{1-\delta} \right), O\left(\frac{1}{\mu}\right) \right\} = O\left( \mu n \cdot  \frac{1+\delta}{1-\delta} \right). \]
In addition, if it holds that $\mu = O(1)$ and $(1-\delta)^{-1}=O(1)$, then the complexity is upper-bounded by $O(n)$, which matches the minimum possible complexity in Theorem \ref{thm:max-dist} up to a constant. Although the complexity metric may have a large value for extreme instances (i.e., instances with a large incoherence), the complexity of regular instances achieves the optimal value up to a constant.
\revise{Furthermore, we conjecture in Section \ref{sec:theory} that the condition $\mathbb{D}_\alpha(C,u^*) = O(n\mu/\alpha)$ is sufficient to guarantee the success of local search methods. Assuming that this conjecture is true, then the condition $(1-\delta)^{-1}=O(1)$ alone is sufficient to guarantee that the optimization landscapes are benign regardless of the value of the incoherence $\mu$. This is consistent with the existing results on the RIP condition.
}
%
We conclude the discussion of instances with the RIP condition by showing that the dependence of $\delta$ in Theorem \ref{thm:rip-2} is tight up to a constant.
\begin{theorem}\label{thm:rip-3}
Suppose that $n\geq 4$, $\alpha\in[0,1]$, $\mu\in[1,n]$ and $\delta\in[0,1)$. Let $\ell := \lceil n/\mu \rceil$. 
Then, there exists an instance $\cMC(C,u^*)$ such that $\cMC(C,u^*)$ satisfies the $\delta$-RIP$_{2,2}$ condition, $u^*$ has incoherence $\mu$ and
\[ \mathbb{D}_\alpha(C,u^*) \geq \frac{n(1+\delta)}{4\alpha(1-\delta)} \cdot \min\left\{ \frac{n\mu}{\mu \ell - \mu}, \mu \right\}. \]
\end{theorem}

\revisee{The proof of Theorem \ref{thm:rip-3} can be found in Appendix \ref{adp:rip-3}.}

\subsection{Matrix completion problem: Bernoulli model and incoherence condition}
\label{sec:mc}

Next, we consider instances $\cMC(C,u^*)$ of problem \eqref{eqn:gmc} where the global solution $u^*$ is $\mu$-incoherent and the random weight matrix $C$ obeys the Bernoulli model. Similar to the RIP condition, we need to generalize the definition of the Bernoulli model under the normalization constraint.
\begin{definition}
Given the sampling rate $p\in(0,1]$, a random matrix $C\in\Sn_{+,1}$ is said to obey the \textbf{Bernoulli model} if
\[ C_{ij} = \frac{\delta_{ij}}{\sum_{k,\ell\in[n]} \delta_{k\ell}},\quad \forall i,j\in[n], \]
where $\{\delta_{k\ell}|k,\ell\in[n]\}$ are independent Bernoulli random variables with the parameter $p$. 
\end{definition}
We note that the above model is well defined only when $\sum_{i,j}\delta_{ij} > 0$, which happens with probability $1-(1-p)^{n^2} \geq 1-\exp(-n^2p)$. This probability is sufficiently large if $n^2 p\gg 1$. In \cite{candes2010power}, the authors showed that $p\geq \Theta( \mu \log{n} / n )$ is necessary and under this condition, the success probability is at least $1 - O(n^{-\mu n})$. Therefore, we only focus on the case when the event $\sum_{i,j}\delta_{ij} > 0$ happens.
In the existing literature \cite{candes2009exact,ge2016matrix,chen2020nonconvex}, the instances obeying the Bernoulli model are proven to have no spurious local minima. We show that our complexity metric is able to characterize this property by proving an upper bound on the complexity metric.
%
\begin{theorem}\label{thm:incoh}
Given $\mu\in[1,n]$ and $p\in(0,1]$, suppose that the weight matrix $C$ obeys the Bernoulli model with the parameter $p$ and that $u^*$ has incoherence $\mu$. If $\eta>2$ is a constant and the sampling rate satisfies
\[ p \geq  \min\left\{ 1,  \frac{16(1+\eta\mu)\log{n} + 16}{n} \right\}, \]
it holds with probability at least $1 - 3n^{-\eta/2 + 1}$ that
\[ \revise{\mathbb{D}_\alpha(C,u^*) \leq \max \left\{ \frac{3n}{4\alpha}, \frac{1}{2(1-\alpha)\mu} \right\}  \times \min\left\{ \left(\frac{1}{\mu} - \frac{1}{n}\right)^{-1}, 3\mu \right\}}. \]
\end{theorem}

\revisee{The proof of Theorem \ref{thm:incoh} can be found in Appendix \ref{apd:incoh}.}
By Theorem \ref{thm:incoh}, if $1-\alpha = \Theta(n^{-1}\mu^{-1})$, then the complexity of instances obeying the Bernoulli model is on the order of $\Theta[ n^2\mu /(n - \mu) ]$. If the incoherence $\mu = O(1)$, the complexity is on the order of $O(n)$, which matches the minimum possible complexity up to a constant. Therefore, the proposed metric can also serve as a good indicator for the matrix completion problem with the Bernoulli model. Finally, we note that the bound $p\geq \Theta( \mu \log{n} / n )$ is optimal up to a constant \cite{candes2010power}; see also the discussions in Appendix E of \cite{fattahi2020exact}.

Finally, we note that problem \eqref{eqn:gmc} may still have spurious local minima when the sampling probability $p$ and the incoherence $\mu$ satisfy the condition in Theorem \ref{thm:incoh}. In the existing literature, the global convergence of randomly initialized local search methods is established for problem \eqref{eqn:gmc} only under an extra regularizer or an extra constraint on the incoherence of $u$. That being said, our proposed complexity metric correctly reflects the commonsense that the matrix completion problem is generally easier to solve when the incoherence is small or when the sampling rate $p$ is large. When the complexity is small, it is possible to apply local search methods to find the ground truth. The local search methods may be different for different classes of low-rank matrix optimization problems. In addition, the new complexity metric has the advantage that it is able to simultaneously capture the RIP condition, the incoherence condition and potentially other existing complexity metrics.

\subsection{One-parameter class of instances}
\label{sec:aistats}

In Sections \ref{sec:rip} and \ref{sec:mc}, we provided several upper bounds on the complexity metric. In this part, we consider a class of instances that are parameterized by a single parameter $\epsilon\in[0,1]$. Intuitively, when the parameter grows from $0$ to $1$, the optimization landscape of the instance becomes more benign. Unlike the previous results in this section, the analysis of the small parameter case provides necessary conditions for the existence of spurious local minima. More specifically, we fix $\mathbb{G} = (\mathbb{V}, \mathbb{E})$ to be an unweighted undirected graph without self-loops, where the node set is $\mathbb{V} = [n]$. We consider the maximal independent set of $\mathbb{G}$, which is defined as follows:
\begin{definition}
For an undirected graph $\mathbb{G} = (\mathbb{V}, \mathbb{E})$, a set $\cS\subset\mathbb{V}$ is called an \textit{independent set} if no two nodes in $\cS$ are adjacent. The set $\cS$ is called a \textit{maximal independent set} if it is an independent set with the maximum number of nodes \footnote{We note that this definition is different from the common definition of maximum independent set, which only requires that a maximum independent set is not a proper subset of an independent set.}. 
\end{definition}
Suppose that $\cS\subset [n]$ is a maximal independent set of $\mathbb{G}$. For every $\epsilon\in[0,1]$, the instance $\cMC(C^\epsilon,u^*)$ is defined by
\begin{align}\label{eqn:syn}
    C^\epsilon_{ij} &:= {\epsilon}/{Z_\epsilon},\quad \forall i,j\in\cS \quad \st\quad i\neq j;\quad C^{\epsilon}_{ij} := {1}/{Z_\epsilon},\quad \text{if }\{i,j\}\in\mathbb{E};\\
    \nonumber C^{\epsilon}_{ii} &:= {1}/{Z_\epsilon},\quad \forall i\in[n],\quad C^{\epsilon}_{ij} := 0,\quad \text{otherwise},\\
    \nonumber u_i^* &:= {1}/{m},\quad \forall i\in\cS;\quad u_{i}^* := 0,\quad \forall i\notin\cS,
\end{align}
where $m := |\cS|$ and $Z_\epsilon := 2|\mathbb{E}| + n + m(m-1)\epsilon$ is the normalization constant. In the remainder of this subsection, we assume without loss of generality that $\cS = [m]$.

First, we study for what values of $\epsilon$ the instance $\cMC(C^\epsilon, u^*)$ has benign landscape or has spurious local minima. The following theorem guarantees that the threshold $\epsilon = \Theta(m^{-1})=\Theta(\mu/n)$ separates the regimes where the instance possesses and does not possess spurious local minima, where $\mu:=n/m$ denotes the incoherence of $u^*$.
\begin{theorem}\label{thm:upper-lower}
If $\epsilon \geq \Theta(m^{-1})$, the instance $\cMC(C^\epsilon, u^*)$ does not have \textit{spurious second-order critical points}\footnote{A point $u\in\Rn$ is called a spurious second-order critical point if it satisfies the first-order and the second-order necessary optimality conditions and $uu^T \neq u^*(u^*)^T$.} (SSCPs), namely, all second-order critical points are global minima associated with the ground truth solution $M^*$. If $\epsilon = O(m^{-1})$, the instance $\cMC(C^\epsilon, u^*)$ has at least $O(2^{m/2})$ spurious local minima.
\end{theorem}

\revisee{The proof of Theorem \ref{thm:upper-lower} can be found in Appendix \ref{apd:upper-lower}.} In the case when $m=2$, the proof of Theorem \ref{thm:upper-lower} (more specifically, Theorem \ref{thm:even}) states that the instance $\cMC(C^\epsilon,u^*)$ has spurious local minima if $\epsilon < 1/3$. The condition $\epsilon=1/3$ corresponds to the $\delta$-RIP$_{2,2}$ condition holding with $\delta = 1/2$. Therefore, the RIP constant $\delta \leq 1/2$ is necessary for the instance $\cMC(C^\epsilon, u^*)$ to have no spurious local minima. Combined with the results in \cite{zhang2021general,bi2021local}, we can see that the one-parameter group $\cMC(C^\epsilon, u^*)$ also contains difficult instances of the general problem \eqref{eqn:obj-bm}.

Furthermore, we note that the constants in \revise{the proof of Theorem \ref{thm:upper-lower}} are not optimal.
We conjecture that the instance $\cMC(C^\epsilon, u^*)$ has spurious solutions if $\epsilon < (m+1)^{-1} + o(m^{-1})$ and does not have spurious solutions if $\epsilon > (m+1)^{-1} + o(m^{-1})$.
We numerically verify this conjecture in the special case when $m=n$. In numerical examples, we consider the scaled parameter $\eta := (n+1)\epsilon$. For each instance, we implement the randomly initialized gradient descent algorithm for $200$ times and check the number of implements for which the distance between the last iterate and $\pm u^*$ has Frobenius norm at most $10^{-5}$. The results are plotted in Figure \ref{fig:2}.
%
\begin{figure}[t]
\centering
\includegraphics[scale=0.43]{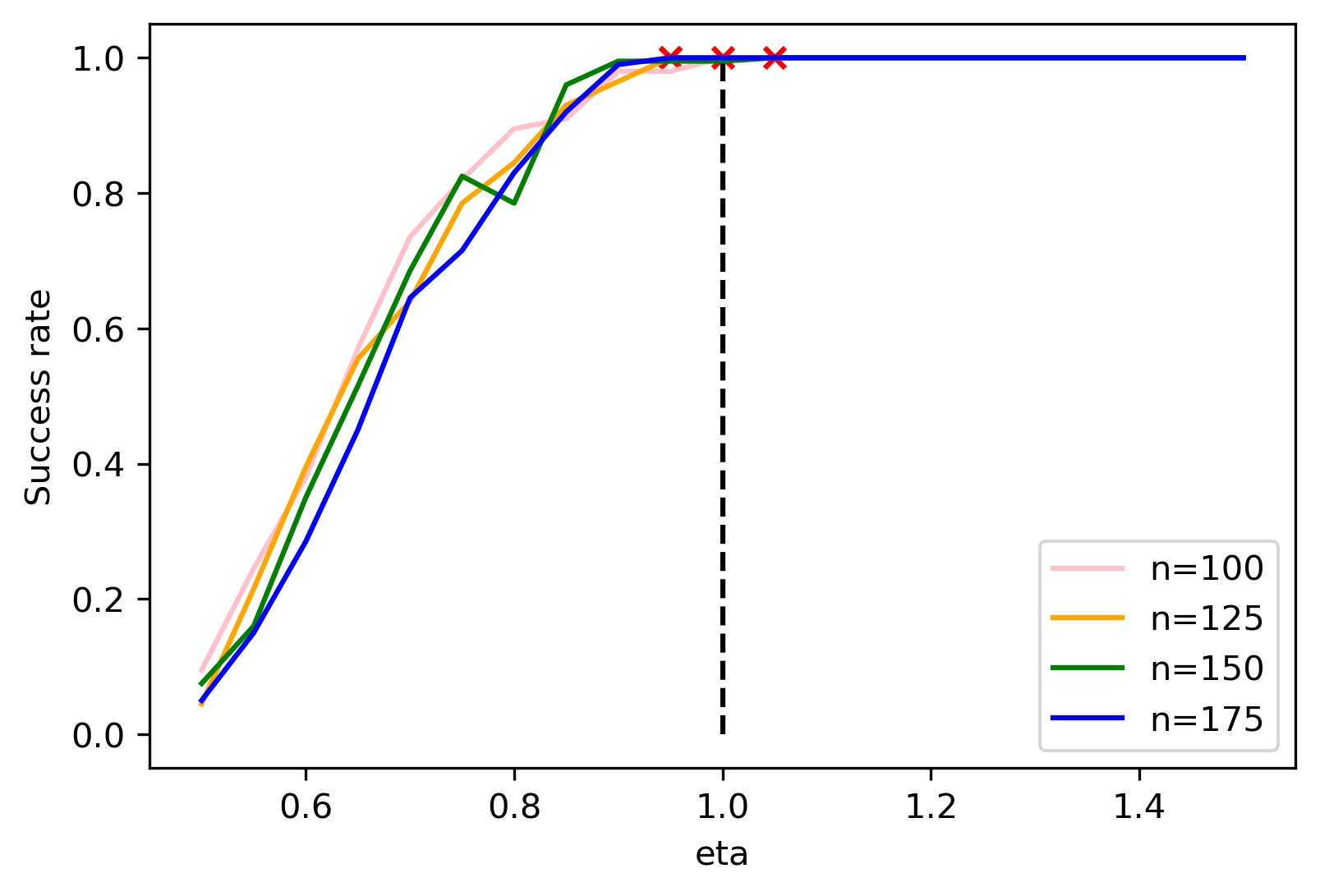}
\includegraphics[scale=0.43]{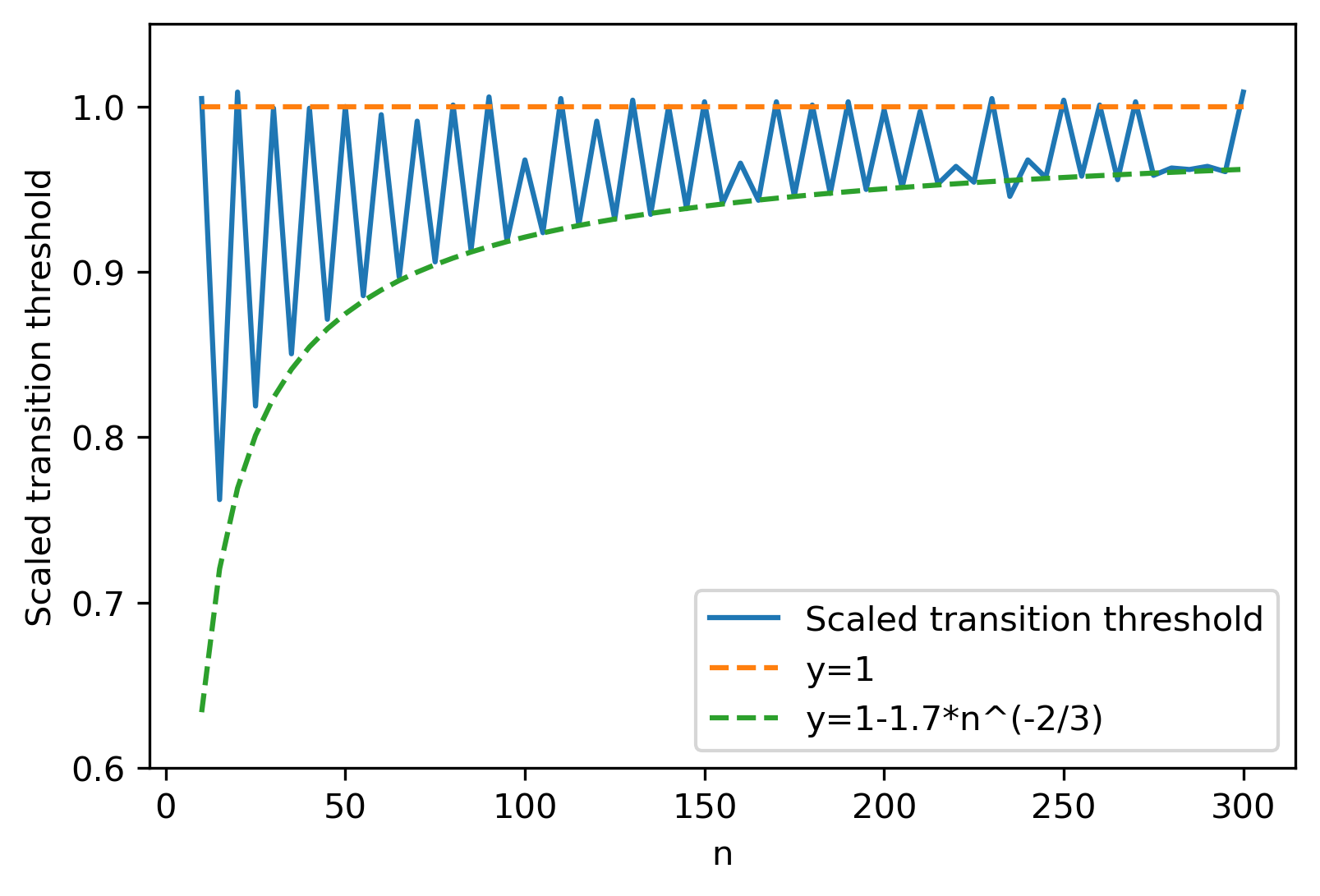}
\caption{The left plot shows the transitions of the success rate of the gradient descent algorithm when $n=100,125,150,175$. The red ``$\times$'' sign refers to the \textit{transition threshold}, i.e., the smallest value of $\eta$ that attains $100\%$ success rate. In the right plot, the transition thresholds of $\eta$ are compared with the curves $y=1$ and $y=1-1.7(n+1)^{-2/3}$.}
\label{fig:2}
\end{figure}
In the left plot, we can see that in most cases, the success rate grows with the parameter $\eta$, which is proportional to $\epsilon$. This indicates that the optimization landscape becomes more benign when $\epsilon$ is larger. In addition, the transition thresholds of $\eta$ are very close to $1$ (to be more accurate, the thresholds of $\eta$ are between $0.95$ and $1.05$). This observation is consistent with our conjecture. In the right plot, we compare the transition thresholds of $\eta$ against the constant number $1$. We observe that the thresholds are approximately located between $1$ and $1-1.7(n+1)^{-2/3}$, which implies that the original thresholds of $\epsilon$ are between $(n+1)^{-1}$ and $(n+1)^{-1} - 1.7(n+1)^{-5/3}$. Hence, the thresholds become close to $(n+1)^{-1}$ when $n$ is large, which is also consistent with our conjecture. Moreover, we can see that the threshold of $\eta$ is not monotone in $n$ and is slightly smaller when $n$ is odd.

Finally, we transform the estimates on the parameter $\epsilon$ to the complexity metric.
\begin{theorem}\label{thm:aistats-1}
Suppose that $n\geq m \geq 36$, $\alpha\in[0,1]$ and $\epsilon\in[0,1]$. Then, the following statements hold true:
\begin{enumerate}
    \item If
    \[ \mathbb{D}_\alpha(C^\epsilon, u^*) \leq \left[\frac{36\alpha}{n^2} + \min\left\{ 72\alpha \cdot \frac{ m}{n^2}, 2(1-\alpha) \right\}\right]^{-1} , \]
    then the instance $\cMC(C^\epsilon,u^*)$ has no spurious local minima;
    \item If
    \[ \mathbb{D}_\alpha(C^\epsilon, u^*) \geq \frac{18}{17} \max\left\{ \frac{13 n^2}{2\alpha}, \frac{1}{2(1-\alpha)} \right\}, \]
    then the instance $\cMC(C^\epsilon,u^*)$ has spurious local minima.
\end{enumerate}
\end{theorem}

\revisee{The proof of Theorem \ref{thm:aistats-1} can be found in Appendix \ref{apd:aistats-1}.} In the case when $1-\alpha \geq \Theta(m/n^2)$, the upper bound on $\mathbb{D}_\alpha(C^\epsilon,u^*)$ is on the order of $O(n\mu/\alpha)$, where $\mu := n/m$ is the incoherence of $u^*$. This result is consistent with the upper bounds in Sections \ref{sec:rip} and \ref{sec:mc}. In addition, the RIP constant is $1-O(1/m)$ if $\epsilon=O(1/m)$, which shows that the proposed complexity metric can provide better guarantees on the optimization complexities than the RIP constant. On the other hand, the lower bound in Theorem \ref{thm:aistats-1} is on the order of $O(n^2/\alpha)$ in the case when $1-\alpha\geq\Theta(n^{-2})$. 

\revise{
In summary, we have provided a consistent upper bound on the complexity metric that is on the order of $\Theta(n\mu/\alpha)$ for all three examples ($\Theta[n\mu/\alpha \cdot (1+\delta)/(1-\delta)]$ for the RIP case) if we choose $1-\alpha=O(n^{-1})$. These theoretical results provide strong evidence that our proposed complexity metric is able to capture the properties of the optimization landscape for several different models, even when other existing conditions fail to provide theoretical guarantees; see the comparison of the condition and our complexity metric in Section \ref{sec:aistats}. In Section \ref{sec:theory}, we make some conjectures based on these observations and provide a partial theoretical explanation.}

\section{Theoretical results for general instances}
\label{sec:theory}

In this section, we provide a theoretical analysis for the proposed complexity metric \eqref{eqn:metric-new} on the general problem \eqref{eqn:gmc}. Intuitively, we expect the problem \eqref{eqn:gmc} to have a benign landscape when the complexity metric is small and vice versa. We first prove that the proposed complexity metric is able to provide a sufficient condition on the absence of SSCPs of problem \eqref{eqn:gmc}. Then, we construct another complexity metric that lower-bounds the metric \eqref{eqn:metric} and show that the alternative complexity metric is able to provide necessary conditions on the absence of SSCPs.

Recalling the analysis in Section \ref{sec:exm}, one might have the following questions:
%
Suppose that $1-\alpha \geq \Theta(n^{-1})$ and the solution $u^*$ is $\mu$-incoherent. Can we find two constants $\delta,\Delta > 0$ such that
\begin{enumerate}
    \item If $\mathbb{D}_\alpha(C,u^*) \leq \delta \mu n / \alpha$, the instance $\cMC(C,u^*)$ has no SSCPs;
    \item If $\mathbb{D}_\alpha(C,u^*) \geq \Delta n^2 / \alpha$, the instance $\cMC(C,u^*)$ has SSCPs?
\end{enumerate}

Suppose that the first property in the above question holds. The results in Section \ref{sec:rip} imply that the proposed complexity metric guarantees the absence of SSCPs when the RIP constant is $O[(\delta-1)/(\delta+1)]$, which is independent of $\mu$. In addition, the matrix completion problem under the Bernoulli model does not have SSCPs when $p\geq O( \mu \log{n}/n )$, which matches the lower bound in \cite{candes2010power}. In Section \ref{sec:small}, we prove a weaker version of the first property in the case when $\alpha$ is equal to $\alpha^*$ or $\alpha^\diamond$, which are defined in Section \ref{sec:metric}. We note that both $\alpha^*$ and $\alpha^\diamond$ satisfy the condition that $1 - \alpha = \Theta(n^{-1})$. On the other hand, in Section \ref{sec:large}, we refute the second property in the above question by constructing counterexamples. This observation implies that similar to the RIP constant and the incoherence, the proposed complexity metric cannot provide necessary conditions on the absence of spurious local solutions. However, if we substitute the degenerate set $\cD$ with a slightly smaller set, we prove that the complexity metric is able to provide a necessary condition.

\subsection{Small complexity case}
\label{sec:small}

We first consider instances with a small complexity metric. In the case when $\alpha$ is equal to $\alpha^*$ or $\alpha^\diamond$, we prove that $\mathbb{D}_\alpha(C,u^*)\leq \delta n / \alpha$ serves as a sufficient condition for the absence of SSCPs, where $\delta>0$ is an absolute constant. Since the incoherence $\mu$ is at least $1$, the aforementioned condition is weaker than the first property in the aforementioned question. By Theorem \ref{thm:max-dist}, the minimum possible value of the complexity metric is on the order of $O(n/\alpha)$. In this subsection, we show that the constant $\delta$ can be chosen such that $\delta n / \alpha$ is strictly larger than the minimum possible complexity.
The following theorem deals with the case when $\alpha = \alpha^*$.
\begin{theorem}\label{thm:small-1}
Suppose that $n\geq 5$ and $\alpha = \alpha^*$. Then, there exists a constant $\delta > 1/4$ such that for every instance $\cMC(C,u^*)$ satisfying
\[ \mathbb{D}_{\alpha^*}(C,u^*) \leq {\delta n}/{\alpha^*}, \]
the instance $\cMC(C,u^*)$ does not have any SSCPs.
\end{theorem}

Since the minimum possible complexity metric is $n/(4\alpha^*)$, the upper bound in Theorem \ref{thm:small-1} is \textit{non-trivial} in the sense that there exist instances satisfying the inequality.
By Theorem \ref{thm:optimal}, the minimum complexity metric $n/(4\alpha^*)$ is only attained by instances in $\cM$, where
\begin{align*} 
\cM := \Bigg\{ (C,u^*) ~\Bigg|~ |u^*_i| &= \frac1n,~ C_{ii}=0,~\forall i\in[n],~
C_{ij} =\frac{1}{n(n-1)},~ \forall i,j\in[n],~i\neq j \Bigg\}. 
\end{align*}
In the next lemma, we prove the strict-saddle property \cite{sun2018geometric} of the $\ell_1$-norm for instances in $\cM$, which can be viewed as a robust version of the \revise{absence} of SSCPs.
\begin{lemma}\label{lem:ssp}
Suppose that $n\geq2$ and $(C^0,u^0)\in\cM$. Then, there exist a positive constant $\eta_0$ and two positive-valued functions $\beta(\eta)$ and $\gamma(\eta)$ such that for all $\eta \in(0,\eta_0]$ and $u\in\mathbb{R}^n$, at least one of the following properties holds:
\begin{enumerate}
    \item $\min\{\|u - u^*\|_1,\|u + u^*\|_1\} \leq \eta$;
    \item $\|\nabla g(u;C,u^*)\|_\infty \geq \beta(\eta)$;
    \item $\lambda_{min}[\nabla^2 g(u;C,u^*)]\leq -\gamma(\eta)$.
\end{enumerate}
\end{lemma}

We then show that after a sufficiently small perturbation to any point $(C^0,x^0)\in\cM$, the new instance does not have any SSCPs. 
\begin{lemma}\label{lem:pert}
Suppose that $n\geq 3$. There exists a small positive constant $\epsilon$ such that for every pair $(C^0,u^0)\in\cM$ and $(\tilde{C}, \tilde{u}^*)$ satisfying
\[ \alpha^*\|\tilde{C} - C^0\|_1 + (1-\alpha^*) \|\tilde{u}^* - u^0\|_1 < \epsilon, \]
the instance $\cMC(\tilde{C}, \tilde{u}^*)$ does not have SSCPs. 
\end{lemma}

\revisee{The proofs of the last two lemmas involve several standard calculations and can be found in Appendices \ref{apd:ssp} and \ref{apd:pert}.} Now, we prove the existence of a non-trivial upper bound on the metric.
\begin{proof}[Proof of Theorem \ref{thm:small-1}]
Let $\epsilon$ be the constant in Lemma \ref{lem:pert}.
We consider the compact set
\begin{align*} 
\mathcal{C} := \bigg\{ (C,u^*) ~\bigg|~ &\|C\|_1 = \|u^*\|_1 = 1,\\
& \alpha^*\|\tilde{C} - C^0\|_1 + (1-\alpha^*) \|\tilde{u}^* - u^0\|_1 \geq \epsilon,\quad\forall (C^0,u^0)\in\cM \bigg\}. 
\end{align*}
Since the minimum possible complexity metric $n/(4\alpha^*)$ is only attained by points in $\cM$, it holds that
\[ \mathbb{D}_{\alpha^*}(\mathcal{C}) := \max_{(C,u^*)\in\mathcal{C}} \mathbb{D}_{\alpha^*}(C,u^*) > n/(4\alpha^*). \]
Therefore, choosing
\[ \delta := (\alpha^*/n) \cdot \mathbb{D}_{\alpha^*}(\mathcal{C}) > 1/4, \]
we have
\begin{align*} 
\mathbb{D}_{\alpha^*}(C,u^*) \leq \delta n / \alpha^* &\implies (C,u^*) \notin \mathcal{C} \implies \text{the instance }\cMC(C,u^*)\text{ has no SSCPs}. 
\end{align*}
This completes the proof.
\qed\end{proof}

The case when $\alpha = \alpha^\diamond$ can be analyzed in a similar way. We note that the strict-saddle property of the instances in Theorem \ref{thm:optimal-2} has been established in \cite{jin2017escape}. Hence, we present the results in the following theorem and omit the proof.
\begin{theorem}\label{thm:small-2}
Suppose that $n\geq 5$ and $\alpha = \alpha^\diamond$. Then, there exists a constant $\delta > 1/4$ such that for every pair $(C,u^*)$ satisfying
\[ \mathbb{D}_{\alpha^\diamond}(C,u^*) \leq {\delta n(n+2)}/{(n+1)}, \]
the instance $\cMC(C,u^*)$ does not have any SSCPs.
\end{theorem}
%
Similar to Theorem \ref{thm:small-1}, since the minimum possible complexity metric is attained with $\delta=1/4$, the upper bound in Theorem \ref{thm:small-2} is non-trivial.

\subsection{Large complexity case}
\label{sec:large}

In this subsection, we first refute the second property in the question that we asked in the beginning of Section \ref{sec:theory} and then refine its statement to make it hold true. We note that the RIP condition and the incoherence condition cannot provide necessary conditions for the \revise{absence} of SSCPs either. Namely, there exist instances that satisfy the $\delta$-RIP$_{2,2}$ condition with $\delta$ as high as $1$ which do not have SSCPs. Similarly, in the case when the incoherence of the global solution is $n$, it is still possible to have an instance of the matrix completion problem without any SSCPs. In other words, although small values for the RIP constant and incoherence guarantee the \revise{absence} of spurious solutions, these notions cannot capture the complexity of the problem since there are low-complexity problems with large values for these parameters. We first show that our new metric suffers from the same shortcoming, but we then propose a simple refinement to address this issue. 
\begin{example}
Suppose that the weight matrix and the ground truth are
\[ C^\delta := \frac{1}{1+3\delta}\begin{bmatrix} 1 & \delta\\ \delta & \delta \end{bmatrix},\quad u^* := \begin{bmatrix} 1 \\ 0\end{bmatrix}, \]
where $\delta \geq 0$ is a constant. One can verify that $\pm u^*$ are the only local minima to the instance $\cMC(C^\delta,u^*)$ for all $\delta > 0$. However, in the case when $\delta = 0$, the instance $\cMC(C^0,u^*)$ has the set of global solutions
\[ \pm \begin{bmatrix} 1 \\ c\end{bmatrix},\quad \forall c \in\mathbb{R}. \]
Moreover, we consider the case when both components of $u^*$ are measured, where the instance $\cMC(\tilde{C}^\epsilon,\tilde{u}^\epsilon)$ is defined by
\[ \tilde{C}^\epsilon := \frac{1}{1+\epsilon}\begin{bmatrix} 1 & 0\\ 0 & \epsilon \end{bmatrix},\quad \tilde{u}^\epsilon := \frac{1}{1+\epsilon}\begin{bmatrix} 1 \\ \epsilon\end{bmatrix}, \]
where $\epsilon$ is a positive constant. One can verify that  the pair $(\tilde{C}^\epsilon,\tilde{u}^\epsilon)$ belongs to $\mathcal{D}$ for all $\epsilon > 0$. Setting $\delta$ and $\epsilon$ to be small enough, the instances $\cMC(C^\delta,u^*)$ and $\cMC(\tilde{C}^\epsilon,\tilde{u}^\epsilon)$ can be arbitrarily close to each other in the sense that
\[ \alpha \|C^\delta - \tilde{C}^\epsilon\|_1 + (1-\alpha)\|u^* - \tilde{u}^\epsilon\|_1 = O(\alpha\delta + \epsilon). \]
Therefore, the complexity metric of $\cMC(C^\delta,u^*)$ can be arbitrarily large. This example shows that instances without SSCPs can be arbitrarily close to those in $\mathcal{D}$, which have non-unique global solutions.
\end{example}

Nevertheless, we derive a lower bound on the complexity metric \eqref{eqn:metric-new} by constructing a subset of $\cD$, which allows obtaining a necessary condition. Intuitively, if an instance has multiple global minima, these global minima are still locally optimal after a sufficiently small perturbation to the instance. To ensure the ``robustness'' of the local optimality, we require the positive-definiteness of the Hessian matrix.
For each instance $\cMC(C,u^*)$, let $\mathbb{G}_{1k}(C,u^*)$ for all $k\in[n_1]$ be the connected components of $\mathbb{G}_{1}(C,u^*)$, where $n_1$ is the number of connected components. Moreover, we use $\cI_{1k}(C,u^*)$ to denote the node set of $\mathbb{G}_{1k}(C,u^*)$ for all $k\in[n_1]$. We define the following subset of $\cD$:
\begin{align*}
    \cSD := \{ (C,u^*)\in\cD ~|~ &\mathbb{G}_{1k}(C,u^*)\text{ is \revise{not bipartite} for all }k\in[n_1],\\
                                 &\mathbb{G}_1(C,u^*)\text{ is disconnected},\quad \mathcal{I}_{00}(C,u^*) = \emptyset\}.
\end{align*}
The following theorem provides a characterization of the Hessian matrix at global solutions for pairs in $\cSD$.
\begin{theorem}\label{thm:large-1}
Suppose that $(C,u^*)\in\cSD$. Then, the Hessian matrix is positive definite at all global solutions of the instance $\cMC(C,u^*)$.
\end{theorem}

\revisee{The proof of Theorem \ref{thm:large-1} can be found in Appendix \ref{adp:large-1}.}
\revisee{Using the positive-definiteness of the Hessian matrix, we are able to apply the implicit function theorem to guarantee the existence of spurious local minima in a neighbourhood of each instance in $\cSD$; see Appendix \ref{apd:implicit} for more details.} The global guarantee can be established by considering closed subsets of $\cSD$.
For every constant $\epsilon\geq 0$, we consider the closed subset $\mathcal{SD}_\epsilon$, which is defined as
\begin{align*} 
    \mathcal{SD}_\epsilon := \big\{ (C,u^*) \in\mathcal{SD} ~|~  &C_{ij} \in \{0\} \cup [\epsilon,1],\quad \forall i,j\in[n],\\
    &\quad |u_i^*| \in \{0\} \cup [\epsilon,1],\quad \forall i\in[n] \big\}. 
\end{align*}
Basically, the extra condition in the definition of $\mathcal{SD}_\epsilon$ requires that the nonzero components of $C$ and $u^*$ be at least $\epsilon$.
We can verify that the set $\mathcal{SD}_\epsilon$ is a compact set and for every $\epsilon_n \rightarrow 0$, it holds that
\begin{align*}
{\lim}_{n \rightarrow \infty} \cup_{i=1}^n \mathcal{SD}_{\epsilon_i} = \mathcal{SD}_0 = \cSD. 
\end{align*}
%
Now, we define the alternative complexity metric
\begin{align}\label{eqn:metric-large}
    \mathbb{D}_{\alpha,\epsilon}(C,u^*) := \left[ \min_{(\tilde{C},\tilde{u}^*)\in\cSD_\epsilon} \alpha \|C - \tilde{C}\|_1 + (1-\alpha)\|u^* - \tilde{u}^*\|_1 \right]^{-1}.
\end{align}
Since $\cSD_\epsilon$ is a subset of $\cD$, it holds that 
\[ \mathbb{D}_{\alpha,\epsilon}(C,u^*) \leq \mathbb{D}_{\alpha}(C,u^*). \]
Similar to Theorem \ref{thm:unique-closure}, we can prove the following relation:
\begin{align*} 
\overline{\cSD} &= \{ (C,u^*) ~|~C\in\Sn_{+,1},u^*\in\Snn_1, \mathbb{G}_1(C,u^*) \text{ is disconnected} \}\\
&\hspace{5em} \cup \{ (C,u^*) ~|~ C\in\Sn_{+,1},u^*\in\Snn_1, \cI_{00}(C,u^*) \text{ is not empty} \}. 
\end{align*}
Hence, the closure of $\cSD$ is a proper subset of $\overline{\cD}$. Combining with the fact that $\cSD_\epsilon$ is a subset of $\cSD$, the metric $\mathbb{D}_{\alpha,\epsilon}(C,u^*)$ is not equivalent to $\mathbb{D}_{\alpha}(C,u^*)$. 
Using the compactness of $\mathcal{SD}_\epsilon$, the following theorem provides a necessary condition for the existence of spurious local minima.
\begin{theorem}\label{thm:global}
Suppose that $\epsilon > 0 $ is a constant. Then, there exists a large constant $\Delta(\epsilon)>0$ such that for every instance $\cMC(C,u^*)$ satisfying
\[ \mathbb{D}_{\alpha,\epsilon}(C,u^*) \geq \Delta(\epsilon), \]
the instance $\cMC(C,u^*)$ has spurious local minima.
\end{theorem}
\begin{proof}
For every pair $(C,u^*) \in \mathcal{SD}_{\epsilon}$, Lemma \ref{lem:local} implies that there exists an open neighborhood of $(C,u^*)$ such that the desired properties hold. Now, we consider the union of such open neighborhoods over all points $(C,u^*)\in\mathcal{SD}_\epsilon$, which is an open cover of $\mathcal{SD}_\epsilon$. Using the Heine-Borel covering theorem, there exists an open sub-cover of $\mathcal{SD}_\epsilon$. Therefore, we obtain the existence of $\Delta(\epsilon)$.
\qed\end{proof}

We note that the maximum possible value of $\mathbb{D}_{\alpha,\epsilon}(C,u^*)$ is $+\infty$, which is attained by instances in $\cSD_\epsilon$. Therefore, there exist instances satisfying the condition of Theorem \ref{thm:global} and the lower bound is \textit{non-trivial}. Using Theorem \ref{thm:global}, the slightly modified complexity metric is able to provide a necessary condition on the \revise{absence} of SSCPs. This result implies that our complexity metric is able to provide conditions that are much better than the RIP condition and the incoherence condition that fail to provide necessary conditions.

Finally, we conjecture that the second property in the question we asked in the beginning of the section holds for any fixed weight matrix. More specifically, we define
\begin{align}\label{eqn:metric-u}
    \mathbb{D}_C(u^*) := \left( \min_{(C,\tilde{u}^*)\in\overline{\cD}} \|u^* - \tilde{u}^*\|_1 \right)^{-1}.
\end{align}
We have the following conjecture:
\begin{conjecture}\label{conj:2}
Suppose that $\epsilon\in[0,1]$. Then, there exists a large constant $\Gamma(\epsilon)>0$ such that for every instance $\cMC(C,u^*)$ satisfying
\[ C_{ij}\in\{0\}\cup[\epsilon,1],\quad \mathbb{D}_C(u^*) \geq \Gamma(\epsilon), \]
the instance $\cMC(C,u^*)$ has spurious local minima.
\end{conjecture}
We note that the metric $\mathbb{D}_C(u^*)$ is equal to $0$ if $\cMC(C,u^*)$ satisfies the $\delta$-RIP$_{2,2}$ condition with $\delta\in[0,1)$. 

\section{Conclusions}
\label{sec:cls}

In this work, we propose a new complexity metric for an important class of the low-rank matrix optimization problems, which has the potential to \revisee{generalize major existing recovery guarantees} and is applicable to a much broader set of problems. The proposed complexity metric aims to measure the complexity of the non-convex optimization landscape of each problem and quantifies the likelihood of local search methods in successfully solving each instance of the problem under a random initialization. We focus on the rank-$1$ generalized matrix completion problem \eqref{eqn:gmc} to mathematically prove the usefulness of the new metric from three aspects. Namely, we show that the complexity metric has a small value if the instance satisfies the RIP condition or the incoherence condition. The results in these two scenarios are consistent with the existing results on the RIP condition and the incoherence condition. In addition, we analyze a one-parameter class of instances to illustrate that the proposed metric captures the true complexity of this class as the parameter varies and has consistent behavior with the aforementioned two scenarios. \revise{This consistency implies that our proposed complexity metric is able to characterize the optimization landscapes of different applications, which the RIP condition and the incoherence condition fail to capture.} Finally, we provide strong theoretical results on the generalized matrix completion problem by showing that a small value for the proposed complexity metric guarantees the \revise{absence} of spurious solutions, whereas a large value for a slightly modified complexity metric guarantees the existence of spurious solutions. This also shows the superiority of this metric over the RIP condition and the incoherence condition since those notions cannot offer any necessary conditions on having spurious solutions.

\bibliographystyle{spmpsci}      
\bibliography{ref}   


\begin{appendix}
\normalsize

\section{Analysis of the degenerate case}

In this section, we provide a detailed analysis on instances with $u^* = 0$. The optimization problem of the instance $\cMC(C,0)$ can be written as
\begin{align}\label{eqn:obj-zero}
    \min_{u\in\Rn} {\sum}_{i,j\in[n]} C_{ij} u_i^2 u_j^2.
\end{align}
We prove that problem \eqref{eqn:obj-zero} either has multiple global solutions or has no SSCPs. 
\begin{theorem}
If $C_{ii} > 0$ for all $i\in[n]$, the instance $\cMC(C,0)$ has no SSCPs. Otherwise if $C_{ii}=0$ for some $i\in[n]$, the instance $\cMC(C,0)$ has nonzero global solutions.
\end{theorem}
\begin{proof}
We first consider the case when $C_{ii} > 0$ for all $i\in[n]$. Let $u^0\in\Rn$ be a second-order critical point. By the first-order optimality conditions, it holds that
\begin{align*}
    \frac14 \nabla_i g(u^0;C,0) = C_{ii}(u_i^0)^3 + \sum_{j\in[n], j\neq i} C_{ij} u_i^0(u_j^0)^2  = 0,\quad \forall i\in[n].
\end{align*}
Multiplying $u_i^0$ on both sides, we have
\[ 0 = C_{ii}(u_i^0)^4 + \sum_{j\in[n], j\neq i} C_{ij} (u_i^0)^2(u_j^0)^2 \geq C_{ii}(u_i^0)^4 \geq 0, \]
which implies that $C_{ii}(u_i^0)^4 = 0$. Since $C_{ii} > 0$, it follows that
\[ u_i^0 = 0,\quad \forall i\in[n]. \]
Hence, $u^0=0$ is the unique second-order critical point.

Next, we consider the case when there exists an index $i_0$ such that $C_{i_0i_0} = 0$. In this case, define $u^0\in\Rn$ by
\[ u^0_{i_0} = 1,\quad u^0_i = 0,\quad \forall i\in[n]\backslash\{i_0\}. \]
Then, we have
\[ \left[u^0(u^0)^T\right]_{i_0i_0} = 1,\quad \left[u^0(u^0)^T\right]_{ij} = 0,\quad \text{otherwise}. \]
Since the $(i_0,i_0)$ entry is not observed, the point $u^0$ leads to the same measurements as $u^*=0$. Therefore, $u^0$ is a nonzero global solution to the instance $\cMC(C,0)$.

\qed\end{proof}

\section{Proofs in Section \ref{sec:metric}}

\subsection{Proof of Theorem \ref{thm:unique-closure}}
\label{adp:unique-closure}

\begin{proof}
%
We denote the set on the right-hand side as $\cD'$.
We first prove that
\begin{align}\label{eqn:unique-1} 
\overline{\cD} &\supset \cD'.  
\end{align}
Suppose that $(C, u^*)\in\cD'$. If $\mathbb{G}_1(C,u^*)$ is disconnected or bipartite, the instance $\cMC(C,u^*)$ already belongs to $\cD$ and, therefore, belongs to the closure $\overline{\cD}$. We only need to consider the case when $\cI_{00}(C,u^*)$ is not empty.
For every constant $\epsilon > 0$, we construct a new global solution $\tilde{u}^*$ as follows:
\[ \tilde{u}^*_{i} := \begin{cases}  u^*_{i} + \epsilon & \text{if }i\in\cI_{00}(C,u^*)\\ u^*_{i} & \text{otherwise}. \end{cases} \]
Let $\tilde{M}^* := \tilde{u}^*(\tilde{u}^*)^T$. For the instance $\cMC(C,\tilde{u}^*)$, we have
\[ \cI_1(C,\tilde{u}^*) = \cI_1(C,u^*) \cup \cI_{00}(C,u^*). \]
By the definition of $\cI_{00}(C,u^*)$, the nodes in $\cI_1(C,u^*)$ and $\cI_{00}(C,u^*)$ are disconnected. Therefore, the new subgraph $\mathbb{G}_1(C,\tilde{u}^*)$ is disconnected and the new instance $\cMC(C,\tilde{u}^*)$ belongs to $\cD$. By letting $\epsilon\rightarrow0$, it follows that $(C,u^*)$ is a limit point of $\cD$ and belongs to $\overline{\cD}$. This completes the proof of the relation \eqref{eqn:unique-1}.

Then, we prove the other direction $\overline{\cD}\subset\cD'$. By Theorem \ref{thm:unique}, we have $\cD\subset \cD'$.
Hence, it remains to prove that the set $\cD'$ is closed. Equivalently, we prove that $(\cD')^c$ is open, where $(\cD')^c$ is the complementary set with respect to $\R^{n\times n} \times \R^{n}$. Suppose that $(C, u^*)\in(\cD')^c$. If $\|C\|_1 \neq 1$ or $\|u^*\|_1 \neq 1$, changing $C$ and $u^*$ by a small perturbation will not make $\|C\|_1 = \|u^*\|_1 = 1$. Now, we only consider the case when $\|C\|_1 = \|u^*\|_1 = 1$. Since $(C,u^*)\in(\cD')^c$, the subgraph $\mathbb{G}_1(C,u^*)$ is connected and \revise{not bipartite} and the set $\cI_{00}(C,u^*) = \emptyset$. Denote
\[ \epsilon := \min\left\{\min_{C_{ij} > 0}  C_{ij}, \min_{u_i^* \neq 0} |u_{i}^*| \right\} > 0. \]
Suppose that we add a sufficiently small perturbation to the point $(C,u^*)$ such that each component of $C$ and $u^*$ is changed by at most $\epsilon/2$. Then, all nonzero components of $C$ and $u^*$ are still nonzero after the perturbation. Therefore, the edges of the subgraph $\mathbb{G}_1(C,M^*)$ are not deleted after the perturbation and, thus, the subgraph is still connected and \revise{not bipartite}. Similarly, after perturbation, each node in $\cI_0(C,M^*)$ either becomes nonzero or is connected to $\mathbb{G}_1(C,M^*)$, which implies that $\cI_{00}(C,M^*)$ is still an empty set. Therefore, the perturbed instance still belongs to $(\cD')^c$. Hence, the set $(\cD')^c$ is open and we obtain the relation $\overline{\cD} \subset \cD'$. 
\qed\end{proof}

\subsection{Proof of Theorem \ref{thm:max-dist}}
\label{adp:max-dist}

The proof of Theorem \ref{thm:max-dist} relies on the following two lemmas, which transform the computation of $\mathbb{D}_\alpha^{min}$ \revise{into} a one-dimensional optimization problem. The first lemma upper-bounds the maximum possible distance.
\begin{lemma}\label{lem:max_dist-1}
Suppose that $n\geq 2$. It holds that
\[ \left(\mathbb{D}_\alpha^{min}\right)^{-1} \leq \max_{c\in\left[0,\frac{1}{n(n-1)}\right]} g(\alpha,c), \]
where the function $g(\alpha,c)$ is defined by
\begin{align*} 
    g(\alpha,c) := \min\bigg\{ &2(1-\alpha)\cdot \frac{n-2}{n} + 4\alpha c, \quad 4\alpha(n-1)c, \\
    &2(1-\alpha)\cdot \frac{n-4}{n} + 2\alpha\left(\frac4n - 4(n-2)c\right), \\
    &2(1-\alpha)\cdot \frac{n-3}{n} + 2\alpha\left(\frac3n - (3n-5)c\right), \\
    &2(1-\alpha)\cdot \frac{n-2}{n} + 2\alpha\left(\frac2n - 2(n-1)c\right), \\
    &2(1-\alpha)\cdot \frac{n-1}{n} + 2\alpha\left(\frac1n - (n-1)c\right) \bigg\}.
\end{align*}
\end{lemma}
\begin{proof}
Denote the distance between $(C,u^*)$ and $\cD$ as
\[ \mathbb{T}_\alpha(C,u^*) := \min_{(\tilde{C},\tilde{u}^*)\in\overline{\cD}} \alpha \|C - \tilde{C}\|_1 + (1-\alpha)\|u^* - \tilde{u}^*\|_1. \]
%
%
We fix the pair $(C,u^*)$ and let
\[ \eta := \frac{1}{n(n-1)} \sum_{i,j\in[n], i\neq j} C_{ij} \in \left[0,\frac{1}{n(n-1)}\right]. \]
Using the condition $\|C\|_1 = 1$, it follows that
\[ \theta := \frac{1}{n} \sum_{i\in[n]} C_{ii} = \frac1n \left( 1 - \sum_{i,j\in[n], i\neq j} C_{ij} \right) = \frac1n - (n-1)\eta \in[0,n^{-1}]. \]
Our goal is to prove that
\begin{align*}
    \mathbb{T}_\alpha(C,u^*) \leq g(\alpha,\eta).
\end{align*}
%
%
In the remainder of the proof, we upper-bound the distance $\mathbb{T}_\alpha(C,u^*)$ by constructing some instances in $\overline{\cD}$. 

We first consider those instances in $\overline{\cD}$ with a disconnected subgraph $\mathbb{G}_1$. For every $k\in\{2,\dots,n\}$, let $\cI_1$ be a subset of $[n]$ satisfying $|\cI_1|=k$ and $\cI_0 := [n]\backslash\cI_1$. Suppose that $\epsilon>0$ is a sufficiently small constant. For every $i_0\in\cI_1$, we consider the pair $(\tilde{C},\tilde{u}^*)$, where
\begin{align}\label{eqn:upper-7}
    \tilde{u}^*_i = 0,\quad \forall i\in\cI_0;\quad \tilde{u}^*_i = (1-\epsilon)u^*_i + \epsilon \cdot \frac{\left\| u^*_{\cI_1} \right\|_1}{|\cI_1|} + \frac{\left\| u^*_{\cI_0} \right\|_1}{|\cI_1|}  ,\quad \forall i\in\cI_1
\end{align}
and
\begin{align*}
    \tilde{C}_{i_0j} &= \tilde{C}_{ji_0} = 0,\quad \forall j\in\cI_1 \backslash\{i_0\};\\
    \tilde{C}_{ij} &= C_{ij} + \frac{2}{n^2 - 2(k-1)}{\sum}_{j\in\cI_1 \backslash\{i_0\}} C_{i_0j},\quad \text{otherwise}.
\end{align*}
By choosing a sufficiently small $\epsilon$, it can be shown that
\[ \cI_1(\tilde{C},\tilde{u}^*) = \cI_1;\quad \cI_0(\tilde{C},\tilde{u}^*) = \cI_0. \]
The node $i_0$ is disconnected from other nodes in $\mathbb{G}_1(\tilde{C},\tilde{u}^*)$ and, therefore, $(\tilde{C},\tilde{u}^*)\in\overline{\cD}$. 
%
%
%
The distance between $u^*$ and $\tilde{u}^*$ is
\begin{align}\label{eqn:upper-1}
    \|u^* - \tilde{u}^*\|_1 &= 2 \left\| u^*_{\cI_0} \right\|_1 + 2\epsilon\left\| u^*_{\cI_1} \right\|_1 \leq 2 \left\| u^*_{\cI_0} \right\|_1 + 2\epsilon.
\end{align}
In addition, the distance between $C$ and $\tilde{C}$ can be calculated as
\begin{align}\label{eqn:upper-2}
    \|C - \tilde{C}\|_1 &= 4 {\sum}_{j\in\cI_1 \backslash\{i_0\}} C_{i_0j}.
\end{align}
Combining inequalities \eqref{eqn:upper-1} and \eqref{eqn:upper-2}, we have
\begin{align}\label{eqn:upper-3}
    \mathbb{T}_\alpha(C,u^*) \leq 2(1-\alpha)\left\| u^*_{\cI_0} \right\|_1 + 4\alpha {\sum}_{j\in\cI_1 \backslash\{i_0\}} C_{i_0j} + 2\epsilon.
\end{align}
%
Taking the average of inequality \eqref{eqn:upper-3} over $i_0$ for $\cI_1$, we have
\begin{align}\label{eqn:upper-5}
    \mathbb{T}_\alpha(C,u^*) \leq 2(1-\alpha)\left\| u^*_{\cI_0} \right\|_1 + 4\alpha (k-1) {\sum}_{i,j\in\cI_1, i\neq j} C_{ij} + 2\epsilon.
\end{align}
Then, we take the average of \eqref{eqn:upper-5} over $\cI_1$ for all $k$-element subsets of $[n]$, which leads to
\begin{align*}
    \mathbb{T}_\alpha(C,u^*) \leq 2(1-\alpha) \cdot \frac{n-k}{n} + 4\alpha (k-1) \eta + 2\epsilon.
\end{align*}
By setting $\epsilon \rightarrow 0$, we obtain that 
\begin{align}\label{eqn:upper-6}
    \mathbb{T}_\alpha(C,u^*) \leq 2(1-\alpha) \cdot \frac{n-k}{n} + 4\alpha (k-1) \eta.
\end{align}
Since inequality \eqref{eqn:upper-6} is linear in $k$, the minimum of the right-hand side over $k\in\{2,\dots,n\}$ is attained by either $2$ or $n$. Hence, it holds that
\begin{align}\label{eqn:upper-12}
    \mathbb{T}_\alpha(C,u^*) \leq \min\left\{ 2(1-\alpha) \cdot \frac{n-2}{n} + 4\alpha \eta, 4\alpha(n-1)\eta \right\}.
\end{align}
Using a similar analysis, we can obtain inequality \eqref{eqn:upper-6} by considering instances in $\overline{\cD}$ whose $\cI_{00}$ is non-empty.

Finally, we check those instances in $\overline{\cD}$ whose $\mathbb{G}_{1}$ is bipartite. Let $\cI_1$ be a subset of $[n]$ satisfying $|\cI_1| = 4$, and let $\cI_0 = [n]\backslash\cI_1$. We define $\tilde{u}^*$ in the same way as \eqref{eqn:upper-7}. For every subset $\cI_{11}\subset \cI_1$ such that $|\cI_{11}|=2$, the new weight matrix is defined as
\begin{align*}
    \tilde{C}_{ii} &= 0,\quad \forall i\in\cI_1;\quad \tilde{C}_{ij} = 0,\quad \forall i,j\in\cI_{11};\quad \tilde{C}_{ij} = 0,\quad \forall i,j\in\cI_1 \backslash\cI_{11};\\
    \tilde{C}_{ij} &= C_{ij} + \frac{2}{n^2 - 8}\left(\sum_{i\in\cI_1} C_{ii} + \sum_{i,j\in\cI_{11},i\neq j} C_{ij} + \sum_{i,j\in\cI_1\backslash\cI_{11},i\neq j} C_{ij} \right).
\end{align*}
The distance between $C$ and $\tilde{C}$ is
\begin{align*}
    \|C - \tilde{C}\|_1 = 2\left(\sum_{i\in\cI_1} C_{ii} + \sum_{i,j\in\cI_{11},i\neq j} C_{ij} + \sum_{i,j\in\cI_1\backslash\cI_{11},i\neq j} C_{ij} \right)
\end{align*}
Therefore, the maximum distance is bounded by
\begin{align}\label{eqn:upper-8}
    \mathbb{T}_\alpha(C,u^*) &\leq 2(1-\alpha)\left\| u^*_{\cI_0} \right\|_1\\
    \nonumber &\quad + 2\alpha\left(\sum_{i\in\cI_1} C_{ii} + \sum_{i,j\in\cI_{11},i\neq j} C_{ij} + \sum_{i,j\in\cI_1\backslash\cI_{11},i\neq j} C_{ij} \right) + 2\epsilon.
\end{align}
By taking the average of \eqref{eqn:upper-8} over $\cI_{11}$ for all $2$-element subsets of $\cI_1$, it follows that
\begin{align}\label{eqn:upper-9}
    \mathbb{T}_\alpha(C,u^*) &\leq 2(1-\alpha)\left\| u^*_{\cI_0} \right\|_1 + 2\alpha\left(\sum_{i\in\cI_1} C_{ii} + \frac{1}{3}\sum_{i,j\in\cI_{1},i\neq j} C_{ij} \right) + 2\epsilon.
\end{align}
Furthermore, we take the average of \eqref{eqn:upper-9} over $\cI_{1}$ for all $4$-element subsets of $[n]$, which gives
\begin{align*}
    \mathbb{T}_\alpha(C,u^*) &\leq 2(1-\alpha)\cdot \frac{k}{n} + 2\alpha\left( 4\theta + 4\eta \right) + 2\epsilon.
\end{align*}
By letting $\epsilon\rightarrow0$, we conclude that
\begin{align}\label{eqn:upper-10}
    \mathbb{T}_\alpha(C,u^*) &\leq 2(1-\alpha)\cdot \frac{4}{n} + 2\alpha\left( 4\theta + 4\eta \right).
\end{align}
By applying a similar technique to subsets of $[n]$ with $1,2,3$ elements, the distance can be bounded as
\begin{align}\label{eqn:upper-11}
    \mathbb{T}_\alpha(C,u^*) &\leq 2(1-\alpha)\cdot \frac{3}{n} + 2\alpha\left( 3\theta + 2\eta \right),\\
    \nonumber\mathbb{T}_\alpha(C,u^*) &\leq 2(1-\alpha)\cdot \frac{2}{n} + 2\alpha\cdot 2\theta,\\
    \nonumber\mathbb{T}_\alpha(C,u^*) &\leq 2(1-\alpha)\cdot \frac{1}{n} + 2\alpha\cdot  \theta.
\end{align}
By combining inequalities \eqref{eqn:upper-6}, \eqref{eqn:upper-10} and \eqref{eqn:upper-11} and recalling the relation that $\theta=1/n - (n-1)\eta$, it follows that
\begin{align*}
    \mathbb{T}_\alpha(C,u^*) &\leq g(\alpha,\eta).
\end{align*}
Now, we take the maximum over $C\in\Sn_{+,1}$ and $u^*\in\Snn_1$, which is equivalent to taking the maximum over $\eta\in\left[0,\frac{1}{n(n-1)}\right]$ in the right-hand side. This yields that 
\begin{align*}
    \max_{\|C\|_1=\|u^*\|_1=1} \mathbb{T}_\alpha(C,u^*) &\leq \max_{c\in\left[0,\frac{1}{n(n-1)}\right]} g(\alpha,c).
\end{align*}
This completes the proof.
\qed\end{proof}

We denote $g_i(\alpha,c)$ be the $i$-th term in the above minimization for all $i\in\{1,\dots,6\}$.
The next lemma proves the other direction.
\begin{lemma}\label{lem:max_dist-2}
Suppose that $n\geq 2$. It holds that
\[ \left(\mathbb{D}_\alpha^{min}\right)^{-1} \geq \max_{c\in\left[0,\frac{1}{n(n-1)}\right]} g(\alpha,c), \]
where the function $g(\alpha,c)$ is defined in Lemma \ref{lem:max_dist-1}.
\end{lemma}
\begin{proof}
Let $\eta\in\left[0,\frac{1}{n(n-1)}\right]$ and define the pair $(C,u^*)$ according to
\begin{align*}
    u^*_i := \frac1n,~ C_{ii} := \frac{1}{n} - (n-1)\eta,\quad \forall i\in[n];\quad C_{ij} := \eta,\quad \forall i,j\in[n]\quad \st\quad i\neq j.
\end{align*}
Our goal is to prove that
\[ \mathbb{T}_\alpha(C,u^*) \geq g(\alpha,\eta). \]
Suppose that $(\tilde{C},\tilde{u}^*)\in\overline{\cD}$ attains the distance $\mathbb{T}_\alpha(C,u^*)$, namely,
\begin{align*}
    \mathbb{T}_\alpha(C,u^*) = \alpha\|C - \tilde{C}\|_1 + (1-\alpha)\|u^*-\tilde{u}^*\|_1.
\end{align*}
We analyze three different cases.

\paragraph{Case I.} We first consider the case when $\mathbb{G}_1(\tilde{C},\tilde{u}^*)$ is disconnected. Denote $k:=|\cI_1(\tilde{C},\tilde{u}^*)|$. The distance between $u^*$ and $\tilde{u}^*$ is lower-bounded by
\begin{align}\label{eqn:lower-1}
    \|u^* - \tilde{u}^*\|_1 \geq 2\| u^*_{\cI_0(\tilde{C},\tilde{u}^*)} - \tilde{u}^*_{\cI_0(\tilde{C},\tilde{u}^*)} \|_1 = 2\| u^*_{\cI_0(\tilde{C},\tilde{u}^*)}\|_1 = \frac{2(n-k)}{n}.
\end{align}
Since there are $k$ nodes in $\mathbb{G}_1(\tilde{C},\tilde{u}^*)$, we need to eliminate at least $k-1$ edges that are not self-loops to make the graph disconnected. Therefore, at least $2(k-1)$ non-diagonal weights of $\tilde{C}$ are $0$ and the distance between $C$ and $\tilde{C}$ is at least
\begin{align}\label{eqn:lower-2}
    \|C - \tilde{C}\|_1 \geq 2 \cdot 2(k-1)\eta = 4(k-1)\eta.
\end{align}
Combining inequalities \eqref{eqn:lower-1} and \eqref{eqn:lower-2}, we obtain that
\begin{align}\label{eqn:lower-3}
    \mathbb{T}_\alpha(C,u^*) \geq 2(1-\alpha) \cdot \frac{n-k}{n} + 4\alpha(k-1)\eta.
\end{align}

\paragraph{Case II.} For the case when $\cI_{00}(\tilde{C},\tilde{u}^*)$ is not empty, similar estimations as \textit{Case I} can be derived and inequality \eqref{eqn:lower-3} also holds true.

\paragraph{Case III.} Finally, we consider the case when $\mathbb{G}_1(\tilde{C},\tilde{u}^*)$ is bipartite. Denote $k:=|\cI_1(\tilde{C},\tilde{u}^*)|$. If $k\geq 5$, we need to eliminate at least $k-1$ edges that are not self-loops to make the graph bipartite. Thus, we can follow the same proof as \textit{Case I} to arrive at inequality \eqref{eqn:lower-3}. If $k=4$, we need to eliminate at least $2$ edges that are not self-loops and $4$ self-loops to make the graph bipartite. Therefore, at least $4$ non-diagonal weights and $4$ diagonal weights of $\tilde{C}$ are $0$, and the distance between $C$ and $\tilde{C}$ is at least
\begin{align}\label{eqn:lower-4}
    \|C - \tilde{C}\|_1 \geq 2 \left[4\eta + 4\left( \frac1n - (n-1) \eta \right) \right] = 2 \left[ \frac4n - (4n-8)\eta \right].
\end{align}
Combining inequalities \eqref{eqn:lower-1} and \eqref{eqn:lower-4} yields that
\begin{align}\label{eqn:lower-5}
    \mathbb{T}_\alpha(C,u^*) \geq 2(1-\alpha) \cdot \frac{n-4}{n} + 2\alpha \left[ \frac4n - (4n-8)\eta \right].
\end{align}
The cases when $k=1,2,3$ can be analyzed similarly, leading to
\begin{align}\label{eqn:lower-6}
    \mathbb{T}_\alpha(C,u^*) &\geq 2(1-\alpha) \cdot \frac{n-3}{n} + 2\alpha \left[ \frac3n - (3n-5)\eta \right],\\
    \nonumber\mathbb{T}_\alpha(C,u^*) &\geq 2(1-\alpha) \cdot \frac{n-2}{n} + 2\alpha \left[ \frac2n - (2n-2)\eta \right],\\
    \nonumber\mathbb{T}_\alpha(C,u^*) &\geq 2(1-\alpha) \cdot \frac{n-1}{n} + 2\alpha \left[ \frac1n - (n-1)\eta \right].
\end{align}
By combining \textit{Cases I-III}, it follows that
\begin{align*}
    \mathbb{T}_\alpha(C,u^*) \geq g(\alpha,\eta).
\end{align*}
Choosing $\eta$ to be the maximizer
\[ \eta^* := \argmax_{c\in\left[0,\frac{1}{n(n-1)}\right]} g(\alpha,c), \]
we have
\begin{align*}
    \mathbb{T}_\alpha(C,u^*) \geq \max_{c\in\left[0,\frac{1}{n(n-1)}\right]} g(\alpha,c).
\end{align*}
Taking the maximum over $C\in\Sn_{+,1}$ and $u^*\in\Snn_1$ gives rise to the desired conclusion.
\qed\end{proof}

\begin{proof}[Proof of Theorem \ref{thm:max-dist}]
By the results of Lemmas \ref{lem:max_dist-1} and \ref{lem:max_dist-2}, we only need to compute $\max_{c\in\left[0,\frac{1}{n(n-1)}\right]} g(\alpha,c)$. Let $\kappa := (1-\alpha)/\alpha\in[0,+\infty]$. We study three cases below.

\paragraph{Case I.} We first consider the case when $\kappa\geq 2(n-3) / [(n-4)(n-1)]$. We prove that $g(\alpha,c) = g_2(\alpha,c)$. Since $g_2(\alpha,c)$ has a larger gradient than $g_1(\alpha,c)$ and the function $g_i(\alpha,c)$ is decreasing in $c$ for $i=3,4,5,6$, we only need to show that
\begin{align}\label{eqn:max_dist-1-1}
g_i\left(\alpha, \frac{1}{n(n-1)}\right) \geq g_2\left(\alpha, \frac{1}{n(n-1)}\right),\quad \forall i\in\{1,3,4,5,6\}. 
\end{align}
The above inequality with $i=1$ is equivalent to $\kappa \geq {2}/(n-1)$, which is guaranteed by the assumption that $\kappa\geq 2(n-3) / [(n-4)(n-1)]$. 
For $i\in\{3,4,5,6\}$, the inequality \eqref{eqn:max_dist-1-1} is equivalent to
\begin{align*}
    \kappa \geq \max\left\{ \frac{2(n-3)}{(n-1)(n-4)}, \frac{2(n-2)}{(n-1)(n-3)}, \frac{2}{n-2}, \frac{2}{n-1} \right\} = \frac{2(n-3)}{(n-1)(n-4)}.
\end{align*}
Therefore, it holds that
\[ g(\alpha,c) = g_2(\alpha,c) = 4\alpha(n-1)c. \]
whose maximum is attained at $c=[n(n-1)]^{-1}$ and
\[ \max_{C,u^*}\mathbb{T}_\alpha(C,u^*) = g_2\left(\alpha, \frac{1}{n(n-1)}\right) = \frac{4\alpha}{n}. \]

\paragraph{Case II.} Then, we consider the case when $\kappa \leq 2/n$. In this case, we prove that the maximum is achieved by the intersection point between $g_1(\alpha,c)$ (an increasing function in $c$) and $\min\{g_5(\alpha,c),g_6(\alpha,c)\}$ (a decreasing function in $c$). The intersection points between $g_1(\alpha,C)$ and the other five functions are
\begin{align*}
    \frac{\kappa}{2n},\quad \frac{2-\kappa}{n(2n-3)},\quad \frac{3-\kappa}{3n(n-1)},\quad \frac{1}{n^2},\quad \frac{1+\kappa}{n(n+1)}.
\end{align*}
In the regime $\kappa \leq 1/n$, we have
\begin{align*}
    \frac{\kappa}{2n} &\leq \frac{1+\kappa}{n(n+1)} \leq \frac{1}{n^2} \leq \min\left\{ \frac{2-\kappa}{n(2n-3)}, \frac{3-\kappa}{3n(n-1)} \right\},
\end{align*}
which implies that the maximum is attained at $c=(1+\kappa)/[n(n+1)]$. Hence, the maximum distance is
\[ \max_{C,u^*}\mathbb{T}_\alpha(C,u^*) = g_1\left(\alpha, \frac{1+\kappa}{n(n+1)}\right) = \frac{2(1-\alpha)(n-2)(n+1) + 4}{n(n+1)}. \]
In the regime $1/n \leq \kappa \leq 2/n$, we have
\begin{align*}
    \frac{\kappa}{2n} & \leq \frac{1}{n^2} \leq \frac{1+\kappa}{n(n+1)} \leq \min\left\{ \frac{2-\kappa}{n(2n-3)}, \frac{3-\kappa}{3n(n-1)} \right\},
\end{align*}
which implies that the maximum is attained at $c=1/n^2$. Hence, the maximum distance is
\[ \max_{C,u^*}\mathbb{T}_\alpha(C,u^*) = g_1\left(\alpha, \frac{1}{n^2}\right) = \frac{2(1-\alpha)(n-2)n + 4\alpha}{n^2}. \]

\paragraph{Case III.} We finally consider the case when $2/n \leq \kappa \leq 2(n-3) / [(n-4)(n-1)]$. In this regime, the intersection point between $g_2(\alpha,c)$ and $g_5(\alpha,c)$ is 
\[ \frac{\kappa(n-2)+2}{4n(n-1)} \leq \frac{\kappa}{2n}. \]
This implies that $g_2(\alpha,c)$ intersects with $g_5(\alpha,c)$ before $g_1(\alpha,c)$. Therefore, the maximum is attained at one of the intersects between $g_2(\alpha,c)$ and $g_i(\alpha,c)$ for $i=3,4,5,6$. By calculating the four intersects, the optimal $c$ that achieves the maximum is given by
\[ c^*(\kappa) := \min\left\{ \frac{\kappa(n-4)+4}{n(6n-10)}, \frac{\kappa(n-3)+3}{n(5n-7)}, \frac{\kappa(n-2)+2}{4n(n-1)},\frac{\kappa(n-1)+1}{3n(n-1)} \right\}, \]
which is an increasing function in $\kappa$. If $\kappa = 2/n$, we can estimate that
\begin{align}\label{eqn:max-dist-4} 
&c^*(\kappa)\\
\nonumber= &\min\left\{ \frac{2(n-4)/n+4}{n(6n-10)}, \frac{2(n-3)/n+3}{n(5n-7)}, \frac{2(n-2)/n+2}{4n(n-1)},\frac{2(n-1)/n+1}{3n(n-1)} \right\} \\
\nonumber= &\min\left\{ \frac{3n-4}{n^2(3n-5)}, \frac{5n-6}{n^2(5n-7)}, \frac{1}{n^2},\frac{3n-2}{n^2(3n-3)} \right\} = \frac{1}{n^2}.
\end{align}
Similarly, if $\kappa = 2(n-3) / [(n-4)(n-1)]$, it holds that
\begin{align}\label{eqn:max-dist-5} 
    c^*(\kappa) = \frac{1}{n(n-1)}.
\end{align}
Combining \eqref{eqn:max-dist-4} and \eqref{eqn:max-dist-5}, we have
\[ c^*(\kappa) \in \left[ \frac{1}{n^2}, \frac{1}{n(n-1)} \right],\quad \forall \kappa \in \left[\frac2n, \frac{2(n-3)}{(n-4)(n-1)}\right]. \]
Therefore, the maximum distance satisfies the bound
\[ \max_{C,u^*}\mathbb{T}_\alpha(C,u^*) = g_2\left[\alpha, c^*(\kappa) \right] \in \left[ \frac{4\alpha(n-1)}{n^2}, \frac{4\alpha}{n} \right]. \]
This completes the proof.
\qed\end{proof}

\subsection{Proof of Theorem \ref{thm:optimal}}
\label{adp:optimal}

\begin{proof}
By the assumption that the complexity metric of $(C,u^*)$ is finite, we have that $(C,u^*)\notin \overline{\cD}$. It follows from Theorem \ref{thm:unique-closure} that the subset $\mathcal{I}_{00}(C,u^*)$ is empty and that $\mathbb{G}_1(C,u*)$ is connected and \revise{not bipartite}. Let $k:=|\cI_1(C,u^*)|$. For each node $i_0\in\cI_1(C,u^*)$, we define the new weight matrix $\tilde{C}$ as
\begin{align*}
    \tilde{C}_{i_0j} &= \tilde{C}_{ji_0} = 0,\quad \forall j\in\cI_1(C,u^*) \backslash\{i_0\};\\
    \tilde{C}_{ij} &= C_{ij} + \frac{2}{n^2 - 2(k-1)}{\sum}_{j\in\cI_1(C,u^*) \backslash\{i_0\}} C_{i_0j},\quad \text{otherwise}.
\end{align*}
The subgraph $\mathbb{G}_1(\tilde{C},u^*)$ is disconnected and, therefore, we have $(\tilde{C},u^*)\in\overline{\cD}$. It follows that
\begin{align}\label{eqn:optimal-1} 
\frac{4\alpha^*}{n} = [\mathbb{D}_{\alpha^*}(C,u^*)]^{-1} \leq \alpha^*\|C - \tilde{C}\|_1 = 4\alpha^*{\sum}_{j\in\cI_1(C,u^*) \backslash\{i_0\}} C_{i_0j}. 
\end{align}
For each node $i_0\in\cI_0(C,u^*)$, a similar construct of $\tilde{C}$ leads to
\begin{align}\label{eqn:optimal-2}
\frac{4\alpha^*}{n} = [\mathbb{D}_{\alpha^*}(C,u^*)]^{-1} \leq 4\alpha^*\sum_{j\in\cI_1(C,u^*)} C_{i_0j}. 
\end{align}
By summing inequality \eqref{eqn:optimal-1} over $i_0$ for all nodes in $\cI_1(C,u^*)$ and summing inequality \eqref{eqn:optimal-2} over $i_0$ for all nodes in $\cI_0(C,u^*)$, it follows that
\begin{align}\label{eqn:optimal-3}
    4\alpha^* &\leq 4\alpha^*\left[ \sum_{i,j\in\cI_1(C,u^*),i\neq j} C_{ij} + \sum_{i\in\cI_1(C,u^*), j\in\cI_0(C,u^*)} C_{ij} \right]\\
    \nonumber&\leq  4\alpha^*\sum_{i,j\in[n],i\neq j} C_{ij} \leq 4\alpha^*,
\end{align}
where all inequalities should hold with equality. Since the last inequality in \eqref{eqn:optimal-3} holds with equality, we obtain that 
\[ C_{ii} = 0,\quad \forall i\in [n]. \]
It follows from the equality of inequalities \eqref{eqn:optimal-1} and \eqref{eqn:optimal-2} that
\begin{align}\label{eqn:optimal-4}
    \sum_{j\in\cI_1(C,u^*)\backslash\{i\}} C_{ij} &= \frac{1}{n},~ \forall i\in\cI_1(C,u^*);\quad \sum_{j\in\cI_1(C,u^*)} C_{ij} = \frac{1}{n},~ \forall i\in\cI_0(C,u^*).
\end{align}
Using the condition that $\|C\|_1 = 1$, the above equalities imply that all weights of $C$ are limited to edges with a node in $\cI_1(C,u^*)$. Namely, we have
\begin{align}\label{eqn:optimal-5}
    {\sum}_{j\in\cI_0(C,u^*)} C_{ij} &= 0,\quad \forall i\in\cI_1(C,u^*).
\end{align}
If $\cI_o(C,u^*)$ is not empty, the above equality contradicts the second equality in \eqref{eqn:optimal-4}. Hence, the point $(C,u^*)$ satisfies that $\cI_0(C,u^*)=\emptyset$. By a similar analysis of the bipartite instance in Lemma \ref{lem:max_dist-1}, for every $4$-element subset $\{i,j,k,\ell\}$ of $[n]$, it holds that
\begin{align*}
    2(1-\alpha^*)( 1 - |u_i^*| - |u_j^*| - |u_k^*| - |u_\ell^*| ) + 4\alpha^*( C_{ij} + C_{k\ell} ) = {4\alpha^*}/{n}.
\end{align*}
Taking the average of the above equality over $\{i,j,k,\ell\}$ for all $4$-element subsets of $[n-1]$, we obtain that
\begin{align*}
    &2(1-\alpha^*)\left( 1 - \frac{3\| u^*_{1:n-1} \|_1}{n-1}  \right) + 4\alpha^* \frac{2}{(n-1)(n-2)}\| C_{1:n-1,1:n-1}\|_1 = \frac{4\alpha^*}{n}.
\end{align*}
Using the first equality in \eqref{eqn:optimal-4} and the symmetry of $C$, it holds that $\| C_{1:n-1,1:n-1}\|_1 = 1-2/n$. Substituting into the above equality, we know
\begin{align*}
    2(1-\alpha^*)\left( 1 - \frac{3\| u^*_{1:n-1} \|_1}{n-1}  \right) = 4\alpha^* \cdot  \frac{n-3}{n(n-1)}.
\end{align*}
By recalling that $\alpha^* = (n-1)(n-4)/(n^2-3n-2)$, the above inequality leads to
\begin{align*}
    \| u^*_{1:n-1} \|_1 = (n-1)/{n},
\end{align*}
which is equivalent to $|u_n^*| = 1/n$. By the same proof technique, we conclude that
\begin{align*}
    |u_i^*| = 1/n,\quad \forall i\in[n].
\end{align*}
By substituting back into equality \eqref{eqn:optimal-5}, it holds for all $4$-element subsets $\{i,j,k,\ell\}\subset[n]$ that
\begin{align*}
    C_{ij} + C_{k\ell} = \frac{2}{n(n-1)},
\end{align*}
which implies that
\[ C_{ij} = \frac{1}{n(n-1)},\quad \forall i,j\in[n]\quad \st\quad i\neq j. \]
\qed\end{proof}

\section{Proofs in Section \ref{sec:exm}}

\subsection{Proof of Theorem \ref{thm:rip-2}}
\label{adp:rip-2}

Before proving the estimation of the complexity metric, we prove two properties of $\mu$-incoherent vectors.
\begin{lemma}\label{lem:incoh}
Given any constant $\mu\in[1,n]$, suppose that $u^*$ has incoherence $\mu$ and $\|u^*\|_1 = 1$. Then, the following properties hold:
\begin{enumerate}
	\item $u^*$ has at least $n / \mu$ nonzero components;
	\item $|u_i^*| \leq \mu / n$ for all $i\in[n]$.
\end{enumerate}
\end{lemma}
\begin{proof}
Assume without loss of generality that
\[ |u_i^*| > 0,\quad \forall i\in[\ell];\quad u_i^* = 0,\quad \forall i\in\{\ell+1,\dots,n\}. \]
By the definition \eqref{eqn:incoh}, we have
\[ (u_i^*)^2 \leq \frac{\mu}{n} \|u^*\|_2^2 = \frac{\mu}{n} {\sum}_{i\in[\ell]} (u_i^*)^2,\quad \forall i\in[\ell]. \]
Summing over $i\in[\ell]$, we obtain that
\[ {\sum}_{i\in[\ell]} (u_i^*)^2 \leq \frac{\ell\mu}{n} {\sum}_{i\in[\ell]} (u_i^*)^2, \]
which implies that $\ell \geq {n}/{\mu}$. Let
\[ c_i := {|u_i^*|}/{\|u^*\|_2},\quad \forall i\in[\ell]. \]
The assumption that the incoherence is equal to $\mu$ implies that
\begin{align}\label{eqn:incoh-2} c_i \in (0,\sqrt{\mu/n}],\quad \forall i\in[\ell]. \end{align}
In addition, it holds that
\begin{align*}
\|u^*\|_2^2 &= {\sum}_{i\in[\ell]} (u_i^*)^2 = {\sum}_{i\in[\ell]} c_i^2 \|u^*\|_2^2,\\
1 = \|u^*\|_1 &= {\sum}_{i\in[\ell]} |u_i^*| = {\sum}_{i\in[\ell]} c_i \|u^*\|_2,
\end{align*}
which implies that
\[ {\sum}_{i\in[\ell]} c_i^2 = 1,\quad {\sum}_{i\in[\ell]} c_i = \|u^*\|_2^{-1}. \]
Combined with \eqref{eqn:incoh-2}, it follows that
\[ \|u^*\|_2^{-1} = {\sum}_{i\in[\ell]} c_i \geq \sqrt{\frac{n}{\mu}} \cdot {\sum}_{i\in[\ell]} c_i^2 = \sqrt{\frac{n}{\mu}}. \]
Therefore,
\[ |u_i^*| = c_i \|u^*\|_2 \leq \sqrt{{\mu}/{n}} \cdot \sqrt{{\mu}/{n}}  = {\mu}/{n}. \]
\qed\end{proof}

The following lemma lower-bounds the perturbation of the weight matrix $C$.
\begin{lemma}\label{lem:rip-2}
Suppose that the instance $\cMC(C,u^*)$ satisfies the $\delta$-RIP$_{2,2}$ condition and the weight matrix $\tilde{C}\in\Sn_{+,1}$ has $N$ zero entries, where $\delta\in[0,1)$ and $N\in[n^2]$. Then, it holds that
\[ \|C - \tilde{C}\|_1 \geq 2 {\sum}_{(i,j)\in\mathcal{N}} C_{ij} \geq \frac{2(1-\delta)N}{ (1+\delta)n^2 - 2\delta N }, \]
where $\cN$ is the set of indices of zero entries of $\tilde{C}$.
\end{lemma}
\begin{proof}
The $\delta$-RIP$_{2,2}$ condition implies that
\[ \frac{\min_{i,j}C_{ij}}{\max_{i,j}C_{ij}} \geq \frac{1-\delta}{1+\delta}. \]
Therefore, considering the average of entries in $\cN$ and that of entries not in $\cN$, we have
\[ \frac{\frac{1}{N}\sum_{(i,j)\in\mathcal{N}} C_{ij}}{\frac{1}{n^2-N}\sum_{(i,j)\notin \mathcal{N}} C_{ij}}\geq \frac{1-\delta}{1+\delta}, \]
which further leads to
\[ \sum_{(i,j)\in\mathcal{N}} C_{ij} \geq \frac{1-\delta}{1+\delta} \cdot \frac{N}{n^2-N} \sum_{(i,j)\notin \mathcal{N}} C_{ij} = \frac{1-\delta}{1+\delta} \cdot \frac{N}{n^2-N} \left(1 - \sum_{(i,j)\in \mathcal{N}} C_{ij}\right). \]
The above inequality is equivalent to
\[ {\sum}_{(i,j)\in\mathcal{N}} C_{ij} \geq \frac{(1-\delta)N}{ (1+\delta)n^2 - 2\delta N }. \]
Hence, the distance between $C$ and $\tilde{C}$ is lower-bounded as
\[ \|C - \tilde{C}\|_1 \geq 2 {\sum}_{(i,j)\in\mathcal{N}} C_{ij} \geq \frac{2(1-\delta)N}{ (1+\delta)n^2 - 2\delta N }. \]
This completes the proof.
\qed\end{proof}

Now, we prove the main theorem.
\begin{proof}[Proof of Theorem \ref{thm:rip-2}]
Suppose that $\cMC(\tilde{C},\tilde{u}^*)\in\overline{\cD}$ is the instance such that
\[ \left[\mathbb{D}_\alpha(C,u^*)\right]^{-1} = \alpha\|C - \tilde{C}\|_1 + (1-\alpha)\|u^* - \tilde{u}^*\|_1. \]
In the following, we split the proof into two steps.

\paragraph{Step I.} We first fix $\tilde{u}^*$ and consider the closest matrix $\tilde{C}$ to $C$ such that $(\tilde{C}, \tilde{u}^*)\in\overline{\cD}$. Let $k:=|\cI_1(\tilde{C}, \tilde{u}^*)|$. Without loss of generality, we assume that
\[ \cI_1(\tilde{C}, \tilde{u}^*) = \{1,\dots,k\},\quad \cI_0(\tilde{C}, \tilde{u}^*) = \{k+1,\dots,n\}. \]
We first consider the case when $k\geq 2$. If $\mathbb{G}_1(\tilde{C}, \tilde{u}^*)$ is disconnected, at least $2(k-1)$ entries of $\tilde{C}$ are $0$. If $\mathbb{G}_1(\tilde{C}, \tilde{u}^*)$ are bipartite, at least $k^2/2 \geq 2(k-1)$ entries of $\tilde{C}$ are $0$. If $\cI_{00}(\tilde{C}, \tilde{u}^*)$ is non-empty, at least $2k$ entries of $\tilde{C}$ are $0$. Otherwise if $k = 1$, at least one entry of $\tilde{C}$ should be $0$ to make $\mathbb{G}_1(\tilde{C}, \tilde{u}^*)$ bipartite. 
In summary, at least $N(k)$ entries of $\tilde{C}$ are $0$, where
\[ N(k) := \max\{ 2(k-1), 1\}. \]
%
Using the results in Lemma \ref{lem:rip-2}, the distance between $C$ and $\tilde{C}$ is at least
\begin{align}\label{eqn:rip-2-1}
    \|C - \tilde{C}\|_1 \geq \frac{2(1-\delta)N(k)}{ (1+\delta)n^2 - 2\delta N(k)}. 
\end{align}
We note that the distance is monotonously increasing as a function of $k$.

\paragraph{Step II.} Now, we consider the optimal choice of $\tilde{u}^*$ based on the lower bound in \eqref{eqn:rip-2-1}. Let
\[ \ell := |\cI_1(C,u^*)|,\quad k:=|\cI_1(\tilde{C}, \tilde{u}^*)|. \]
Since the distance between $C$ and $\tilde{C}$ is a monotonously increasing function of $k$, the minimum distance between $(C,u^*)$ and $(\tilde{C},\tilde{u}^*)$ cannot be attained by $k>\ell$. Therefore, we focus on the case when $k\leq \ell$.
Without loss of generality, we assume that
\[ |u_1^*| \geq |u_2^*| \geq \cdots \geq |u_\ell^*| > 0;\quad |u_i^*|=0,\quad \forall i\geq \ell+1. \]
Then, the distance between $u^*$ and $\tilde{u}^*$ satisfies
\begin{align}\label{eqn:rip-2-2} 
    \|u^* - \tilde{u}^*\|_1 \geq 2{\sum}_{i=k+1}^{\ell} |u_i^*|. 
\end{align}
Denote the distance between $(C,u^*)$ and $(\tilde{C},\tilde{u}^*)$ by
\[ d_\alpha := \alpha\|C - \tilde{C}\|_1 + (1-\alpha)\|u^* - \tilde{u}^*\|_1. \]
%

\paragraph{Step II-1.} We first consider the case when $\mu \leq 2n/3$.
Combining inequalities \eqref{eqn:rip-2-1} and \eqref{eqn:rip-2-2}, we obtain a lower bound on $d_\alpha$:
\begin{align*}
d_\alpha \geq \min_{k\in[\ell]} \left[\frac{2\alpha(1-\delta)N(k)}{n^2(1+\delta) - 2\delta N(k)} + 2(1-\alpha) {\sum}_{i=k+1}^{\ell} |u_i^*| \right]. 
\end{align*}
For every $k\in[\ell]$, the term inside the above minimization can be lower-bounded by
\begin{align*}
&\frac{2\alpha(1-\delta)N(k)}{n^2(1+\delta) - 2\delta N(k)} + 2(1-\alpha) {\sum}_{i=k+1}^{\ell} |u_i^*|\\
&\hspace{11em}\geq \frac{2\alpha(1-\delta)\cdot 2(k-1)}{n^2(1+\delta)} + 2(1-\alpha) {\sum}_{i=k+1}^{\ell} |u_i^*|\\
&\hspace{11em}= \frac{4\alpha(1-\delta)}{n^2(1+\delta)}\cdot (k-1) + 2(1-\alpha) {\sum}_{i=k+1}^{\ell} |u_i^*|.
\end{align*}
The minimum of the right-hand side over $k\in[\ell]$ can be solved in closed form and is equal to
\begin{align*}
    {\sum}_{i=2}^\ell \min\left\{\frac{4\alpha(1-\delta)}{n^2(1+\delta)}, 2(1-\alpha) |u_i^*| \right\}.
\end{align*}
Using the second property in Lemma \ref{lem:incoh}, we have
\begin{align*}
    \min\left\{\frac{4\alpha(1-\delta)}{n^2(1+\delta)}, 2(1-\alpha) |u_i^*| \right\} &\geq \min\left\{\frac{4\alpha(1-\delta)}{n^2(1+\delta)} \cdot \frac{n|u_i^*|}{\mu}, 2(1-\alpha) |u_i^*| \right\}\\
    &= \min\left\{\frac{4\alpha(1-\delta)}{\mu n(1+\delta)}, 2(1-\alpha) \right\} \cdot |u_i^*|.
\end{align*}
Taking the summation over $k\in\{2,\dots,\ell\}$, we can conclude that
\begin{align}\label{eqn:rip-2-4}
d_\alpha &\geq {\sum}_{k=2}^{\ell} \min\left\{\frac{4\alpha(1-\delta)}{\mu n(1+\delta)}, 2(1-\alpha) \right\} \cdot |u_i^*|\\
\nonumber&= \min\left\{\frac{4\alpha(1-\delta)}{\mu n(1+\delta)}, 2(1-\alpha) \right\} \cdot {\sum}_{k=2}^{\ell} |u_i^*|.
\end{align}
Using the second property in Lemma \ref{lem:incoh} and $\|u^*\|_1= 1$, it follows that
\begin{align*}
    {\sum}_{k=2}^{\ell} |u_i^*| \geq 1 - \frac{\mu}{n}.
\end{align*}
Substituting back into inequality \eqref{eqn:rip-2-4}, we have
\[ d_\alpha \geq \min\left\{\frac{4\alpha(1-\delta)}{\mu n(1+\delta)}, 2(1-\alpha) \right\} \cdot \left(1 - \frac{\mu}{n}\right). \]

\paragraph{Step II-2.} Next, we consider the case when $\mu \geq 2n/3$. By Theorem \ref{thm:rip-1}, the distance is at least
\begin{align*}
    d_\alpha &\geq \frac{2\alpha(1-\delta)}{n^2(1+\delta)-2\delta} \geq \frac{2\alpha(1-\delta)}{(3/2)\mu \cdot n(1+\delta)} \geq \min\left\{\frac{4\alpha(1-\delta)}{\mu n(1+\delta)}, 2(1-\alpha) \right\} \cdot \frac13,
\end{align*}
where the second inequality is due to the assumption that $\mu\geq 2n/3$.

By combining \textit{Steps II-1} and \textit{II-2}, the distance is lower-bounded by
\begin{align*}
    d_\alpha &\geq \min\left\{\frac{4\alpha(1-\delta)}{\mu n(1+\delta)}, 2(1-\alpha) \right\} \times \max\left\{ 1 - \frac{\mu}{n}, \frac13 \right\}\\
    &= \min\left\{\frac{4\alpha(1-\delta)}{n(1+\delta)}, 2(1-\alpha)\mu \right\} \times \max\left\{ \frac{1}{\mu} - \frac{1}{n}, \frac1{3\mu} \right\}
\end{align*}
The proof is completed by using the relation between $d_\alpha$ and $\mathbb{T}_\alpha(C,u^*)$.
\qed\end{proof}

\subsection{Proof of Theorem \ref{thm:rip-3}}
\label{adp:rip-3}

\begin{proof}

The proof is split into two different cases.

\paragraph{Case I.} We first consider the case when $\mu \leq n/2$. We construct the weight matrix $\tilde{C}$ as
\begin{align*}
    \tilde{C}_{1i} = \tilde{C}_{i1} = 0,\quad \forall i\in\{2,\dots,\ell\};\quad \tilde{C}_{ij} = \frac{1}{n^2 - 2(\ell-1)},\quad \text{otherwise}.
\end{align*}
For the instance $\cMC(\tilde{C},u^*)$, node $1$ is disconnected from nodes $\{2,\dots,\ell\}$ and thus, the subgraph $\mathbb{G}_1(\tilde{C},u^*)$ is disconnected. This implies that $(\tilde{C},u^*)\in\overline{\cD}$. The matrix $C$ is defined as
\begin{align*}
    {C}_{1i} &= {C}_{i1} = \frac{1-\delta}{(1+\delta)n^2 - 4\delta(\ell-1)},\quad \forall i\in\{2,\dots,\ell\};\\
    {C}_{ij} &= \frac{1+\delta}{(1+\delta)n^2 - 4\delta(\ell-1)},\quad \text{otherwise}.
\end{align*}
We can verify that the weight matrix $C$ ensures that $\cMC(C,u^*)$ satisfies the $\delta$-RIP$_{2,2}$ condition. The complexity of $\cMC(C,u^*)$ is lower-bounded by
\begin{align*}
    \mathbb{D}_\alpha(C,u^*) &\geq \left(\alpha\|C - \tilde{C}\|_1\right)^{-1} = \frac{(1+\delta)n^2 - 4\delta(\ell-1)}{4\alpha(\ell-1)(1-\delta)}\\
    &\geq  \frac{(1+\delta)(n^2-2n)}{4\alpha(\ell-1)(1-\delta)} = \frac{n(1+\delta)}{4\alpha(1-\delta)} \cdot \frac{n-2}{\ell-1} \geq \frac{n(1+\delta)}{4\alpha(1-\delta)} \cdot \frac{n\mu}{2(n\ell-1)},
\end{align*}
where the second last inequality follows from $4\delta \leq 2(1+\delta)$ and the last inequality is due to $n\geq 4$. 

\paragraph{Case II.} Next, we consider the case when $\mu \geq n / 2$. Theorem \ref{thm:rip-1} implies that there exists an instance $\cMC(C,u^*)$ such that
\begin{align*}
    \mathbb{D}_\alpha(C,u^*) &= \frac{n^2(1+\delta)-2\delta}{2\alpha(1-\delta)} \geq \frac{ (n^2-1)(1+\delta) }{2\alpha(1-\delta)} \geq \frac{ n(1+\delta) }{2\alpha(1-\delta)} \cdot \frac{n}{2},
\end{align*}
where the first inequality results from $2\delta \leq 1+\delta$ and the second inequality is in light of $n\geq 4$. Using the condition that $\mu \leq n$, it follows that
\[ \mathbb{D}_\alpha(C,u^*) \geq \frac{ n(1+\delta) }{4\alpha(1-\delta)} \cdot \mu. \]

Combining \textit{Cases I} and \textit{II} completes the proof.
\qed\end{proof}

\subsection{Proof of Theorem \ref{thm:incoh}}
\label{apd:incoh}

We first establish several lemmas before providing the proof of Theorem \ref{thm:incoh}. The first lemma is the Chernoff bound for the sum of Bernoulli random variables, \revise{which is a result of Proposition 2.14 in \cite{wainwright2019high}}.
\begin{lemma}\label{lem:bern}
Suppose that $X_1,\dots,X_m$ are i.i.d. Bernoulli random variables with the parameter $p$. 
Then, it holds that
\begin{align*} 
\mathbb{P}\left( \sum_{i\in[m]} X_i \leq \frac{mp}{2} \right) &\leq \exp\left( \frac{-mp}{8} \right),~\mathbb{P}\left( \sum_{i\in[m]} X_i \geq \frac{3mp}{2} \right) \leq \exp\left( \frac{-mp}{10} \right).
\end{align*}
\end{lemma}

The next lemma provides an upper bound on the total number of nonzero entries.
\begin{lemma}\label{lem:nonzero}
Suppose that $n\geq 3$. With probability at least $1 - \exp(-np/10)$, there are at most $3n^2p / 2$ nonzero entries in $C$. With the same probability, it holds that
\[ C_{ij} \geq \frac{2}{3n^2p},\quad \forall i,j\in[n]\quad \mathrm{s.t.}\quad C_{ij} > 0. \]
\end{lemma}
\begin{proof}
For the $n(n-1)$ non-diagonal entries of $C$, Lemma \ref{lem:bern} implies that there are at most $({3}/{2}) \cdot n(n-1)p$ nonzero entries with probability at least $1 - \exp\left( -{n(n-1)p}/{20} \right)$. For the $n$ diagonal entries of $C$, the same lemma implies that there are at most $({3}/{2})\cdot np$ nonzero entries with probability at least $1 - \exp\left( -{np}/{10} \right)$. Combining both parts concludes that there are at most $({3}/{2}) \cdot n^2p$ nonzero entries in $C$ with probability at least
\[ 1 - \exp\left( -{n(n-1)p}/{20} \right) - \exp\left( -{np}/{10} \right) \geq 1 - 2\exp\left( -{np}/{10} \right), \]
where the last inequality is due to $n\geq 3$. The lower bound on $C_{ij}$ follows from the normalization constraint.
\qed\end{proof}

For every fixed global solution $\tilde{u}^*$, the next lemma estimates the distance between $(C,\tilde{u}^*)$ and $\cD$.
\begin{lemma}\label{lem:fix-x}
Suppose that $\tilde{u}^*$ is a given vector and the random matrix $C$ obeys the Bernoulli model. In addition, suppose that $\eta > 2$ is a constant and 
\[ \|\tilde{u}^*\|_0 \geq \frac{n}{2\mu},\quad p \geq  \min\left\{ 1,  \frac{16(1+\eta\mu)\log{n} + 16}{n} \right\}, \]
where $\|\tilde{u}^*\|_0$ is the number of nonzero entries of $\tilde{u}^*$.
For every instance $(\tilde{C},\tilde{u}^*)\in\overline{\cD}$, it holds with probability at least $1 - 3 n^{-\eta/2}$ that 
\[ \|C - \tilde{C}\|_1 \geq \frac{4(\|\tilde{u}^*\|_0-1)}{3n^2}. \]
\end{lemma}
\begin{proof}
For all $i,j\in[n]$, we define Bernoulli random variables $X_{ij}$ to be $1$ if $C_{ij} > 0$ and $0$ otherwise. Then, $X_{ij}$ are independent identically distributed Bernoulli random variables with the parameter $p$. Let $N := \sum_{i,j}X_{ij}$ be the number of nonzero weights in $C$. By the definition of the Bernoulli model, all nonzero entries of $C$ are equal to $N^{-1}$. Since the global solution $\tilde{u}^*$ is fixed, we assume without loss of generality that
\[ \cI_1(C, \tilde{u}^*) = [\ell],\quad \cI_0(C, \tilde{u}^*) = \{\ell+1,\dots,n\}. \]
We fix $\tilde{C}$ to be a weight matrix such that $(\tilde{C},\tilde{u}^*)\in\overline{\cD}$ and investigate three cases.
\paragraph{Case I.} We first consider the case when $\mathbb{G}_1(\tilde{C},\tilde{u}^*)$ is disconnected. Suppose that $\tilde{\mathcal{I}}_{11}$ and $\tilde{\mathcal{I}}_{12}$ are a division of $[\ell]$ such that the nodes in $\tilde{\mathcal{I}}_{11}$ are not connected with the nodes in $\tilde{\mathcal{I}}_{12}$. In addition, we denote $k:=|\tilde{\mathcal{I}}_{11}|$ and assume that $k \leq \ell/2$. Since the nodes in $\tilde{\cI}_{11}$ are disconnected from the nodes in $\tilde{\cI}_{12}$, at least
\[ 2{\sum}_{i\in\tilde{\mathcal{I}}_{11}, j\in\tilde{\mathcal{I}}_{12}} X_{ij} \]
nonzero entries in $C$ are equal to $0$ in $\tilde{C}$. 
Therefore, we have
\begin{align*}
    \|C - \tilde{C}\|_1 \geq \frac{1}{N} \cdot 4\sum_{i\in\tilde{\mathcal{I}}_{11}, j\in\tilde{\mathcal{I}}_{12}} X_{ij} = \frac{4}{N} \sum_{i\in\tilde{\mathcal{I}}_{11}, j\in\tilde{\mathcal{I}}_{12}} X_{ij}. 
\end{align*}
Using Lemma \ref{lem:bern}, it holds that
\[ \sum_{i\in\tilde{\mathcal{I}}_{11}, j\in\tilde{\mathcal{I}}_{12}} X_{ij} \geq  \frac12 \cdot |\tilde{\mathcal{I}}_{11}||\tilde{\mathcal{I}}_{12}|p = \frac{k(\ell-k)p}{2}  \]
with probability at least $1-\exp( -k(\ell-k)p/8 )$. Since $k(\ell-k) \geq \ell - 1$, one can write:
\begin{align}\label{eqn:mc-1}
    \|C - \tilde{C}\|_1 \geq \frac{4}{N} \sum_{i\in\tilde{\mathcal{I}}_{11}, j\in\tilde{\mathcal{I}}_{12}} X_{ij} \geq \frac{4}{N} \cdot \frac{(\ell-1)p}{2} = \frac{2(\ell-1)p}{N} 
\end{align}
with the same probability. Considering the union bound over all weight matrices $\tilde{C}$ for which $\mathbb{G}_1(\tilde{C},\tilde{u}^*)$ is disconnected, inequality \eqref{eqn:mc-1} holds
with probability at least
\begin{align*} 
1 - \sum_{k=1}^{\lfloor{\ell/2}\rfloor} \binom{\ell}{k} \exp\left[-\frac{k(\ell-k)p}{8}\right] &\geq 1- \sum_{k=1}^{\lfloor{\ell/2}\rfloor}\left( \frac{\ell e}{k} \right)^{k} \exp\left[-\frac{k(\ell-k)p}{8}\right]\\
&= 1 - \sum_{k=1}^{\lfloor{\ell/2}\rfloor}\exp\left[ k + k\log\left(\frac{\ell}{k}\right) - \frac{k(\ell-k)p}{8} \right], 
\end{align*}
where the inequality uses the relation $\binom{\ell}{k} \leq ( \ell e / k )^{k}$.
Using the relation that $k\leq \ell/2$, we can estimate that
\begin{align*}
&\exp\left[ k + k\log\left(\frac{\ell}{k}\right) - \frac{k(\ell-k)p}{8} \right] \leq \exp\left[ k + k\log{\ell} - \frac{k\ell p}{16} \right]\\
= &\exp\left[ -\frac{k\ell}{16} \left(p -  \frac{16(1+\log{\ell})}{\ell} \right) \right] \leq \exp\left[ -\frac{k\ell}{16} \left(p -  \frac{16(1+\log{n})}{n} \right) \right]\\
\leq &\exp\left[ -\frac{k\ell}{16} \cdot \frac{16\eta \mu \log{n}}{n} \right] = \exp\left( -\frac{\eta \mu k\ell \log{n}}{n} \right) = n^{ -\frac{\eta\mu\ell}{n} \cdot k } \leq n^{-\frac{\eta}{2} \cdot k},
\end{align*}
where the second last inequality is from the assumption on $p$ and the last inequality is from $\ell \geq n/(2\mu)$. By taking the summation over $k=1,\dots,\lfloor{\ell/2}\rfloor$, it follows that
\begin{align*}
1 - \sum_{k=1}^{\lfloor{\ell/2}\rfloor}\exp\left[ k + k\log\left(\frac{\ell}{k}\right) - \frac{k(\ell-k)p}{8} \right] &\geq 1 - \sum_{k=1}^{\lfloor{\ell/2}\rfloor}n^{-\frac{\eta}{2} \cdot k}\\
&\geq 1 - \frac{ n^{ -\frac{\eta}{2} } }{1 - n^{ -\frac{\eta}{2} }} \geq 1 - 2 n^{-\eta/2},
\end{align*}
where the last inequality is due to $n^{-{\eta}/{2}} \geq n^{-1} \geq 1/2$. Therefore, inequality \eqref{eqn:mc-1} holds with probability at least $1-2n^{-\eta/2}$.
Using the lower bound of $N$ in Lemma \ref{lem:nonzero}, the distance between $C$ and $\tilde{C}$ is at least
\[  \frac{2}{3n^2p} \cdot 2(\ell - 1) p = \frac{4(\ell-1)}{3n^2} \]
with probability at least
\[ 1-2n^{-\eta/2} - \exp( -np / 10 ) \geq 1-2n^{-\eta/2} - n^{-4\mu\eta/5} \geq 1-3n^{-\eta/2}. \]

\paragraph{Case II.} For the case when $\cI_{00}(\tilde{C},\tilde{u}^*)$ is non-empty, the analysis is the same as \textit{Case I.} and it holds that
\[  \|C - \tilde{C}\|_1 \geq \frac{2}{3n^2p} \cdot 2(\ell - 1) p = \frac{4(\ell-1)}{3n^2} \]
with probability at least $1-3n^{-\eta/2}$.

\paragraph{Case III.} Finally, we consider the case when $\mathbb{G}_1(\tilde{C},\tilde{u}^*)$ is bipartite. In this case, we show that there exists a set of indices $\cI \subset [n]^2$ with at least $\max\{ \ell^2/2, 1\}$ elements such that 
\[ \tilde{C}_{ij} = 0,\quad \forall (i,j)\in\cI. \]
The proof of the above claim can be found in the proof of Theorem \ref{thm:rip-2} and we omit it here. If $\ell \geq 2$, we have $\ell^2/2 \geq 2(\ell - 1)$ and the proof is the same as \textit{Case I}. Otherwise if $\ell = 1$, the inequality
\[ \|C - \tilde{C}\|_1 \geq \frac{4(\ell-1)}{3n^2} = 0 \]
always holds.

By combining the above three cases, it holds with probability at least $1-9n^{-\eta/2}$ that
\[ \|C - \tilde{C}\|_1 \geq \frac{4(\ell-1)}{3n^2}. \]
\qed\end{proof}

Now, we are ready to prove Theorem \ref{thm:incoh}.
\begin{proof}[Proof of Theorem \ref{thm:incoh}]
Suppose that the instance $\cMC(\tilde{C}, \tilde{u}^*)\in\overline{\cD}$ attains the maximum in \eqref{eqn:metric-new}. Denote 
\[ d_\alpha := \alpha\|C - \tilde{C}\|_1 + (1-\alpha)\|u^* - \tilde{u}^*\|_1. \]
Let
\[ k:=|\cI_1(C,u^*)|,\quad \ell:=|\cI_1(\tilde{C},\tilde{u}^*)|. \]
Similar to Theorem \ref{thm:rip-2}, our goal is to decide the optimal global solution $\tilde{u}^*$. By Lemma \ref{lem:fix-x}, the high-probability lower bound of $\|C-\tilde{C}\|_1$ is increasing in $\ell$. Hence, the optimal choice of $\ell$ is not larger than $k$. We then analyze two cases.
\paragraph{Case I.} We first consider the case when $\ell < n / (2\mu)$. Since $\ell \geq 1$, it follows that $\mu < n / 2$. By Lemma \ref{lem:incoh}, at least $k - \ell > n / (2\mu)$ nonzero entries in $u^*$ are equal to $0$ in $\tilde{u}^*$. Hence, the distance between $u^*$ and $\tilde{u}^*$ satisfies
\[ \|u^* - \tilde{u}^*\|_1 \geq 2\left( 1 - \frac{n}{2\mu} \cdot \frac{\mu}{n} \right) \geq 1. \]
Therefore, it holds that
\begin{align*} 
    \mathbb{D}_\alpha(C,u^*) &= d_\alpha^{-1} = \left[ \alpha\|C-\tilde{C}\|_1 + (1-\alpha)\|u^*-\tilde{u}^*\|_1 \right]^{-1}\\
    &\leq \frac{1}{1-\alpha} \leq \frac{1}{2(1-\alpha)} \cdot\left( 1-\frac{\mu}{n} \right)^{-1} = \frac{1}{2(1-\alpha)\mu} \cdot\left( \frac{1}{\mu}-\frac{1}{n} \right)^{-1}. 
\end{align*}

\paragraph{Case II.} Next, we focus on the case when $\ell \geq n / (2\mu)$. By Lemma \ref{lem:fix-x}, it holds with probability at least $1-3n^{-\eta/2}$ that
\begin{align}\label{eqn:dist-c} \|C - \tilde{C}\|_1 \geq \frac{4(\ell-1)}{3n^2}. \end{align}
By considering the union bound over $\ell \in\cL:=\{ \lceil n/(2\mu)\rceil,\dots, k \}$, the probability that inequality \eqref{eqn:dist-c} holds for all $\ell\in\cL$ is at least
\[ 1- \left(\ell-\frac{n}{2\mu} \right) \cdot 3n^{-\eta/2} \geq 1 - 3n^{-\eta/2 + 1}. \]
In the remainder of this proof, we assume that inequality \eqref{eqn:dist-c} holds for all $\ell\in\cL$.
In addition, we assume without loss of generality that
\[ |u_1^*| \geq |u_2^*| \geq \cdots \geq |u_k^*| > 0;\quad |u_i^*|=0,\quad \forall i\geq k+1. \]
By the assumption of this case, at least $k - \ell$ nonzero entries in $u^*$ are equal to $0$ in $\tilde{u}^*$. Then, we can estimate that
\begin{align*}
d_\alpha &\geq \min_{n/(2\mu) \leq \ell\leq k}\left[ \frac{4\alpha(\ell-1)}{3n^2} + 2(1-\alpha) \sum_{i=\ell+1}^{k} |u_i^*|\right]\\
&\geq \min_{1\leq \ell\leq k} \left[\frac{4\alpha(k-1)}{3n^2} + 2(1-\alpha) \sum_{i=\ell+1}^{k} |u_i^*|\right].
\end{align*}
The above minimization problem can be solved in closed form, which leads to
\begin{align*}
    d_\alpha &\geq {\sum}_{\ell=1}^k \min \left\{ \frac{4\alpha}{3n^2}, 2(1-\alpha)|u_i^*| \right\}.
\end{align*}
By the second property in Lemma \ref{lem:incoh}, we have
\begin{align*}
    d_\alpha &\geq \sum_{i=2}^k \min \left\{ \frac{4\alpha}{3\mu n}|u_i^*|, 2(1-\alpha)|u_i^*| \right\} = \min \left\{ \frac{4\alpha}{3\mu n}, 2(1-\alpha) \right\} \sum_{i=2}^k |u_i^*|\\
    &\geq \min \left\{ \frac{4\alpha}{3\mu n}, 2(1-\alpha) \right\} \cdot \left(1 - \frac{\mu}{n}\right) = \min \left\{ \frac{4\alpha}{3n}, 2(1-\alpha)\mu \right\} \cdot \left(\frac{1}{\mu} - \frac{1}{n}\right).
\end{align*}
The desired upper bound follows from $\mathbb{D}_\alpha(C,u^*) = d_\alpha^{-1}$.

By combining the above two cases, the distance $d_\alpha$ satisfies
\begin{align} \label{eqn:mc-rev-1}
\mathbb{D}_\alpha(C,u^*) \leq \max \left\{ \frac{3n}{4\alpha}, \frac{1}{2(1-\alpha)\mu} \right\} \cdot \left(\frac{1}{\mu} - \frac{1}{n}\right)^{-1} 
\end{align}
with probability at least $1 - 3n^{-\eta/2 + 1}$.

In the case when $\mu \geq n / 16$, the sampling probability $p$ is equal to $1$ and the instance $\cMC(C,u^*)$ satisfies the RIP$_{2,2}$ condition with $\delta= 0$. Hence, we can utilize the upper bound in Theorem \ref{thm:rip-2} to obtain
\[ \mathbb{D}_\alpha(C,u^*) \leq \max\left\{\frac{n}{4\alpha}, \frac{1}{2(1-\alpha)\mu} \right\} \times \min\left\{ \left(\frac{1}{\mu} - \frac{1}{n}\right)^{-1}, 3\mu \right\}. \]
Combing with the upper bound in \eqref{eqn:mc-rev-1}, we conclude the proof of the theorem.
\qed\end{proof}

\subsection{Reduction of problem \eqref{eqn:syn}}

Before discussing the properties of problem instances in Section \ref{sec:aistats}, we prove that the SSCPs of the instance $\cMC(C^\epsilon, u^*)$ are closely related to those of the $m$-dimensional problem
\begin{align}\label{eqn:obj-aistats}
    \min_{x\in\mathbb{R}^{m}} \quad {\sum}_{i\in[m]} (x_i^2 - 1)^2 + \epsilon {\sum}_{i,j\in[m], i\neq j} (x_i x_j-1)^2.
\end{align}
\begin{lemma}\label{lem:aistats}
If problem \eqref{eqn:obj-aistats} has no SSCPs, then the instance $\cMC(C^\epsilon, u^*)$ has no SSCPs. In addition, given a number $N\in\mathbb{N}$, suppose that problem \eqref{eqn:obj-aistats} has $N$ SSCPs with nonzero components at which the objective function has a positive definite Hessian matrix. Then, the instance $\cMC(C^\epsilon, u^*)$ has at least $N$ spurious local minima.
\end{lemma}
\begin{proof}
To prove the first part of the theorem, we assume that problem \eqref{eqn:obj-aistats} has no SSCPs. Suppose that $u^0\in\Rn$ is a second-order critical point of the instance $\cMC(C^\epsilon, u^*)$. Calculating the gradient of $g(u;C,u^*)$ with respect to $u_i$ for any index $i\geq m$ leads to
\begin{align*}
    Z_\epsilon \nabla_i g(u^0;C^\epsilon,u^*) = 4(u_i^0)^3 + 4{\sum}_{j\in[n], \{i,j\}\in\mathbb{E}} u_i^0(u_j^0)^2 = 0,
\end{align*}
where $\nabla_i g(\cdot;C^\epsilon,u^*)$ is $i$-th component of the gradient. By multiplying $u_i^0$ on both sides, it follows that
\[ 4(u_i^0)^4 + 4(u_i^0)^2{\sum}_{j\in[n], \{i,j\}\in\mathbb{E}}(u_j^0)^2 = 0, \]
which implies that $u_i^0 = 0$ for all $i\in\{m+1,\dots,n\}$. Calculating the gradient and the Hessian matrix with respect to $u_{1:m}$ yields that 
\begin{align*}
    Z_\epsilon \nabla_i g(u^0; C^\epsilon,u^*) &= 4\epsilon {\sum}_{j\in[m],j\neq i} u_j^0(u_i^0u_j^0 - 1/m^2)\\
    &\hspace{10em}+ 4u_i^0[(u_i^0)^2 - 1/m^2],\quad \forall i\in[m];\\
    Z_\epsilon \nabla_{ii}^2 g(u^0; C^\epsilon,u^*) &= 12 (u_i^0)^2 - 4/m^2 + 4\epsilon {\sum}_{j\in[m], j\neq i} (u_j^0)^2,\quad \forall i\in[m];\\
    Z_\epsilon \nabla_{ij}^2 g(u^0; C^\epsilon,u^*) &= 4\epsilon (2u_i^0u_j^0 - 1),\quad \forall i,j\in[m]\quad \st\quad i\neq j,
\end{align*}
where $\nabla_{ij} g(\cdot;C^\epsilon,u^*)$ is the $(i,j)$-th component of the Hessian matrix. By defining $x^0\in\R^m$ as $x_i^0 := m u_i^0$ for all $i\in[m]$, the above gradient and Hessian matrix turn out to be the same as those of problem \eqref{eqn:obj-aistats}. Since the first $m$ entries of $\nabla g(u^0;C^\epsilon,u^*)$ are $0$ and the first $m$-by-$m$ principle sub-matrix of $\nabla^2 g(u^0;C^\epsilon,u^*)$ is positive semi-definite, the point $x^0$ is a second-order critical point of problem \eqref{eqn:obj-aistats}. In addition, the point $u^0$ is a global optimum if and only if $|u_i^0|=1/m$ for all $i\in[m]$, which is further equivalent to $x_i^0=1$ for all $i\in[m]$ and $x^0$ is the global solution to problem \eqref{eqn:obj-aistats}. Therefore, the point $x^0$ is a SSCP if $u^0$ is a SSCP, which is a contradiction to the assumption that problem \eqref{eqn:obj-aistats} has no SSCPs. Therefore, the point $u^0$ is a global minimum of the instance $\cMC(C^\epsilon, u^*)$.

For the second part of the theorem, suppose that $x^0$ is a SSCP of problem \eqref{eqn:obj-aistats}, where the Hessian matrix is positive definite and $x^0_i \neq 0$ for all $i\in[m]$. We construct $u^0\in\Rn$ by setting $u^0_i := m^{-1} x^0_i$ for all $i\in[m]$ and $u^0_i=0$ for all $i\in\{m+1,\dots,n\}$. By similar calculations, we can prove that the Hessian matrix at $u^0$ is a block diagonal matrix with two blocks, where the first block is $H(x;\epsilon)$ and the second block is a diagonal matrix with positive diagonal entries. Moreover, the gradient at $u^0$ is equal to $0$. Hence, $u^0$ is a SSCP with a positive definite Hessian matrix. The construction shows that the mapping from $x^0$ to $u^0$ is injective.
\qed\end{proof}

\subsection{Proof of Theorem \ref{thm:upper-lower}}
\label{apd:upper-lower}

To simplify the notations in the following proofs, we denote the gradient and the Hessian matrix of the objective function of problem \eqref{eqn:obj-aistats} by
\begin{align*} 
\mathrm{g}_i(x;\epsilon) &:= 4\left[ x_i^3 - x_i + \epsilon {\sum}_{j\neq i} x_j (x_ix_j - 1) \right],\quad\forall i\in[m];\\
H_{ii}(x;\epsilon) &:= 4\left[ 3x_i^2 - 1 + \epsilon {\sum}_{j\neq i} x_j^2 \right],\quad\forall i\in[m];\\
H_{ij}(x;\epsilon) &:= 4\epsilon (2x_ix_j - 1),\quad \forall i,j\in[m]\quad \mathrm{s.t.}\quad i\neq j.
\end{align*}
The following theorem guarantees that the instance $\cMC(C^\epsilon, u^*)$ does not have spurious local minima when $\epsilon \geq O(m^{-1})$.
\begin{theorem}\label{thm:upper}
If $\epsilon > 18 / m$, the instance $\cMC(C^\epsilon, u^*)$ does not have SSCPs, namely, all second-order critical points are global minima associated with the ground truth solution $M^*$. 
\end{theorem}

\begin{proof}
By Lemma \ref{lem:aistats}, we only need to prove that problem \eqref{eqn:obj-aistats} has no SSCPs. The conclusion holds when $\epsilon=1$ since the $\delta$-RIP$_{2,2}$ condition holds with $\delta=0$ and the results in \cite{zhang2019sharp} guarantee that there is no SSCP. In the remainder of the proof, we assume that $\epsilon\in[0,1)$. Suppose that $x^0\in \R^m$ is a second-order critical point of problem \eqref{eqn:obj-aistats}.
Denote
\[ S_k := {\sum}_{i=1}^m (x_i^0)^k,\quad \forall k\in\mathbb{N}. \]
Using the first-order optimality conditions, we have
\begin{align}\label{eqn:first}
0 &= \frac14 {\sum}_{i\in[m]} g_i(x^0;\epsilon) = (1-\epsilon)S_3 - (1-\epsilon)S_1 - m\epsilon S_1 + \epsilon S_1S_2,\\
\nonumber 0 &= \frac14 {\sum}_{i\in[m]} x_i^0 g_i(x^0;\epsilon) = (1-\epsilon)S_4 - (1-\epsilon)S_2 - \epsilon S_1^2 + \epsilon S_2^2.
\end{align}
Using the second-order necessary optimality conditions, the curvatures of the objective function along the directions
\[ c_+ := (x_1^0-1,\dots, x_m^0-1)\quad \text{and}\quad c_- := (x_1^0+1,\dots, x_m^0+1) \]
are given by
\begin{align*}
c_+^T H(x;\epsilon) c_+ / 4 &=3(1-\epsilon)(S_4-2S_3+S_2) + [\epsilon S_2 - (1-\epsilon)](S_2 - 2S_1 + m)\\
&\hspace{4em} + 2\epsilon(S_2^2 - 2S_2S_1 + S_1^2) - \epsilon(S_1^2 - 2n S_1 + m^2)\geq0,\\
c_-^T H(x;\epsilon) c_- / 4 &= 3(1-\epsilon)(S_4+2S_3+S_2) + [\epsilon S_2 - (1-\epsilon)](S_2 + 2S_1 + m)\\
&\hspace{4em}+ 2\epsilon(S_2^2 + 2S_2S_1 + S_1^2) - \epsilon(S_1^2 + 2n S_1 + m^2)\geq 0.
\end{align*}
Using the relations in \eqref{eqn:first}, we can write $S_3$ and $S_4$ in terms of $S_1$ and $S_2$, which leads to
\begin{align}\label{eqn:second}
    [m\epsilon + 5(1-\epsilon)]S_2 + 4\epsilon S_1^2 - 4[m\epsilon+(1-\epsilon)]\cdot |S_1| - [m^2\epsilon + m(1-\epsilon)] \geq 0.
\end{align}
Let $c$ be a positive number such that
\[ S_1^2 = c S_2. \]
Using H\"{o}lder's inequality, we have $c\in[1,m]$. We note that in the case when $S_2=0$, it holds that $S_1 = 0$ and we can choose $c$ to be any constant in $[1,m]$. Then, inequality \eqref{eqn:second} can be written as
\begin{align}\label{eqn:second-1}
[m\epsilon + 5(1-\epsilon)+4\epsilon c]S_2 - 4[m\epsilon+(1-\epsilon)]\sqrt{c} \cdot\sqrt{S_2} - [m^2\epsilon + m(1-\epsilon)] \geq 0.
\end{align}
Inequality \eqref{eqn:second-1} is a quadratic inequality in $\sqrt{S_2}$ and thus, it can be solved in closed form, namely, inequality \eqref{eqn:second-1} is equivalent to
\begin{align}\label{eqn:second-2}
&\sqrt{S_2}\\
\nonumber\geq& \frac{4[m\epsilon+(1-\epsilon)]\sqrt{c} + \sqrt{ 4[m\epsilon+(1-\epsilon)][8m\epsilon c + 4(1-\epsilon) c+m^2\epsilon+5m(1-\epsilon) ]} }{2[m\epsilon + 5(1-\epsilon)+4\epsilon c]}\\
\nonumber= & m\sqrt{m\epsilon + (1-\epsilon)} \cdot \bigg[ \sqrt{ [8m\epsilon+4(1-\epsilon)]c + m^2\epsilon + 5m(1-\epsilon) }\\
\nonumber&\hspace{21em} - \sqrt{ 4[m\epsilon+(1-\epsilon)]c } \bigg]^{-1}.
\end{align}
Consider the function
\begin{align*} 
e(c) :=& \sqrt{ [8m\epsilon+4(1-\epsilon)]c + m^2\epsilon + 5m(1-\epsilon) } - \sqrt{ 4[m\epsilon+(1-\epsilon)]c },\\
&\hspace{24em}\forall c\in[1,m], 
\end{align*}
which is the negative of a unimodal function\footnote{In this work, we say a function $f:\R\mapsto\R$ is a \textit{unimodal function} if there exists a constant $c\in\R$ such that $f$ is increasing on $(-\infty,c]$ and decreasing on $[c,+\infty)$.}. Hence, the maximum value of $e(c)$ on $[1,m]$ is attained at $1$ or $m$. Let
\[ C := m\epsilon > 18. \]
We calculate that
\begin{align*}
e(m) &= \sqrt{ 9m[m\epsilon + (1-\epsilon)] } - \sqrt{ 4m[m\epsilon + (1-\epsilon)] }\\
&= \sqrt{m[m\epsilon + (1-\epsilon)]} \leq \sqrt{m(C + 1)} \leq \sqrt{2mC},\\
e(1) &= \sqrt{ 8m\epsilon+4(1-\epsilon) + m^2\epsilon + 5m(1-\epsilon) } - \sqrt{ 4[m\epsilon+(1-\epsilon)] }\\
&\leq \sqrt{8C + 4 + mC + 5m} \leq \sqrt{2(m+8) C}.
\end{align*}
Hence, we have
\[ e(c) \leq \sqrt{2(m+8) C},\quad \forall c\in[1,m]. \]
By combining with \eqref{eqn:second-2}, it follows that
\begin{align*}
\sqrt{S_2} &\geq m\sqrt{C + (1-\epsilon)} \cdot \left[\sqrt{2(m+8) C}\right]^{-1}\\
&\geq m\sqrt{C} \cdot \left[\sqrt{2(m+8) C}\right]^{-1} = \frac{m}{\sqrt{2(m+8)}},
\end{align*}
which further leads to
\begin{align}\label{eqn:second-3} 
S_2 \geq \frac{m^2}{2(m+8)} \geq \frac{m}{18}. 
\end{align}
Therefore, we obtain that
\[ \frac{\epsilon}{1-\epsilon} S_2 -1 \geq \frac{\epsilon m}{18} - 1 > 0. \]
Using the first-order optimality condition, each component $x_i^0$ is the solution to the third-order polynomial equation
\begin{align}\label{eqn:first-1} 
    g_i(x;\epsilon) = x_i^3 + \left[ \frac{\epsilon}{1-\epsilon} S_2 - 1 \right] x_i - \frac{\epsilon}{1-\epsilon} S_1 = 0,\quad \forall i\in[m]. 
\end{align}
Since the first-order coefficient $\epsilon/[(1-\epsilon) S_2] - 1$ is positive, the derivative of the polynomial is positive and the equation has a unique real root $x_0$. Hence, we know
\[ x_1^0 =\cdots = x_m^0 = x_0. \]
The equation in \eqref{eqn:first-1} now becomes
\[ x_0^3 + \left[ \frac{\epsilon}{1-\epsilon} \cdot m x_0^2 - 1 \right] x_0 - \frac{\epsilon}{1-\epsilon} \cdot  m x_0 = \left[ \frac{m\epsilon}{1-\epsilon} + 1 \right](x_0^3 - x_0) = 0, \]
which gives $x_0 \in\{-1,0,1\}$. If $x_0 \in \{-1,1\}$, then the point $x^0$ is a global optimum. Otherwise if $x_0 = 0$, it follows that $x^0 = 0$ and $S_2 = 0$, which contradicts \eqref{eqn:second-3}. Combining the two cases, we conclude that problem \eqref{eqn:obj-aistats} does not have SSCPs, which implies that the instance $\cMC(C^\epsilon,u^*)$ also has no SSCPs. 
\qed\end{proof}

Then, we consider the regime of $\epsilon$ where the instance $\cMC(C^\epsilon, u^*)$ has spurious solutions. The following theorem studies the case when $m$ is an even number.
\begin{theorem}\label{thm:even}
Suppose that $m$ is an even number. If $\epsilon < 1 / (m+1)$, then the instance $\cMC(C^\epsilon, u^*)$ has at least $2^{m/2}$ spurious local minima.
\end{theorem}
\begin{proof}
By Lemma \ref{lem:aistats}, we only need to show that problem \eqref{eqn:obj-aistats} has at least $\binom{m}{m/2}$ SSCPs whose associated Hessian matrices are positive definite and whose components are nonzero. We consider a point $x^0\in\R^m$ such that
\[ (x_i^0)^2 = \frac{1-\epsilon}{1+(m-1)\epsilon} > 0,\quad \forall i\in[m];\quad {\sum}_{i\in[m]} x_i^0 = 0. \]
The above equations have a solution since $m$ is an even number. By a direct calculation, we can verify that the gradient $g(x^0;\epsilon)$ is equal to $0$. We only need to show that the Hessian matrix $H(x^0;\epsilon)$ is positive definite, namely
\[ c^T H(x^0;\epsilon) c > 0,\quad \forall c\in\mathbb{R}^m \backslash\{0\}. \]
The above condition is equivalent to
\begin{align*} 
    &\left[(3 + (m-3)\epsilon)\left(x_1^0\right)^2 - 1 + \epsilon\right] {\sum}_{i\in[m]} c_i^2 - \epsilon \left({\sum}_{i\in[m]}  c_i\right)^2\\
    &\hspace{9em} + 2\epsilon \left(x_1^0\right)^2 \left({\sum}_{i\in[m]}  \mathrm{sign}(x_i^0) c_i \right)^2 > 0,\quad \forall c\in\mathbb{R}^n \backslash\{0\}. 
\end{align*}
Under the normalization constraint $\|c\|_2 = 1$, the Cauchy inequality implies that the minimum of the left-hand side is attained by
\[ c_1= \cdots = c_m = 1 / \sqrt{m}. \]
Therefore, the Hessian is positive definite if and only if
\[ (3 + (m-3)\epsilon)\left(x_1^0\right)^2 - 1 + \epsilon > m\epsilon. \]
By substituting $(x_1^0)^2=(1-\epsilon)/[1+(m-1)\epsilon]$, the above condition is equivalent to
\[ 2 - (m+4)\epsilon - (m-2)(m+1)\epsilon^2 > 0. \]
Using the condition that $(m+1)\epsilon < 1$, we obtain that 
\begin{align*}
2 - (m+4)\epsilon - (m-2)(m+1)\epsilon^2 &> 1 - 3\epsilon - (m-2)\epsilon = 1 - (m+1)\epsilon > 0,
\end{align*}
where the first inequality is from the fact that $m\geq 2$, which follows from the assumption that $m>0$ is an even number.

To estimate the number of SSCPs, we observe that $m/2$ components of $x^0$ have a positive sign and the other $m/2$ components have a negative sign. Hence, there are at least 
\[ \binom{m}{m/2} \]
spurious SSCPs. The estimate on the combinatorial number is in light of the inequality $\binom{n}{k}\geq (n/k)^k$.
\qed\end{proof}

The estimation of the odd number case is similar and we present the result in the following theorem.
\begin{theorem}\label{thm:odd}
Suppose that $m$ is an odd number. If $\epsilon < 1 / [13(m+1)]$, then the instance $\cMC(C^\epsilon, u^*)$ has at least $[{2m}/(m+1)]^{(m+1)/2}$ spurious local minima.
\end{theorem}
\begin{proof}
We pursue a similar way as in Theorem \ref{thm:even} to construct spurious solutions. By Lemma \ref{lem:aistats}, we only need to show that problem \eqref{eqn:obj-aistats} has at least $\binom{m}{(m-1)/2}$ SSCPs whose Hessian matrices are positive definite and whose components are nonzero. Let $k:=(m-1)/2\in\mathbb{Z}$. We first choose a subset
\[ \mathcal{I} \subset [m],\quad |\mathcal{I}| = k. \]
Then, we consider the point $x\in\mathbb{R}^m$, where
\[ u_i = y_1,\quad \forall i\in\mathcal{I},\quad u_i = y_2,\quad\forall i\notin \mathcal{I}, \]
where $y_1$ and $y_2$ are real numbers such that
\begin{align}\label{eqn:odd-1} 
&(1+k\epsilon)(1+2k\epsilon) [ (1-\epsilon)y_2^2 ]^3 -2(1+k\epsilon)(1+(k-1)\epsilon)[ (1-\epsilon)y_2^2 ]^2\\
\nonumber& + (1+k\epsilon)(1+(k-1)\epsilon)(2k^2\epsilon^2+2k\epsilon^2-k\epsilon-\epsilon+1)[ (1-\epsilon)y_2^2 ]\\
\nonumber&\hspace{16em} - k^2\epsilon^2(1+(k-1)\epsilon)(1-\epsilon)^2 = 0,\\
\nonumber&y_1 = \frac{y_2}{k\epsilon} \cdot \frac{(1+k\epsilon)[ (1-\epsilon)y_2^2 ] - (k^2\epsilon^2+(k-1)\epsilon+1)}{[ (1-\epsilon)y_2^2 ] + (1+(k-1)\epsilon)}.
\end{align}
We first assume the existence of the constants $y_1$ and $y_2$. After some direct calculations, one can show that the conditions in \eqref{eqn:odd-1} imply the first-order optimality condition of the instance $\cMC(C^\epsilon, u^*)$, i.e.,
\begin{align*}
y_1^3 - y_1 + \epsilon[ (k-1)y_1^2 + (k+1)y_2^2 ]y_1 - \epsilon[ (k-1)y_1 + (k+1)y_2 ] &= 0,\\
y_2^3 - y_2 + \epsilon[ k y_1^2 + ky_2^2 ]y_2 - \epsilon[ ky_1 + ky_2 ] &= 0.
\end{align*}
Therefore, the point $x$ is a first-order critical point of the instance $\cMC(C^\epsilon, u^*)$. In addition, the following relations result from the condition \eqref{eqn:odd-1}:
\begin{align}\label{eqn:odd-7}
    (1-\epsilon)y_1y_2(y_1+y_2) &= -\epsilon[ ky_1 + (k+1)y_2 ],\\
    \nonumber (1-\epsilon)(y_1^2+y_1y_2+y_2^2-1) &= -\epsilon[  ky_1^2 + (k+1)y_2^2  ].
\end{align}

Now, we prove the existence of $y_1,y_2$ and estimate their values. We note that the first equation in \eqref{eqn:odd-1} is a third-order polynomial equation for $(1-\epsilon)y_2^2$, which has at least one real root. To show that the equation has a positive root, we observe that the coefficient of the third-order term is $(1+k\epsilon)(1+2k\epsilon)>0$ and the value at zero is $- k^2\epsilon^2(1+(k-1)\epsilon)(1-\epsilon)^2<0$. Therefore, the polynomial equation for $(1-\epsilon)y_2^2$ has at least one positive root and $y_2$ is well defined. We provide a more accurate estimate to $y_1$ and $y_2$, namely, we show that there exists a solution $(y_1,y_2)$ to equations \eqref{eqn:odd-1} such that
\begin{align*} 
y_1 \in [-2, -3/5],\quad y_2 \in [ 1/2, 1 ].
\end{align*}
Define the polynomial function
\begin{align*}
g(z) :=& (1+k\epsilon)(1+2k\epsilon) z^3 -2(1+k\epsilon)(1+(k-1)\epsilon)z^2\\
& + (1+k\epsilon)(1+(k-1)\epsilon)(2k^2\epsilon^2+2k\epsilon^2-k\epsilon-\epsilon+1)z - k^2\epsilon^2(1+(k-1)\epsilon)(1-\epsilon)^2.
\end{align*}
We first estimate $g( 1 - (2k+1)\epsilon )$ as follows:
\begin{align*}
&g( 1 - (2k+1)\epsilon )\\
= &(1+k\epsilon)[1 - (2k+1)\epsilon]\Big[ (1+2k\epsilon)[1-(2k+1)\epsilon]^2 - 2[1+(k-1)\epsilon][1-(2k+1)\epsilon]\\
&+ [1+(k-1)\epsilon][ 1-(k+1)\epsilon +2k(k+1)\epsilon^2] \Big] - k^2\epsilon^2(1+(k-1)\epsilon)(1-\epsilon)^2\\
=& (1+k\epsilon)[1 - (2k+1)\epsilon]\Big[ k^2\epsilon^2 + 2k^2(5k+4)\epsilon^3 \Big] - k^2\epsilon^2(1+(k-1)\epsilon)(1-\epsilon)^2\\
\geq& k^2\epsilon^2(1+k\epsilon)[1 - (2k+1)\epsilon][ 1 + 2(5k+4)\epsilon ] - k^2\epsilon^2(1+             k\epsilon)(1-\epsilon)^2\\
=& k^2\epsilon^2(1+k\epsilon)[ (8k+9)\epsilon - [2(2k+1)(5k+4)+1]\epsilon^2 ]\\
\geq& k^2\epsilon^2(1+k\epsilon)[ (8k+8)\epsilon - 20(k+1)^2\epsilon^2 ] > 0,
\end{align*}
where the last inequality is due to $(k+1)\epsilon = (n+1)\epsilon/2 < 2/5$. Next, we estimate $g( 1 - (3k/2+1)\epsilon )$ as follows:
\begin{align*}
&g( 1 - (3k/2+1)\epsilon )\\
= &(1+k\epsilon)[1 - (k+1)\epsilon ]\Big[ (1+2k\epsilon)[1-(3k/2+1)\epsilon ]^2 - 2[1+(k-1)\epsilon][1-(3k/2+1)\epsilon ]\\
&+ [1+(k-1)\epsilon][ 1-(k+1)\epsilon +2k(k+1)\epsilon^2] \Big] - k^2\epsilon^2(1+(k-1)\epsilon)(1-\epsilon)^2\\
=& (1+k\epsilon)[1 - (3k/2 +1)\epsilon] \Big[ k^2\epsilon^2/4 + k^2(13k/2+6)\epsilon^3 \Big] - k^2\epsilon^2(1+(k-1)\epsilon)(1-\epsilon)^2\\
=& k^2\epsilon^2(1+k\epsilon)[1 - (3k/2 +1)\epsilon] [ 1/4 + (13k/2+6)\epsilon ] - k^2\epsilon^2(1+(k-1)\epsilon)(1-\epsilon)^2\\
\leq& k^2\epsilon^2(1+k\epsilon)[1 - (3k/2 +1)\epsilon] [ 1/4 + (13k/2+6)\epsilon ] - k^2\epsilon^2\cdot [(1+k\epsilon)/2] \cdot (1-\epsilon)^2\\
\leq& k^2\epsilon^2(1+k\epsilon)\Big[ [1 - (3k/2+1)\epsilon] [ 1/4 + (13k/2+6)\epsilon ] - (1-\epsilon)^2 / 2\Big]\\
=& k^2\epsilon^2(1+k\epsilon)\Big[ -1/4 + (49k/8+27/4)\epsilon - (39k^2/4+31k/2+13/2)\epsilon^2 \Big]\\
\leq& k^2\epsilon^2(1+k\epsilon)\Big[ -1/4 + 27(k+1)/4\epsilon - 39(k+1)^2\epsilon^2/4 \Big] < 0,
\end{align*}
where the last inequality is in light of $(k+1)\epsilon = (n+1)\epsilon/2 < 1/26 $.
%
%
Combining the above two estimates, we conclude that there exists a solution $y_2$ to the first equation in \eqref{eqn:odd-1} such that
\begin{align}
\label{eqn:odd-2} 
(1-\epsilon)y_2^2 \in [ 1-(2k+1)\epsilon, 1-(3k/2+1)\epsilon ].
\end{align}
Hence,
\begin{align}\label{eqn:odd-5}
y_2 &\leq \sqrt{\frac{1-(3k/2+1)\epsilon}{1-\epsilon}} \leq 1
\end{align}
and
\begin{align}\label{eqn:odd-6}
y_2 &\geq \sqrt{\frac{1-(2k+1)\epsilon}{1-\epsilon}} \geq \sqrt{1-(2k+1)\epsilon} \geq \frac{1}{2}.
\end{align}
Now, we use the second equation in \eqref{eqn:odd-1} to estimate $y_1$, which leads to
\begin{align*}
&\frac{(1+k\epsilon)[ (1-\epsilon)y_2^2 ] - (k^2\epsilon^2+(k-1)\epsilon+1)}{k\epsilon}\\
\geq& \frac{(1+k\epsilon)[ 1-(2k+1)\epsilon ] - (k^2\epsilon^2+(k-1)\epsilon+1)}{k\epsilon} \\
=& -2 - (3k+1)\epsilon
\end{align*}
and
\begin{align*}
&\frac{(1+k\epsilon)[ (1-\epsilon)y_2^2 ] - (k^2\epsilon^2+(k-1)\epsilon+1)}{k\epsilon}\\
\leq &\frac{(1+k\epsilon)[ 1-(3k/2+1)\epsilon ] - (k^2\epsilon^2+(k-1)\epsilon+1)}{k\epsilon} \\
= &-\frac{3}{2} - \left(\frac{5k}{2}+1\right)\epsilon.
\end{align*}
On the other hand, we have
\begin{align*}
\frac{y_2}{[ (1-\epsilon)y_2^2 ] + (1+(k-1)\epsilon)} &= \frac{1}{(1-\epsilon)(y_2+y_2^{-1}) + k\epsilon} \leq \frac{1}{2(1-\epsilon) + k\epsilon}.
\end{align*}
Using the bound in \eqref{eqn:odd-2}, it holds that
\[ y_2 \geq \sqrt{\frac{1-(2k+1)\epsilon}{1-\epsilon}} \geq \frac{1-(2k+1)\epsilon}{1-\epsilon} = \frac12 - \frac{1-(4k+1)\epsilon}{2(1-\epsilon)} \geq \frac12. \]
Therefore,
\begin{align*}
\frac{y_2}{[ (1-\epsilon)y_2^2 ] + (1+(k-1)\epsilon)} &= \frac{1}{(1-\epsilon)(y_2+y_2^{-1}) + k\epsilon} \geq \frac{1}{2.5(1-\epsilon) + k\epsilon}.
\end{align*}
Combining the above inequalities and the second equation in \eqref{eqn:odd-1} yields that 
\begin{align}\label{eqn:odd-3}
y_1 &\geq \frac{-2-(3k+1)\epsilon}{2(1-\epsilon) + k\epsilon} \geq -\left( 1 + \frac{5\epsilon}{1-\epsilon} \right) \geq -2
\end{align}
and
\begin{align}\label{eqn:odd-4}
y_1 &\leq \frac{-3/2-(5k/2+1)\epsilon}{2.5(1-\epsilon) + k\epsilon} \leq -\frac{1.5}{2.5} = -\frac{3}{5},
\end{align}
where the last inequality in \eqref{eqn:odd-3} results from $\epsilon \leq 1/(3(k+1)) \leq 1/6$.
In summary, inequalities \eqref{eqn:odd-5}-\eqref{eqn:odd-4} lead to
\[ y_1 \in [-2, -3/5],\quad y_2 \in [ 1/2, 1 ]. \]
We then prove that $y_1 + 2y_2 > y_2 \geq 0.5$, which is equivalent to
\begin{align*}
\frac{y_2}{k\epsilon} \cdot \frac{(1+k\epsilon)[ (1-\epsilon)y_2^2 ] - (k^2\epsilon^2+(k-1)\epsilon+1)}{[ (1-\epsilon)y_2^2 ] + (1+(k-1)\epsilon)} + y_2 > 0.
\end{align*}
Since $y_2 > 0$, we only need to prove that
\begin{align*}
0 &< (1+k\epsilon)[ (1-\epsilon)y_2^2 ] - (k^2\epsilon^2+(k-1)\epsilon+1) + k\epsilon\Big[[ (1-\epsilon)y_2^2 ] + (1+(k-1)\epsilon) \Big]\\
&= (1+2k\epsilon)[ (1-\epsilon)y_2^2 ] - (k^2\epsilon^2+(k-1)\epsilon+1) + k\epsilon(1+(k-1)\epsilon).
\end{align*}
Using inequality \eqref{eqn:odd-2}, it suffices to show that
\begin{align*}
&(1+2k\epsilon)[1-(3k/2+1)\epsilon)] + k\epsilon(1+(k-1)\epsilon) > 1+(k-1)\epsilon+k^2\epsilon^2\\
\iff & \frac{1}{2} k\epsilon > 3k\left( k + \frac{3}{2} \right)\epsilon^2 \iff 3(2k+3)\epsilon < 1 \Leftarrow 6(k+1)\epsilon < 1 ,
\end{align*}
where the last inequality holds since $(k+1)\epsilon = (n+1)\epsilon/2 < 1/6$.

Now, we verify the second-order sufficient optimality condition. For every $c\in\mathbb{R}^m\backslash\{0\}$, we calculate that
\begin{align*}
c^T H(x;\epsilon)c &= \sum_{i\in\mathcal{I}}\left[ 3y_1^2 - 1 + \epsilon( (k-1)y_1^2 + (k+1)y_2^2 ) \right]c_i^2\\
&\quad + \sum_{i\notin\mathcal{I}}\left[ 3y_2^2 - 1 + \epsilon( ky_1^2 + ky_2^2 ) \right]c_i^2 + \sum_{i,j\in\mathcal{I},i\neq j} \epsilon\left( 2y_1^2 - 1 \right)c_ic_j\\
&\quad + \sum_{i,j\notin\mathcal{I},i\neq j} \epsilon\left( 2y_2^2 - 1 \right)c_ic_j + 2\sum_{i\in\mathcal{I},j\notin\mathcal{I}} \epsilon\left( 2y_1y_2 - 1 \right)c_ic_j\\
&= \left[ 3y_1^2 - 1 + \epsilon( (k-1)y_1^2 + (k+1)y_2^2 ) - \epsilon\left( 2y_1^2 - 1 \right) \right]\sum_{i\in\mathcal{I}}c_i^2\\
&\quad + \left[ 3y_2^2 - 1 + \epsilon( ky_1^2 + ky_2^2 ) - \left( 2y_2^2 - 1 \right) \right] \sum_{i\notin\mathcal{I}}c_i^2\\
&\quad + \epsilon\left( 2y_1^2 - 1 \right)\left(\sum_{i\in\mathcal{I}}c_i\right)^2 + \epsilon\left( 2y_2^2 - 1 \right)\left(\sum_{i\notin\mathcal{I}}c_i\right)^2\\
&\quad + 2 \epsilon\left( 2y_1y_2 - 1 \right)\left(\sum_{i\in\mathcal{I}}c_i\right)\left(\sum_{i\notin\mathcal{I}}c_i\right).
\end{align*}
Using the Cauchy inequality, the above expression is positive if and only if
\begin{align*}
&\left[ 3y_1^2 - 1 + \epsilon( (k-1)y_1^2 + (k+1)y_2^2 ) - \epsilon\left( 2y_1^2 - 1 \right) \right]\cdot\frac{1}{k}\left(\sum_{i\in\mathcal{I}}c_i\right)^2\\
+& \left[ 3y_2^2 - 1 + \epsilon( ky_1^2 + ky_2^2 ) - \left( 2y_2^2 - 1 \right) \right] \cdot\frac{1}{k+1}\left(\sum_{i\notin\mathcal{I}}c_i\right)^2\\
+& \epsilon\left( 2y_1^2 - 1 \right)\left(\sum_{i\in\mathcal{I}}c_i\right)^2 + \epsilon\left( 2y_2^2 - 1 \right)\left(\sum_{i\notin\mathcal{I}}c_i\right)^2\\
+& 2 \epsilon\left( 2y_1y_2 - 1 \right)\left(\sum_{i\in\mathcal{I}}c_i\right)\left(\sum_{i\notin\mathcal{I}}c_i\right)>0.
\end{align*}
We denote
\[ A := {\sum}_{i\in\mathcal{I}}c_i,\quad B := {\sum}_{i\notin\mathcal{I}}c_i. \]
Then, the second-order sufficient condition is equivalent to
\begin{align*}
&\left[ 3y_1^2 - 1 + \epsilon( (k-1)y_1^2 + (k+1)y_2^2 ) - \epsilon\left( 2y_1^2 - 1 \right) \right]\cdot\frac{1}{k}A^2\\
+& \left[ 3y_2^2 - 1 + \epsilon( ky_1^2 + ky_2^2 ) - \left( 2y_2^2 - 1 \right) \right]\\
&\hspace{7em} \cdot\frac{1}{k+1}B^2 + 2\epsilon(y_1A + y_2B)^2 - \epsilon(A+B)^2 > 0.
\end{align*}
The above inequality is a quadratic inequality in $A$ and $B$, which can be rewritten as
\begin{align*}
&\left[\frac{1}{k}\left[ 3y_1^2 - 1 + \epsilon( (k-1)y_1^2 + (k+1)y_2^2 ) - \epsilon\left( 2y_1^2 - 1 \right) \right] + \epsilon (2y_1^2-1)\right] A^2\\
+& 2\epsilon(2y_1y_2 - 1)AB\\
&\hspace{0em} + \left[\frac{1}{k+1}\left[ 3y_2^2 - 1 + \epsilon( ky_1^2 + ky_2^2 ) - \epsilon\left( 2y_2^2 - 1 \right) \right] + \epsilon (2y_2^2-1)\right] B^2 > 0.
\end{align*}
Therefore, the positivity condition can be verified through the discriminant, namely,
\begin{align*}
\epsilon^2(2y_1y_2 - 1)^2 <& \left[\frac{1}{k}\left[ 3y_1^2 - 1 + \epsilon( (k-1)y_1^2 + (k+1)y_2^2 ) - \epsilon\left( 2y_1^2 - 1 \right) \right] + \epsilon (2y_1^2-1)\right]\\
& \cdot \left[\frac{1}{k+1}\left[ 3y_2^2 - 1 + \epsilon( ky_1^2 + ky_2^2 ) - \epsilon\left( 2y_2^2 - 1 \right) \right] + \epsilon (2y_2^2-1)\right].
\end{align*}
Using the second property in \eqref{eqn:odd-7}, the above condition can be simplified into
\begin{align*}
&-(1-\epsilon)^2(y_2-y_1)^2(2y_1+y_2)(y_1+2y_2) + (k+1)\epsilon(1-\epsilon)(2y_2^2-1)(y_1-y_2)(2y_1+y_2)\\
&\quad + k\epsilon(1-\epsilon)(2y_1^2-1)(y_2-y_1)(y_1+2y_2) > k(k+1)\epsilon^2(y_1-y_2)^2.
\end{align*}
Since $y_2 > y_1$, it suffices to have
\begin{align*}
&-(1-\epsilon)^2(y_2-y_1)(2y_1+y_2)(y_1+2y_2) - (k+1)\epsilon(1-\epsilon)(2y_2^2-1)(2y_1+y_2)\\
&\hspace{13em} + k\epsilon(1-\epsilon)(2y_1^2-1)(y_1+2y_2) > k(k+1)\epsilon^2(y_2-y_1).
 \end{align*}
We can estimate that
\begin{align*}
&-(1-\epsilon)^2(y_2-y_1)(2y_1+y_2)(y_1+2y_2) - (k+1)\epsilon(1-\epsilon)(2y_2^2-1)(2y_1+y_2)\\
&\quad + k\epsilon(1-\epsilon)(2y_1^2-1)(y_1+2y_2) - k(k+1)\epsilon^2(y_2-y_1)\\
&\geq [1-(k+1)\epsilon]^2 \cdot 1.1 \cdot 0.2 \cdot 0.5 - (k+1)\epsilon \cdot 1 \cdot 0.5 \cdot 2 -(k+1)\epsilon \cdot 1 \cdot 0.64 \cdot 1.4 - (k+1)^2\epsilon^2 \cdot 3\\
&= 0.11[1-(k+1)\epsilon]^2 - 1.896(k+1)\epsilon - 3(k+1)^2\epsilon^2= 0.11 - 2.116(k+1)\epsilon - 2.89(k+1)^2\epsilon^2\\
&\geq 0.11 - 2.116(k+1)\epsilon - 2.89(k+1)^2\epsilon^2 > 0,
 \end{align*}
where the last inequality is due to $(k+1)\epsilon = (n+1)\epsilon/2 < 1/26 $. Thus, we have shown that the Hessian matrix is positive definite and the point $x$ is a SSCP.

To count the number of spurious solutions, we notice that the subset $\mathcal{I}$ has $\binom{m}{(m+1)/2}$ different choices. Hence, the total number of SSCPs is at least $\binom{m}{(m+1)/2}$. The estimate on the combinatorial number follows from $\binom{n}{k}\geq (n/k)^k$.
\qed\end{proof}
By combining Theorems \ref{thm:upper}-\ref{thm:odd}, we complete the proof of Theorem \ref{thm:upper-lower}.

\subsection{Proof of Theorem \ref{thm:aistats-1}}
\label{apd:aistats-1}

The proof of Theorem \ref{thm:aistats-1} relies on the following lemma, which calculates the complexity metric of the instance $\cMC(C^\epsilon,u^*)$. The proof of Lemma \ref{lem:aistats-1} is similar to that of Theorem \ref{thm:max-dist}.
%
%
\begin{lemma}\label{lem:aistats-1}
Suppose that $n\geq m \geq5$, $\alpha\in[0,1]$ and $\epsilon\in[0,1]$. 
The complexity metric $\mathbb{D}_\alpha(C^\epsilon,u^*)$ has the closed form
\begin{align*}
    [\mathbb{D}_\alpha(C^\epsilon,u^*)]^{-1} = \min\bigg\{ \frac{2\alpha}{Z_\epsilon} + \frac{2(1-\alpha)(m-1)}{m},\frac{4\alpha \epsilon}{Z_\epsilon} +& \frac{2(1-\alpha)(m-2)}{m},\\
    &\hspace{3em}\frac{4\alpha (m-1)\epsilon}{Z_\epsilon}  \bigg\}.
\end{align*}
Moreover, $\mathbb{D}_\alpha(C^\epsilon,u^*)$ is strictly decreasing in $\epsilon$ on $[0,1/2 ]$.
\end{lemma}
\begin{proof}
We fix $\epsilon$, $\alpha$ and $m$ in the proof.
Let $\cMC(\tilde{C}, \tilde{u}^*)$ be an instance that attains the minimum in \eqref{eqn:metric-new} and $\ell:=|\cI_1(\tilde{C},\tilde{u}^*)|$.
%
%
Denote 
\[ d_\alpha := \alpha\|C - \tilde{C}\|_1 + (1-\alpha)\|u^* - \tilde{u}^*\|_1. \]
Then, we investigate three different cases.

\paragraph{Case I.} Suppose that $\mathbb{G}_1(\tilde{C},\tilde{u}^*)$ is disconnected. In this case, at least $2(\ell-1)$ non-diagonal entries of $\tilde{C}$ are equal to $0$. This implies that
\begin{align}\label{eqn:aistats-2}
    \|C^\epsilon - \tilde{C}\|_1 \geq 4(\ell-1) \cdot ({\epsilon}/{Z_\epsilon}).
\end{align}

\paragraph{Case II.} The case when $\cI_{00}(\tilde{C},\tilde{u}^*)$ is non-empty can be analyzed similarly as \textit{Case I} and the inequality \eqref{eqn:aistats-2} holds. We omit the proof for brevity.

\paragraph{Case III.} Finally, we consider the case when $\mathbb{G}_1(\tilde{C},\tilde{u}^*)$ is bipartite. If $\ell \geq 5$, at least $2(\ell-1)$ non-diagonal entries of $\tilde{C}$ are equal to $0$ and inequality \eqref{eqn:aistats-2} holds. If $\ell=4$, at least $4$ non-diagonal entries and $4$ diagonal entries of $\tilde{C}$ are equal to $0$. Hence, we have
\begin{align}\label{eqn:aistats-3}
    \|C^\epsilon - \tilde{C}\|_1 \geq 8 \cdot \frac{\epsilon}{Z_\epsilon} + 8\cdot \frac{1}{Z_\epsilon} = \frac{8\epsilon+8}{Z_\epsilon} \geq \frac{12\epsilon}{Z_\epsilon}.
\end{align}
Similarly, it follows from analyzing the cases with $\ell=1,2,3$ that
\begin{align}\label{eqn:aistats-4}
    \|C^\epsilon - \tilde{C}\|_1 &\geq (4\epsilon+6)/(Z_\epsilon) \geq 8\epsilon/Z_\epsilon,\\
    \nonumber\|C^\epsilon - \tilde{C}\|_1 &\geq {4}/{Z_\epsilon} \geq 4\epsilon/Z_\epsilon,\\
    \nonumber\|C^\epsilon - \tilde{C}\|_1 &\geq {2}/{Z_\epsilon}.
\end{align}

Combining inequalities \eqref{eqn:aistats-2}, \eqref{eqn:aistats-3} and \eqref{eqn:aistats-4}, we know that
\begin{align}\label{eqn:aistats-5}
    \|C^\epsilon - \tilde{C}\|_1 \geq {N(\ell)}/{Z_\epsilon},
\end{align}
where $N(\ell) := 4(\ell-1)\epsilon$ if $\ell\geq 2$ and $N(1) := 2$.

Now, we consider the optimal choice of $\tilde{u}^*$. Since the distance in \eqref{eqn:aistats-5} is increasing in $\ell$, it is not optimal to choose $\ell > m$. For every $\ell\in[m]$, at least $m-\ell$ of the first $m$ entries of $\tilde{u}^*$ are $0$. Hence, we have the lower bound
\begin{align}\label{eqn:aistats-6}
    \|u^* - \tilde{u}^*\|_1 \geq 2(m-\ell) \cdot m^{-1}.
\end{align}
Combining inequalities \eqref{eqn:aistats-5} and \eqref{eqn:aistats-6}, we have
\begin{align*}
    d_\alpha \geq \frac{N(\ell)\cdot\alpha}{Z_\epsilon} + \frac{2(1-\alpha)(m-\ell)} {m}.
\end{align*}
Taking the minimum over $\ell\in[m]$ leads to
\begin{align*}
    d_\alpha \geq \min_{\ell\in[m]} \left[ \frac{N(\ell)\cdot\alpha}{Z_\epsilon} + \frac{2(1-\alpha)(m-\ell)} {m} \right].
\end{align*}
We note that the above inequality indeed attains equality with a suitable choice of $\tilde{C}$ and $\tilde{u}^*$. For all $\ell \geq 2$, we can set $\tilde{u}_i^*=0$ for all $i\in\{\ell+1,m\}$ and make node $1$ disconnected from nodes $\{2,\dots,\ell\}$. If $\ell=1$, we can remove the self-loop at node $1$. Therefore, it holds that
\begin{align*}
    d_\alpha = \min_{\ell\in[m]} \left[ \frac{\alpha N(\ell)}{Z_\epsilon} + \frac{2(1-\alpha)(m-\ell)} {m} \right].
\end{align*}
The minimum in the above equality is attained at one of the points $1,2,m$, which gives
\begin{align*}
    d_\alpha = \min\left\{ \frac{2\alpha}{Z_\epsilon} + \frac{2(1-\alpha)(m-1)}{m},\frac{4\alpha \epsilon}{Z_\epsilon} + \frac{2(1-\alpha)(m-2)}{m}, \frac{4\alpha (m-1)\epsilon}{Z_\epsilon}  \right\}.
\end{align*}
Since each component in the minimization is an increasing function in $\epsilon$, the distance $d_\alpha$ is also increasing in $\epsilon$. Results for $\mathbb{D}_\alpha(C^\epsilon,u^*)$ follow accordingly by taking the inverse of $d_\alpha$. 

Since the closed form expression of $\mathbb{D}_\alpha(C^\epsilon,u^*)$ is the minimum of three monotone functions in $\epsilon$, the complexity metric is the negative of a unimodal function. For every $\epsilon \leq 1/(2m)$, we can prove that
\[ \frac{2\alpha}{Z_\epsilon} + \frac{2(1-\alpha)(m-1)}{m} > \min\left\{\frac{4\alpha\epsilon}{Z_\epsilon} + \frac{2(1-\alpha)(m-2)}{m}, \frac{4\alpha(m-1)\epsilon}{Z_\epsilon} \right\}. \]
Therefore, in the regime $[0,1/2]$, the complexity metric $\mathbb{D}_\alpha(C^\epsilon,u^*)$ is the minimum of two strictly decreasing functions and, thus, is also strictly decreasing in $\epsilon$.
\qed\end{proof}
Combining \revise{Theorem \ref{thm:upper-lower}} and Lemma \ref{lem:aistats-1}, we are able to estimate the range of the complexity metric.
\begin{proof}[Proof of Theorem \ref{thm:aistats-1}]
By defining constants $\delta:=1/26$ and $\Delta:=18$, \revise{Theorem \ref{thm:upper-lower} implies} that
\begin{enumerate}
    \item If $\epsilon < \delta / m$, the instance $\cMC(C^\epsilon, u^*)$ has spurious local minima;
    \item If $\epsilon > \Delta / m$, the instance $\cMC(C^\epsilon, u^*)$ has no spurious local minima.
\end{enumerate}
Then, we study two different cases.

\paragraph{Case I.} We first consider the case when $m \epsilon$ is large. Since $\epsilon < \Delta/m \leq 1/2$, the threshold is located in the regime where $\mathbb{D}_\alpha(C^\epsilon,u^*)$ is strictly decreasing. Hence, it suffices to show that
\[ \left[\frac{2\alpha \Delta}{n^2} + \min\left\{ 4\alpha \Delta \cdot \frac{ m}{n^2}, 2(1-\alpha) \right\}\right]^{-1}  \]
is a lower bound on $\mathbb{D}_\alpha(C^\epsilon,u^*)$ when $\epsilon = \Delta / m$. 
By Lemma \ref{lem:aistats-1}, it holds that
\begin{align*}
    &\left[\mathbb{D}_\alpha(C^\epsilon,u^*)\right]^{-1}\\
    = &\min\left\{ \frac{2\alpha}{Z_\epsilon} + \frac{2(1-\alpha)(m-1)}{m},\frac{4\alpha \epsilon}{Z_\epsilon} + \frac{2(1-\alpha)(m-2)}{m}, \frac{4\alpha (m-1)\epsilon}{Z_\epsilon}  \right\}\\
    \leq &\min\left\{ \frac{4\alpha \epsilon}{Z_\epsilon} + \frac{2(1-\alpha)(m-2)}{m}, \frac{4\alpha (m-1)\epsilon}{Z_\epsilon}  \right\}\\
    = & \frac{4\alpha \epsilon}{Z_\epsilon} + (m-2)\min\left\{ \frac{4\alpha \epsilon}{Z_\epsilon}, \frac{2(1-\alpha)}{m} \right\} \leq \frac{4\alpha \epsilon}{Z_\epsilon} + m\min\left\{ \frac{4\alpha \epsilon}{Z_\epsilon}, \frac{2(1-\alpha)}{m} \right\}.
\end{align*}
%
Since the graph $\mathbb{G}$ does not contain any independence set with $m+1$ nodes, Tur\'{a}n's theorem \cite{aigner1995turan} implies that the graph $\mathbb{G}$ has at least $n^2 / (2m)$ edges, namely, 
\[ |\mathbb{E}| \geq {n^2}/(2m). \]
We note that the above bound is asymptotically tight and is attained by the Tur\'{a}n graph. Hence, we obtain that 
\[ Z_\epsilon=2|\mathbb{E}|+n+m(m-1)\epsilon \geq 2|\mathbb{E}| \geq n^2 / m. \]
By substituting into the estimate of $\mathbb{D}_\alpha(C^\epsilon,u^*)$, it follows that
\begin{align*}
\left[\mathbb{D}_\alpha(C^\epsilon,u^*)\right]^{-1} &\leq \frac{4\alpha \epsilon \cdot m}{n^2} + m\min\left\{ \frac{4\alpha \epsilon \cdot m}{n^2}, \frac{2(1-\alpha)}{m} \right\}\\
&= \frac{2\alpha \Delta}{n^2} + \min\left\{ 4\alpha \Delta \cdot \frac{ m}{n^2}, 2(1-\alpha) \right\}.
\end{align*}
%

\paragraph{Case II.} Next, we consider the case when $\epsilon m$ is small. Similar to \textit{Case I}, it suffices to show that 
\[ \frac{18}{17} \max\left\{ \frac{n^2}{4\alpha \delta}, \frac{1}{2(1-\alpha)} \right\} \]
is an upper bound for $\mathbb{D}_\alpha(C^\epsilon,u^*)$ when $\epsilon = \delta / m$. Since $\delta < 1/2$, we have
\[ {2\alpha}/{Z_\epsilon} > {4\alpha \epsilon}/{Z_\epsilon}. \]
By Lemma \ref{lem:aistats-1}, it holds that
\begin{align*}
    &\left[\mathbb{D}_\alpha(C^\epsilon,u^*)\right]^{-1}\\
    = &\min\left\{ \frac{2\alpha}{Z_\epsilon} + \frac{2(1-\alpha)(m-1)}{m},\frac{4\alpha \epsilon}{Z_\epsilon} + \frac{2(1-\alpha)(m-2)}{m}, \frac{4\alpha (m-1)\epsilon}{Z_\epsilon}  \right\}\\
    = &\min\left\{ \frac{4\alpha \epsilon}{Z_\epsilon} + \frac{2(1-\alpha)(m-2)}{m}, \frac{4\alpha (m-1)\epsilon}{Z_\epsilon}  \right\}\\
    = & \frac{4\alpha \epsilon}{Z_\epsilon} + (m-2)\min\left\{ \frac{4\alpha \epsilon}{Z_\epsilon}, \frac{2(1-\alpha)}{m} \right\} \geq \frac{17}{18} \min\left\{ \frac{4\alpha \epsilon m}{Z_\epsilon}, 2(1-\alpha) \right\},
\end{align*}
where the last inequality is from $m \geq 36$. Since $\epsilon\leq 1$, the definition of $Z_\epsilon$ implies that $Z_\epsilon \leq n^2$. By substituting into the estimate of $\mathbb{D}_\alpha(C^\epsilon,u^*)$, it follows that
\begin{align*}
    \left[\mathbb{D}_\alpha(C^\epsilon,u^*)\right]^{-1} &\geq \frac{17}{18} \min\left\{ \frac{4\alpha \epsilon m}{n^2}, 2(1-\alpha) \right\} = \frac{17}{18} \min\left\{ \frac{4\alpha \delta}{n^2}, 2(1-\alpha) \right\}.
\end{align*}

By combining \textit{Cases I} and \textit{II}, we complete the proof.
\qed\end{proof}

\section{Proofs in Section \ref{sec:theory}}

\subsection{Proof of Lemma \ref{lem:ssp}}
\label{apd:ssp}
\begin{proof}
Without loss of generality, we assume that
\[ u^0_i = 1 / n,\quad \forall i\in[n]. \]
We first consider the scaled problem instance
\begin{align}\label{eqn:ssp-5} \min_{x\in\mathbb{R}^n} {\sum}_{i,j\in[n], i\neq j} (x_ix_j - 1)^2. \end{align}
We denote the gradient and the Hessian matrix of problem \eqref{eqn:ssp-5} as $g(x)\in\Rn$ and $H(x)\in\R^{n\times n}$, respectively. Then, we can calculate that
\begin{align*}
\frac14 g_i(x) &= -x_i^3 + (\|x\|_2^2 + 1) x_i - {\sum}_{k\in[n]} x_k,\quad\forall i\in[n];\\
\frac14 H_{ii}(x) &= {\sum}_{k\in[n],k\neq i} x_k^2,\quad \frac14 H_{ij}(x) = 2x_ix_j - 1,\quad \forall i,j\in[n].
\end{align*}
Let $c$ be a small positive constant and define $\epsilon:=c/n$. Suppose that $x\in\Rn$ satisfies
\begin{align}\label{eqn:ssp-2} \|g(x)\|_\infty < 4\epsilon. \end{align}
Then, we study three different cases.

\paragraph{Case I.} We first consider the case when $\sum_{i\in[n]} x_i > 2\epsilon$. For all $i\in[n]$, the condition \eqref{eqn:ssp-2} implies that
\begin{align}\label{eqn:ssp-1} 
\frac14|g_i(x)| = \left| \left({\sum}_{j\in[n], j\neq i} x_j^2 + 1 \right)x_i - {\sum}_{j\in[n]} x_j \right| < \epsilon. 
\end{align}
If $x_i \leq \epsilon$, it holds that
\[ \left({\sum}_{j\in[n],j\neq i}x_j^2 + 1 \right)x_i - {\sum}_{j\in[n]} x_j \leq x_i - {\sum}_{j\in[n]} x_j < -\epsilon, \]
which contradicts \eqref{eqn:ssp-1}. Hence,
\[ x_i > \epsilon ,\quad \forall i\in[n]. \]
Define three index sets
\begin{align*} 
\mathcal{I}_1 &:= \{ i\in[n] ~|~ x_i \geq 1 + \epsilon \},~ \mathcal{I}_2 := \{ i\in[n] ~|~ x_i \leq 1 - \epsilon \},~ \mathcal{I}_3 := [n] \backslash (\cI_1 \cup \cI_2). 
\end{align*}
Choosing the perturbation direction $q\in\mathbb{R}^n$ to be
\[ q_i = -x_i,\quad \forall i\in\mathcal{I}_1;\quad q_i = x_i,\quad \forall i\in\mathcal{I}_2;\quad q_i = 0,\quad \forall i\in\mathcal{I}_3, \]
we can calculate that
\begin{align}\label{eqn:ssp-3}
\frac14 q^T g(x) &= \sum_{i,j\in\mathcal{I}_1,i\neq j} -x_ix_j(x_ix_j - 1) + \sum_{i,j\in\mathcal{I}_2,i\neq j} x_ix_j(x_ix_j - 1)\\
\nonumber&\quad  + \sum_{i\in\mathcal{I}_1,j\in\mathcal{I}_3} -x_ix_j(x_ix_j - 1) + \sum_{i\in\mathcal{I}_2,j\in\mathcal{I}_3} x_ix_j(x_ix_j - 1).
\end{align}
We then consider four sub-cases.
\paragraph{Case I-1.} We first assume that $|\mathcal{I}_1| \geq 2$. In this case, we have
\begin{align}\label{eqn:ssp-4}
\sum_{i,j\in\mathcal{I}_1,i\neq j} -x_ix_j(x_ix_j - 1) &\leq \sum_{i,j\in\mathcal{I}_1,i\neq j} -x_ix_j[(1+\epsilon)^2 - 1] \leq -2\epsilon \sum_{i,j\in\mathcal{I}_1,i\neq j} x_ix_j\\
\nonumber&\leq -2\epsilon (|\mathcal{I}_1|-1)\|x_{\mathcal{I}_1}\|_1 \leq -2 \|x_{\mathcal{I}_1}\|_1 \cdot \epsilon,\\
\nonumber\sum_{i,j\in\mathcal{I}_2,i\neq j} x_ix_j(x_ix_j - 1) & = \sum_{i,j\in\mathcal{I}_2,i\neq j} -x_i\cdot [ x_j - x_ix_j^2 ]\\
\nonumber& \leq \sum_{i,j\in\mathcal{I}_2,i\neq j} x_i\cdot [ x_j - (1-\epsilon) x_j^2 ]\\
\nonumber&\leq \sum_{i\in\mathcal{I}_2} x_i \max(|\mathcal{I}_2| - 1, 0) \cdot \epsilon[ 1 - (1-\epsilon)\epsilon ]\\
\nonumber&= -\max(|\mathcal{I}_2| - 1, 0) \|x_{\mathcal{I}_2}\|_1  \cdot \epsilon + O(n\epsilon^2) ,\\
\nonumber\sum_{i\in\mathcal{I}_1,j\in\mathcal{I}_3} -x_ix_j(x_ix_j - 1) &= \sum_{i\in\mathcal{I}_1,j\in\mathcal{I}_3}-\frac14 (2x_ix_j - 1)^2 + \frac14\\
\nonumber& \leq |\mathcal{I}_1||\mathcal{I}_3| \left[-\frac14 [ 2(1+\epsilon)(1-\epsilon) - 1 ]^2 + \frac14\right],\\
\nonumber&= |\mathcal{I}_1||\mathcal{I}_3| \left(\epsilon^2 - \epsilon^4\right) = O(n^2\epsilon^2),\\
\nonumber\sum_{i\in\mathcal{I}_2,j\in\mathcal{I}_3} x_ix_j(x_ix_j - 1) &= \sum_{i\in\mathcal{I}_2,j\in\mathcal{I}_3} \frac14 (2x_ix_j - 1)^2 - \frac14\\
\nonumber & \leq |\mathcal{I}_2||\mathcal{I}_3| \left[\frac14 [2(1+\epsilon)(1-\epsilon) - 1]^2 - \frac14 \right] \leq 0.
\end{align}
%
Choosing $\epsilon$ to be small enough and substituting the above four estimates into \eqref{eqn:ssp-3}, we obtain that
\begin{align*} 
\frac14 q^T g(x) &\leq -2\epsilon \|x_{\mathcal{I}_1}\|_1 -\max(|\mathcal{I}_2| - 1, 0) \epsilon \|x_{\mathcal{I}_2}\|_1 + O(n^2\epsilon^2)\\
&\leq - \left[\|x_{\mathcal{I}_1}\|_1 + \max(|\mathcal{I}_2| - 1, 0) \|x_{\mathcal{I}_2}\|_1 \right] \cdot \epsilon. 
\end{align*}
If $|\mathcal{I}_2| \geq 2$, it follows from H{\"o}lder's inequality that
\[ \|g(x)\|_\infty \geq \frac{4(\|x_{\mathcal{I}_1}\|_1 + \|x_{\mathcal{I}_2}\|_1) \cdot \epsilon}{\|q\|_1} = \frac{\|x_{\mathcal{I}_1}\|_1 + \|x_{\mathcal{I}_2}\|_1}{\|x_{\mathcal{I}_1}\|_1 + \|x_{\mathcal{I}_2}\|_1} \cdot 4\epsilon = 4\epsilon. \]
which is a contradiction to \eqref{eqn:ssp-2}. Otherwise if $|\mathcal{I}_2|\leq 1$, it also follows from H{\"o}lder's inequality that
\[ \|g(x)\|_\infty \geq \frac{4\|x_{\mathcal{I}_1}\|_1 \epsilon}{\|q\|_1} = \frac{4\|x_{\mathcal{I}_1}\|_1}{\|x_{\mathcal{I}_1}\|_1 + \|x_{\mathcal{I}_2}\|_1} \cdot \epsilon \geq \frac{\|x_{\mathcal{I}_1}\|_1}{\|x_{\mathcal{I}_1}\|_1 + 1} \cdot 4\epsilon \geq 2\epsilon. \]
In summary, in this sub-case, we have
\[ \|g(x)\|_\infty \geq {2\epsilon}. \]

\paragraph{Case I-2.} Now, we consider the case when $|\mathcal{I}_1| = 1$ and $|\mathcal{I}_2| \geq 2$. Assume without loss of generality that $\mathcal{I}_1 = \{1\}$. A similar calculation as \eqref{eqn:ssp-4} leads to
\[ \frac14 q^T g(x) \leq -\max(|\mathcal{I}_2| - 1, 0) \epsilon \|x_{\mathcal{I}_2}\|_1 + O(n^2\epsilon^2) \leq -\frac12 \|x_{\mathcal{I}_2}\|_1 \cdot \epsilon. \]
If $x_1 \leq 2\epsilon^{-1}$, H\"{o}lder's inequality gives
\[ \|g(x)\|_\infty \geq \frac{4\epsilon\|x_{\mathcal{I}_2}\|_1}{2\|q\|_1} = 2\epsilon \cdot \frac{\|x_{\mathcal{I}_2}\|_1}{\|x_{\mathcal{I}_1}\|_1 + \|x_{\mathcal{I}_2}\|_1} \geq 2\epsilon \cdot \frac{2\epsilon}{2\epsilon^{-1} + 2\epsilon} \geq 2\epsilon \cdot \frac{\epsilon^2}{2} = \epsilon^3. \]
Now, we assume that $x_1 > 2\epsilon^{-1}$. The first component of the gradient is
\begin{align*}
\frac14 g_1(x) &= {\sum}_{j\in[n],j\neq 1}( x_j^2 x_i - x_j ) \geq {\sum}_{j\in[n],j\neq 1}( \epsilon^2 x_i - \epsilon )\\
&= (n-1)\epsilon^2 \cdot x_1 - (n-1)\epsilon > (n-1)\epsilon > \epsilon,
\end{align*}
%
which contradicts \eqref{eqn:ssp-2}. 
In summary, in this sub-case, we have
\[ \|g(x)\|_\infty \geq {\epsilon^3}. \]

\paragraph{Case I-3.} In this case, we assume $|\mathcal{I}_1| = 1$ and $|\mathcal{I}_2| \leq 1$. In addition, we assume $\mathcal{I}_1 = \{1\}$. If $x_1 \geq (1-\epsilon)^{-1} + \epsilon$, the third estimate in \eqref{eqn:ssp-4} becomes
\begin{align*}
{\sum}_{j\in\mathcal{I}_3} -x_1x_j(x_1x_j - 1) &\leq {\sum}_{j\in\mathcal{I}_3} - x_1 (1-\epsilon)[ x_1(1-\epsilon) - 1 ]\\
&\leq -(1-\epsilon)^2\epsilon x_1 \leq -\frac12 \|x_{\mathcal{I}_1}\|_1 \cdot \epsilon.
\end{align*}
Then, using a similar analysis and by applying H\"{o}lder's inequality, it follows that
\[ \frac14 q^T g(x) \leq -\frac12 \|x_{\mathcal{I}_1}\|_1 \epsilon \quad \text{ and } \quad  \|g(x)\|_\infty \geq 2\epsilon \cdot \frac{\|x_{\mathcal{I}_1}\|_1}{\|x_{\mathcal{I}_1}\|_1 + \|x_{\mathcal{I}_2}\|_1} > \epsilon. \]
%
Otherwise, if $x_1 < (1-\epsilon)^{-1} + \epsilon$,
\[ |x_1 - 1| < \frac{\epsilon}{1-\epsilon} + \epsilon < 3\epsilon. \]
Hence,
\[ \|x - x^0\|_1 \leq 3\epsilon + (n-1)\epsilon = (n+2)\epsilon. \]
In summary, in this sub-case, we have
\[ \|g(x)\|_\infty < {\epsilon}/{4}\quad \text{or} \quad \|x - x^0\|_1 \leq (n+2)\epsilon. \]

\paragraph{Case I-4.} Finally, we assume $|\mathcal{I}_1| = 0$. If $|\mathcal{I}_2|\geq 2$, we can use a similar analysis as \textit{Case I-2} to conclude that 
\[ \frac14 q^T g(x) \leq -\frac12 \|x_{\mathcal{I}_2}\|_1 \cdot \epsilon + O(n\epsilon^2) \]
and thus 
\[ \|g(x)\|_\infty \geq {\epsilon}. \]
Next, we consider the case when $|\mathcal{I}_2| = 1$ and we assume $\mathcal{I}_2 = \{1\}$. The fourth term in \eqref{eqn:ssp-4} can be estimated as
\begin{align*}
{\sum}_{i\in\mathcal{I}_2,j\in\mathcal{I}_3} x_ix_j(x_ix_j - 1) &= \left({\sum}_{j=2}^n x_j^2\right) x_1^2 - \left({\sum}_{j=2}^n x_j\right) x_1.
\end{align*}
Since $x_j \in [1-\epsilon,1+\epsilon]$ for all $j\in\{2,\dots,n\}$, it holds that
\[ \frac{\sum_{j=2}^n x_j}{\sum_{j=2}^n x_j^2} \geq \frac{1}{1+\epsilon} > 1-\epsilon. \]
Therefore, 
\begin{align*}
{\sum}_{i\in\mathcal{I}_2,j\in\mathcal{I}_3} x_ix_j(x_ix_j - 1) &= \left({\sum}_{j=2}^n x_j^2\right) x_1^2 - \left({\sum}_{j=2}^n x_j\right) x_1\\
& \leq \left({\sum}_{j=2}^n x_j^2\right) (1-\epsilon)^2 - \left({\sum}_{j=2}^n x_j\right) (1-\epsilon)\\
&={\sum}_{j=2}^n \left[ (1-\epsilon)^2 x_j^2 - (1-\epsilon) x_j \right]\\
& \leq {\sum}_{j=2}^n \left[ (1-\epsilon)^2 (1+\epsilon)^2 - (1-\epsilon)(1+\epsilon) \right]\\
&\leq -(n-1)\epsilon^2 + O(n\epsilon^3).
\end{align*}
Thus, it holds that
\[ \frac14 q^T g(x) \leq -(n-1)\epsilon^2 + O(n\epsilon^3) \geq -\epsilon^2. \]
H\"{o}lder's inequality implies that
\[ \|g(x)\|_\infty \geq \frac{4\epsilon^2}{\|q\|_1} = \frac{4\epsilon^2}{x_1} \geq \frac{4\epsilon^2}{1-\epsilon} \geq 4\epsilon^2. \]
The only remaining case is when $|\mathcal{I}_2| = 0$. In this case, we have
\[ x_i \in [1-\epsilon, 1+\epsilon],\quad \forall i\in[n]. \]
Therefore, it holds that
\[ \|x - x^0\|_1 \leq n\epsilon. \]
In summary, in this sub-case, we have
\[ \|g(x)\|_\infty \geq 4\epsilon^2\quad \text{or} \quad \|x - x^0\|_1 \leq n\epsilon. \]

Combining \textit{Cases I-1} to \textit{I-4} yields that 
\[ \|g(x)\|_\infty \geq {\epsilon^3} \quad \text{or} \quad \|x - x^0\|_1 \leq (n+4)\epsilon \]
in \textit{Case I}.

\paragraph{Case II.} For the case when $\sum_{i\in[n]} x_i < -2\epsilon$, one can obtain the same conclusions as \textit{Case I} by the symmetry of the landscape.

\paragraph{Case III.} We finally consider the case when $\sum_{i\in[n]} x_i \in [-2\epsilon,2\epsilon]$. Considering the assumption \eqref{eqn:ssp-2}, we have
\[ \frac14 g_i(x) = \left( {\sum}_{j\in[n],j\neq i} x_j^2 + 1 \right) x_i - {\sum}_{j\in[n]} x_j \in [-\epsilon,\epsilon],\quad \forall i\in[n]. \]
Combined with the assumption that $\sum_{i\in[n]} x_i \in [-2\epsilon,2\epsilon]$, it follows that
\[ \left( {\sum}_{j\in[n],j\neq i} x_j^2 + 1 \right) x_i \in [-3\epsilon,3\epsilon]. \]
Furthermore, since $\sum_{j\in[n],j\neq i} x_j^2 + 1 \geq 1$, we have
\[ x_i \in [-3\epsilon, 3\epsilon],\quad \forall i\in[n]. \]
We consider the descent direction $p\in\mathbb{R}^n$, where
\[ p_i = {1}/{\sqrt{n}},\quad \forall i\in[n]. \]
Then, we can calculate that
\begin{align*}
\frac14 p^T H(x) p &= {\sum}_{i, j\in[n] j\neq i} \left[x_j^2 p_i^2 + (2x_ix_j - 1) p_ip_j \right]\\
& = \frac{1}{n} {\sum}_{i, j\in[n] j\neq i} \left[x_j^2 + (2x_ix_j - 1) \right]\\
&= \frac{1}{n} \left[ (n-1) {\sum}_{i\in[n]} x_i^2 + 2{\sum}_{i,j\in[n], i\neq j} x_ix_j - n(n-1)\right]\\
&\leq \frac{1}{n} \left[ (n-1) \cdot 9n\epsilon^2 + 2n(n-1) \cdot 9\epsilon^2 - n(n-1)\right]\\
&= 27(n-1)\epsilon^2 - (n-1) \leq -{n}/{2},
\end{align*}
where the last inequality is because $\epsilon$ is sufficiently small.

Combined \textit{Cases I-III}, we have proved that under assumption \eqref{eqn:ssp-2}, it holds that
\[ \min\{\|x - x^0\|_1,\|x + x^0\|_1\} \leq (n+4)\epsilon \quad \text{or} \quad \|g(x)\|_\infty \geq {\epsilon^3} \quad \text{or} \quad \lambda_{min}[H(x)] \leq -2n. \]
Letting $\epsilon := \eta / (n+4) \ll 1$, we know that the property stated in the theorem holds for problem \eqref{eqn:ssp-5} with
\[ \beta(\eta) = \frac{\eta^3}{(n+4)^3},\quad \gamma(\eta) = 2n. \]
In addition, we have $\eta_0 = O(1)$, $\beta(\eta) = O(n^{-3}\eta^3)$ and $\gamma(\eta) = O(n)$. 
Transforming back to the instance $(C^0, u^0)$, the property stated in the theorem holds with
\[ \eta_0 = O(n^{-0.5}),\quad \beta(\eta) = O(n^{-6.5}\eta^3),\quad \gamma(\eta) = O(n^{-2}). \]
This completes the proof.
\qed\end{proof}

\subsection{Proof of Lemma \ref{lem:pert}}
\label{apd:pert}
\begin{proof}
Similar to Lemma \ref{lem:ssp}, it is equivalent to prove the results for the scaled instance $\cMC( n(n-1)\tilde{C}, n\tilde{u}^*)$. With a little abuse of notations, we use $(\tilde{C}\tilde{u}^*)$ to denote the scaled pair of parameters.
Denote
\[ \delta := \max\left\{ \frac{n(n-1)\epsilon}{\alpha^*}, \frac{n \epsilon}{1-\alpha^*} \right\}. \]
Then, the condition stated in the lemma implies that
\[ \tilde{C}_{ij} \in [1 -\delta, 1+\delta],\quad \forall i,j\in[n] \quad \st\quad i\neq j;\quad \tilde{C}_{ii} \in[0,\delta],\quad \tilde{u}^*_i \in [1-\delta,1+\delta],\quad \forall i\in[n]. \]
Let $R > 0$ be a large enough constant. Suppose that $u\in\mathbb{R}^n$ is a stationary point of the instance $(\tilde{C}, \tilde{u}^*)$ such that $\|u\|_2 = R$. Denote the gradient and the Hessian matrix of the instance $\cMC(\tilde{C},\tilde{u}^*)$ at $u$ as $g(u)\in\Rn$ and $H(u)\in\R^{n\times n}$, respectively. Then, it holds that
\begin{align}\label{eqn:pert-1} 
\frac14 g_i(u) = {\sum}_{j\in[n]} \tilde{C}_{ij}u_j( u_iu_j - \tilde{u}_i^*\tilde{u}_j^* ) = 0,\quad \forall i\in[n]. 
\end{align}
We assume without loss of generality that 
\[ u_1 = {\max}_{i\in[n]} |u_i| \geq R / \sqrt{n} > 0. \]
If $u_i = 0$ for all $i\in[n]\backslash\{1\}$, we have
\begin{align*}
\frac14 g_2(u) &= \left( \tilde{C}_{21} u_1^2 + {\sum}_{j\geq 2} \tilde{C}_{2j}u_j^2 \right) u_2 - \left( \tilde{C}_{21} \tilde{u}_1^*\tilde{u}_2^*u_1 + {\sum}_{j\geq 2} \tilde{C}_{2j} \tilde{u}_j^*\tilde{u}_2^*u_j \right)\\
&= - \tilde{C}_{21} \tilde{u}_1^*\tilde{u}_2^*u_1 \leq -(1-\delta)\cdot (1-\delta)^2\cdot R < 0,
\end{align*}
where the last inequality is in light of $\tilde{C}_{21} > 1-\delta$ and $\tilde{u}^*_i > 1-\delta$. This contradicts the stationarity of point $x$ and thus
\[ {\sum}_{j\geq 2} u_j^2 > 0. \]
Moreover, since $\tilde{C}_{1j} > 1-\delta$ for all $j\in[n]\backslash\{1\}$, we have
\[ {\sum}_{j\in[n]} \tilde{C}_{1j} u_j^2 \geq {\sum}_{j\geq 2} \tilde{C}_{1j} u_j^2 > (1-\delta){\sum}_{j\geq 2} u_j^2 > 0. \]
Similarly, for all $i\in[n]\backslash\{1\}$, it holds that
\[ {\sum}_{j\in[n]} \tilde{C}_{ij} u_j^2 \geq {\sum}_{j\in[n],j\neq i} \tilde{C}_{ij} u_j^2 > (1-\delta){\sum}_{j\in[n],j\neq i} u_j^2 > (1-\delta) u_i^2 > 0. \]
Solving \eqref{eqn:pert-1} for all $i\in[n]$, we conclude that
\begin{align}\label{eqn:pert-4} u_i = \frac{\sum_{j\in[n]} \tilde{C}_{ij}\tilde{u}_i^* \tilde{u}_j^* u_j }{\sum_{j\in[n]} \tilde{C}_{ij}{u}_j^2 }. \end{align}
Assuming that 
\[ u_1 < R - \frac{2n}{(1-\delta)R}, \]
it follows that
\begin{align}\label{eqn:pert-2} 
{\sum}_{j\in [n]} \tilde{C}_{1j} u_j^2 > (1-\delta){\sum}_{j\geq 2} u_j^2 \geq 4n - \frac{4n^2}{(1-\delta)R^2}. 
\end{align}
In addition, we can calculate that
\begin{align}\label{eqn:pert-3}
{\sum}_{j\in [n]} \tilde{C}_{1j}\tilde{u}_1^* \tilde{u}_j^* u_j &\leq {\sum}_{j\in [n]} \tilde{C}_{1j}\tilde{u}_1^* \tilde{u}_j^* |u_j|\\
\nonumber&< (1+\delta)\cdot(1+\delta)^2 {\sum}_{j\in [n]}|u_j| \leq 2\|u\|_1 \leq 2\sqrt{n}R,
\end{align}
where the second last inequality is because $\delta$ is a sufficiently small constant.
Combining inequalities \eqref{eqn:pert-2}-\eqref{eqn:pert-3}, we have
\[ u_1 = \frac{\sum_{j\in [n]} \tilde{C}_{1j}\tilde{u}_1^* \tilde{u}_j^* u_j }{\sum_{j\in [n]} \tilde{C}_{1j}{u}_j^2 } < \frac{2\sqrt{n}R}{4n - 4n^2 / [(1-\delta)^2R^2]}. \]
Choosing $R \geq 4n \geq 2(1-\delta)^{-1}n$, the above inequality leads to
\[ u_1 < \frac{2\sqrt{n}R}{4n - 4n^2 / [(1-\delta)^2R^2]} < \frac{2\sqrt{n}{R}}{2n} = \frac{R}{\sqrt{n}}, \]
which contradicts the assumption that $u_1 \geq R / \sqrt{n}$. 
Therefore,
\[ u_1 \geq R - \frac{2n}{(1-\delta)R}. \]
Using the condition that $\|x\|_2 = R$, it holds that
\[ {\sum}_{j\geq 2} u_j^2 \leq \frac{2n}{1-\delta} - \frac{4n^2}{(1-\delta)^2R^2} < \frac{2n}{1-\delta}. \]
For all $i\in[n]\backslash\{1\}$, the relation \eqref{eqn:pert-4} implies that
\begin{align*}
u_i &= \frac{\sum_{j\in [n]} \tilde{C}_{ij}\tilde{u}_i^* \tilde{u}_j^* u_j }{\sum_{j\in [n]} \tilde{C}_{ij}{u}_j^2 } = \frac{\tilde{C}_{1i}\tilde{u}_i^*\tilde{u}_1^*u_1 + \sum_{j\geq 2} \tilde{C}_{ij}\tilde{u}_i^* \tilde{u}_j^* u_j}{\sum_{j\in [n]} \tilde{C}_{ij}{u}_j^2 }\\
&\geq \frac{ (1-\delta)\cdot(1-\delta)^2  (R - {2n}/[(1-\delta)R]) - (1+\delta)\cdot(1+\delta)^2 \sqrt{ n \cdot \sum_{j\geq 2} u_j^2 } }{\sum_{j\in [n]} \tilde{C}_{ij}{u}_j^2}\\
&\geq \frac{ (1-\delta)\cdot(1-\delta)^2 (R - {2n}/[(1-\delta)R]) - (1+\delta)\cdot(1+\delta)^2 \sqrt{ n \cdot 2n / (1-\delta) } }{\sum_{j\in [n]} \tilde{C}_{ij}{u}_j^2}\\
&\geq \frac{ 1/2 \cdot (R - 1) - 2n \sqrt{ 2 / (1-\delta) } }{\sum_{j\in [n]} \tilde{C}_{ij}{u}_j^2} \geq \frac{ R/2 - 1/2 - 4n }{\sum_{j\in [n]} \tilde{C}_{ij}{u}_j^2} > 0,
\end{align*}
where the last inequality is due to choosing $R > 8n + 1$ and the second last inequality results from the fact that $\delta$ is sufficiently small. Using the same relation, it follows that
\begin{align*}
u_i &= \frac{\sum_{j\in [n]} \tilde{C}_{ij}\tilde{u}_i^* \tilde{u}_j^* u_j }{\sum_{j\in [n]} \tilde{C}_{ij}{u}_j^2 } \geq \frac{ \tilde{C}_{1i}\tilde{u}_i^* \tilde{u}_1^* u_1 }{\sum_{j\in[n]} \tilde{C}_{ij}{u}_j^2 } \geq \frac{(1-\delta)(1-\delta)^2 u_1}{(1+\delta) \cdot R^2 }\\
& \geq \frac{1}{4R^2} \cdot \left(R - \frac{2n}{(1-\delta)R} \right) \geq \frac{1}{8R},
\end{align*}
where the last inequality is due to choosing $R \geq 8n \geq 4(1-\delta)^{-1}n$. 
Furthermore, using the relation \eqref{eqn:pert-4} with $i=1$, we have
\begin{align*}
{\sum}_{j\geq2}u_j &\geq \frac{1}{(1+\delta)(1+\delta)^2} {\sum}_{j\geq 2} \tilde{C}_{1j}\tilde{u}_i^* \tilde{u}_j^* u_j\\
& = \frac{1}{(1+\delta)(1+\delta)^2}\cdot u_1 \left[{\sum}_{j\geq 2} \tilde{C}_{1j} u_j^2 + \tilde{C}_{11}[u_1^2 - (\tilde{u}_1^*)^2] \right] \\
&\geq \frac{1}{(1+\delta)(1+\delta)^2}\cdot u_1 \bigg[(1-\delta){\sum}_{j\geq 2} u_j^2\\
&\hspace{13em} + \tilde{C}_{11} [(R-2n/[(1-\delta)R])^2 - (1+\delta)^2] \bigg]\\
&\geq \frac{1-\delta}{(1+\delta)(1+\delta)^2} \cdot u_1\left({\sum}_{j\geq 2} u_j^2 \right)\\
& \geq \frac{1-\delta}{(1+\delta)(1+\delta)^2}\left( R - \frac{2n}{(1-\delta)R} \right) \cdot {\sum}_{j\geq 2} u_j^2\\
&\geq \frac14 \left( R - \frac{2n}{(1-\delta)R} \right) \cdot {\sum}_{j\geq 2} u_j^2.
\end{align*}
Since $\sum_{j\geq2}u_j \leq \sqrt{ n(\sum_{j\geq2}u_j^2) }$, it follows that
\begin{align*}
\sqrt{ n \left({\sum}_{j\geq2}u_j^2\right) } \geq \frac14 \left( R - \frac{2n}{(1-\delta)R} \right) \cdot {\sum}_{j\geq 2} u_j^2,
\end{align*}
which further implies that
\begin{align*}
{\sum}_{j\geq 2} u_j^2 \leq \frac{16n}{ ( R - {2n}/[(1-\delta)R] )^2 } \leq \frac{16n}{(R-1)^2} \leq \frac{1}{4},
\end{align*}
where the last inequality is because of choosing $R \geq 1 + 8\sqrt{n}$.
Now, we consider the descent direction $q\in\mathbb{R}^n$, where
\[ q_1 = - u_1;\quad q_i = u_i ,\quad \forall i\in[n]\backslash\{1\}. \]
Similar to the proof of Lemma \ref{lem:ssp}, we can calculate that
\begin{align*}
\frac14 \langle g(u), q \rangle &= {\sum}_{i,j\geq2,i\neq j}\tilde{C}_{ij} u_i u_j \left( u_iu_j  - \tilde{u}_i^*\tilde{u}_j^* \right) - \tilde{C}_{11}u_1^2 [u_1^2 - (\tilde{u}_1^*)^2]\\
&\hspace{16em} + {\sum}_{i\geq 2} \tilde{C}_{ii} u_i^2[u_i^2 - (\tilde{u}_i^*)^2]  \\
&\leq {\sum}_{i,j\geq2,i\neq j}\tilde{C}_{ij} u_i u_j \left( u_iu_j  - \tilde{u}_i^*\tilde{u}_j^* \right) + {\sum}_{i\geq 2} \tilde{C}_{ii} u_i^2[u_i^2 - (\tilde{u}_i^*)^2]  \\
&\leq {\sum}_{i,j\geq2,i\neq j}\tilde{C}_{ij} u_i u_j \left[ 1/4  - (1-\delta)^2 \right]\\
&\hspace{14em} + {\sum}_{i\geq 2} \tilde{C}_{ii} u_i^2 [1/4 - (1-\delta)^2 ] \\
&\leq {\sum}_{i,j\geq2,i\neq j}(1-\delta) \cdot (8R)^{-2} \cdot \left( 1/4  - 1/2 \right)\\
&\hspace{12em} + {\sum}_{i\geq 2} \delta \cdot (8R)^{-2} \cdot ( 1/4 - 1/2 ) < 0,
\end{align*}
which contradicts the assumption that $x$ is a stationary point. Therefore, the above analysis implies that the instance $(\tilde{C},\tilde{u}^*)$ has no stationary point in the region $\{u\in\mathbb{R}^n~|~\|u\|_2 > 8n+1\}$.

Now, We focus on the compact region $\{u\in\mathbb{R}^n~|~\|u\|_2 \leq 8n+1\}$. Since the gradient and the Hessian matrix are continuous functions of $(C,u^*)$, the $\ell_\infty$-norm of the gradient and the eigenvalues of the Hessian matrix are also continuous functions of $(C,u^*)$. Intuitively, a small perturbation to $(C,u^*)$ would not significantly change the norms of the gradient and the Hessian matrix. Thus, the strict-saddle property still holds after a small perturbation. More rigorously, let $(C^0,u^0)\in\cM$ and $\eta \in (0,\eta_0]$. In the region
\[ \mathcal{R}_\eta := \{ u\in\mathbb{R}^n~|~ \|u\|_2 \leq 8n+1, \|u - u^0\|_1 \geq \eta, \|u + u^0\|_1 \geq \eta \}, \]
at least one of the following properties holds:
\[ \|\nabla g(u;C^0,u^0)\|_\infty \geq \beta(\eta),\quad \lambda_{min}[\nabla^2 g(u;C^0,u^0)] \leq - \gamma(\eta). \]
Since $\mathcal{R}_\eta$ is a compact set and we constrain $(C,u^*)$ by $\|C\|_1 = 1$ and $\|u^*\|_1 = 1$, the functions
\[ \|\nabla g(u;C,u^*)\|_\infty \quad \text{and} \quad \lambda_{min}[\nabla^2 g(x;C,u^*)] \]
are Lipschitz continuous in $(C,u^*)$. Suppose that the Lipschitz constants are $L_g$ and $L_H$ under the weighted $\ell_1$-norm, namely
\begin{align*} 
&\left| \|\nabla g(u;C,u^*)\|_\infty - \|\nabla g(u;\tilde{C},\tilde{u}^*)\|_\infty \right|\\
&\hspace{14em} \leq L_g \left[ \alpha^*\|\tilde{C} - C\|_1 + (1-\alpha^*) \|\tilde{u}^* - u^*\|_1 \right],\\
&\left| \lambda_{min}[\nabla^2 g(u;C,u^*)] - \lambda_{min}[\nabla^2 g(u;\tilde{C},\tilde{u}^*)] \right|\\
&\hspace{14em} \leq L_H \left[ \alpha^*\|\tilde{C} - C\|_1 + (1-\alpha^*) \|\tilde{u}^* - u^*\|_1 \right],\\
&\hspace{15em} \forall x\in\mathcal{R}_\eta,~ (C,u^*)\quad \mathrm{s.t.}\quad \|C\|_1 = \|u^*\|_1 = 1. 
\end{align*}
Let 
\[ \epsilon := \min\left\{ \frac{\beta(\eta)}{2L_g}, \frac{\gamma(\eta)}{2L_H} \right\}. \]
Then, for every pair $(\tilde{C},\tilde{u}^*)$ satisfying
\[ \alpha^*\|\tilde{C} - C^0\|_1 + (1-\alpha^*) \|\tilde{u}^* - u^0\|_1 < \epsilon, \]
at least one of the following properties holds for all $x\in\mathcal{R}_\eta$:
\[ \|\nabla g(u;\tilde{C},\tilde{u}^*)\|_\infty \geq \beta(\eta)/2,\quad \lambda_{min}[\nabla^2 (u;\tilde{C},\tilde{u}^*)] \leq - \gamma(\eta)/2. \]
This implies that the strict-saddle property holds for the the perturbed instance $\cMC(\tilde{C}, \tilde{u}^*)$. Letting $\eta \rightarrow 0$, it follows that $\pm \tilde{u}^*$ are the only points satisfying the second-order necessary optimality conditions, and thus $\cMC(\tilde{C}, \tilde{u}^*)$ does not have SSCPs.
\qed\end{proof}

\subsection{Proof of Theorem \ref{thm:large-1}}
\label{adp:large-1}

The proof of Theorem \ref{thm:large-1} directly follows from the next two lemmas.
%
\begin{lemma}\label{lem:large-1}
Suppose that $(C,u^*)\in\cSD$ and that $u^0$ is a global solution to $\cMC(C,u^*)$. Then, for all $k\in[n_1]$, it holds that $u_i^0u_j^0 = u_i^*u_j^*$ for all $i,j\in\cI_{1k}$. In addition, $u^0_i = 0$ for all $i\in\cI_0(C,u^*)$.
\end{lemma}
\begin{proof}
Denote $M^* := u^*(u^*)^T$. 
We first consider nodes in $\mathcal{G}_{1k}$ for some $k\in[n_1]$. Since the subgraph is not bipartite, there exists a cycle with an odd length $2\ell+1$, which we denote as
\[ \{i_1,\dots,i_{2\ell+1}\}. \]
Then, we have
\begin{align*} 
    (u^0_{i_1})^2 &= {\prod}_{s=1}^{2\ell+1} (u^0_{i_s}u^0_{i_{s+1}})^{(-1)^{s-1}} = {\prod}_{s=1}^{2\ell+1}(M^*_{i_si_{s+1}})^{(-1)^{s-1}}\\
    &= {\prod}_{s=1}^{2\ell+1} (u^*_{i_s}u^*_{i_{s+1}})^{(-1)^{s-1}} = (u_{i_1}^*)^2,
\end{align*}
which implies that the conclusion holds for $i=j=i_1$. Using the connectivity of $\mathbb{G}_{1k}(C,u^*)$, we know
\[ u^0_{i}u^0_{j} = u^*_{i}u^*_{k},\quad \forall i,j\in\mathcal{I}_{1k}(C,u^*). \]
%
Then, we consider nodes in $\mathcal{I}_0(C,u^*)$. Since $\cI_{00}(C,u^*)$ is empty, for every node $i\in\mathcal{I}_0(C,u^*)$, there exists another node $j\in\mathcal{I}_1(C,u^*)$ such that $C_{ij}>0$. Hence, we have
\[ u^0_i = {M_{ij}^*}/{u^0_j} = 0. \]
This completes the proof.
\qed\end{proof}

The following lemma provides a necessary and sufficient condition for instances with a positive definite Hessian matrix at global solutions, which is stronger than what Theorem \ref{thm:large-1} requires.
\begin{lemma}\label{lem:large-2}
Suppose that $u^0\in\mathbb{R}^n$ is a global minimizer of the instance $\cMC(C,u^*)$ such that the conditions in Lemma \ref{lem:large-1} hold. Then, the Hessian matrix is positive definite at $u^0$ if and only if
\begin{enumerate}
	\item $\mathbb{G}_{1i}(C,u^*)$ is \revise{not bipartite} for all $i\in[n_1]$;
	\item $\mathcal{I}_{00}(C,u^*) = \emptyset$.
\end{enumerate}
\end{lemma}
\begin{proof}
We first construct counterexamples for the necessity part and then prove the positive definiteness of the Hessian matrix for the sufficiency part.

\paragraph{Necessity.} 
We construct counterexamples by discussing two different cases.
\paragraph{Case I.} We first consider the case when there exists $k\in[n_1]$ such that $\mathbb{G}_{1k}(C,u^*)$ is bipartite. Suppose that $\mathbb{G}_{1i}(C,u^*) = \mathbb{G}_{1k1} \cup \mathbb{G}_{1k2}$ is a partition of $\mathbb{G}_{1k}(C,u^*)$. Let the sets $\mathcal{I}_{1k},\mathcal{I}_{1k1}$ and $\mathcal{I}_{1k2}$ be the node sets of the corresponding graphs. Define $q\in\mathbb{R}^n$ as
\[ q_i := u^0_i ,\quad \forall i\in\mathcal{I}_{1k1};\quad q_i := -u^0_i,\quad \forall i\in\mathcal{I}_{1k2};\quad q_i := 0,\quad \forall i\notin \mathcal{I}_{1k}. \]
Then, the curvature of the Hessian along the direction $q$ is
\begin{align*}
&\frac14 [\nabla^2 g(u^0;C,u^*)](q,q)\\
=& \sum_{i\in \mathcal{I}_{1k1},j \in \mathcal{I}_{1k2}} C_{ij}\left[(u_i^0)^2q_j^2 + (u_j^0)^2q_i^2 \right] + 2\sum_{i\in \mathcal{I}_{1k1},j \in \mathcal{I}_{1k2}} C_{ij}(2u_i^0u_j^0 - u_i^*u_j^*) q_iq_j\\
=& \sum_{i\in \mathcal{I}_{1k1},j \in \mathcal{I}_{1k2}} C_{ij}\left[(u_i^0)^2q_j^2 + (u_j^0)^2q_i^2 \right] + 2\sum_{i\in \mathcal{I}_{1k1},j \in \mathcal{I}_{1k2}} C_{ij}u_i^0u_j^0q_iq_j\\
=& \sum_{i\in \mathcal{I}_{1k1},j \in \mathcal{I}_{1k2}} 2C_{ij}(u_i^0u_j^0)^2 - 2\sum_{i\in \mathcal{I}_{1k1},j \in \mathcal{I}_{1k2}} C_{ij}(u_i^0u_j^0)^2 = 0.
\end{align*}
We note that there is no self-loop in $\mathcal{G}_{1k}(C,u^*)$ and, thus, the diagonal entries of the weight matrix are equal to $0$. Therefore, the Hessian matrix has a zero curvature along $q$ and is not positive definite.

\paragraph{Case II.} We consider the case when $\mathcal{I}_{00}(C,u^*) \neq \emptyset$. Suppose that $k\in\mathcal{I}_{00}(C,u^*)$. Define the vector $q\in\mathbb{R}^n$ as
\[ q_k := 1;\quad q_i := 0,\quad \forall i\neq k. \]
The curvature of the Hessian along the direction $q$ is
\begin{align*}
\frac14 [\nabla^2 g(u^0;C,u^*)](q,q) &= C_{kk}\left[2(u^0_k)^2 - (u_k^*)^2 \right] q_k^2 + \sum_{j\in\cI_0(C,u^*), j\neq k} C_{kj} (u^0_{j})^2q_k^2\\
&= C_{kk}(u^0_k)^2 + \sum_{j\in\cI_0(C,u^*), j\neq k} C_{kj} (u^0_{j})^2 = 0.
\end{align*}
Therefore, the Hessian matrix is not positive-definite at $u^0$.

\paragraph{Sufficiency.} Next, we consider the sufficiency part, namely, we prove that the Hessian matrix is positive definite under the two conditions stated in the theorem. Suppose that there exists a nonzero vector $q\in\mathbb{R}^n$ such that
\[ [\nabla^2 g(u^0;C,u^*)](q,q) = 0. \]
Then, after straightforward calculations, we arrive at
\begin{align*}
u^0_i q_j + u^0_j q_i = &0 ,\quad \forall i,j \quad \st\quad C_{ij} > 0,~ i\neq j;\\
[2(u^0_i)^2 - (u^*_i)^2]q_i^2 = (u^0_iq_i)^2 = &0 ,\quad \forall i \quad \mathrm{s.t.}\quad C_{ii} > 0.
\end{align*}
The two conditions can be written compactly as
\begin{align}\label{eqn:gen-3}
u^0_i q_j + u^0_j q_i = &0 ,\quad \forall i,j \quad \mathrm{s.t.}\quad C_{ij} > 0.
\end{align}
Consider the index set $\mathcal{I}_{1k}(C,u^*)$ for some $k\in[n_1]$. The equality \eqref{eqn:gen-3} implies that
\begin{align}\label{eqn:gen-4} 
{q_i}/{u^0_i} + {q_j}/{u^0_j} = 0,\quad \forall i,j \in \mathcal{I}_{1k}(C,u^*).
\end{align}
Since the graph $\mathbb{G}_{1k}(C,u^*)$ is \revise{not bipartite}, there exists a cycle with an odd length $2\ell+1$, which we denote as
\[ \{i_1,i_2,\dots,i_{2\ell+1}\}. \]
Denoting $i_{2\ell+2} := i_1$, we can calculate that
\begin{align*}
2\frac{q_{i_1}}{u^0_{i_1}} = \sum_{s=1}^{2\ell+1} (-1)^{s-1} \left(\frac{q_{i_s}}{u^0_{i_s}} + \frac{q_{i_{s+1}}}{u^0_{i_{s+1}}}\right) = 0,
\end{align*}
which leads to $q_{i_1} = 0$. Using the connectivity of $\mathcal{G}_{1k}$ and the relation \eqref{eqn:gen-4}, it follows that
\[ q_{i} = 0,\quad \forall i\in\cI_{1k}(C,u^*). \]
Moreover, the same conclusion holds for all $k\in[n_1]$ and, thus, we conclude that
\[ q_{i} = 0,\quad \forall i\in\mathcal{I}_1(C,u^*). \]
Since $\mathcal{I}_{00}(C,u^*) = \emptyset$, for every node $i\in\mathcal{I}_0(C,u^*)$, there exists another node $j\in\mathcal{I}_1(C,u^*)$ such that $C_{ij} > 0$. Considering the relation \eqref{eqn:gen-4}, we obtain that
\[ q_j = -{u^0_jq_i}/{u^0_i} = 0. \]
In summary, we have proved that $q_i = 0$ for all $i\in[n]$, which contradicts the assumption that $q\neq 0$. Hence, the Hessian matrix at $u^0$ is positive definite.
\qed\end{proof}

\subsection{Application of the implicit function theorem}
\label{apd:implicit}

Using the positive-definiteness of the Hessian matrix, we are able to apply the implicit function theorem to certify the existence of spurious local minima.
\begin{lemma}\label{lem:local}
Suppose that $\alpha\in[0,1]$ and consider a pair $(C,u^*)\in\cSD$. Then, there exists a small constant $\delta(C,x^*) > 0$ such that for every instance $\cMC(\tilde{C},\tilde{u}^*)$ satisfying
\[ \alpha\|\tilde{C} - C\|_1 + (1-\alpha)\|\tilde{u}^* - u^*\|_1 < \delta(C,u^*), \]
the instance $\cMC(\tilde{C},\tilde{u}^*)$ has spurious local minima.
\end{lemma}
\begin{proof}
By Theorem \ref{thm:large-1}, there exists a global solution $u^0$ to the instance $\cMC(C,u^*)$ such that
\[ u^0(u^0)^T \neq u^*(u^*)^T,\quad \nabla^2 g(u^0;C,u^*) \succ 0. \]
Consider the system of equations:
\begin{align*}
    \nabla g(u;C,u^*) = 0. 
\end{align*}
%
Since the Jacobi matrix of $\nabla g(u;C,u^*)$ with respect to $u$ is the Hessian matrix $\nabla^2 g(u;C,u^*)$ and $(u^0, C, u^*)$ is a solution, the implicit function theorem guarantees that there exists a small constant $\delta(C,u^*)>0$ such that in the neighborhood
\begin{align*} 
    \cN := \left\{ (\tilde{C}, \tilde{u}^*) ~\big|~  \alpha\|\tilde{C} - C\|_1 + (1-\alpha)\|\tilde{u}^* - u^*\|_1 < \delta(C,u^*) \right\}, 
\end{align*}
there exists a function $u(\tilde{C}, \tilde{u}^*):\cN\mapsto \Rn$ such that
\begin{enumerate}
	\item $u(C,u^*) = u^0$;
	\item $u(\cdot,\cdot)$ is a continuous function in $\cN$;
	\item $\nabla g[u(\tilde{C}, \tilde{u}^*); \tilde{C}, \tilde{u}^*] = 0$.
\end{enumerate}
%
Using the continuity of the Hessian matrix and $u(\cdot,\cdot)$, we can choose $\delta(C,u^*)$ to be small enough such that
\[ u(\tilde{C},\tilde{u}^*)[u(\tilde{C},\tilde{u}^*)]^T \neq \tilde{u}^*(\tilde{u}^*)^T,\quad \nabla^2 g\left[u(\tilde{C}, \tilde{u}^*); \tilde{C}, \tilde{u}^*\right] \succ 0,\quad \forall (\tilde{C},\tilde{u}^*) \in \cN. \]
Therefore, the point $u(\tilde{C}, \tilde{u}^*)$ is a spurious local minimum of the instance $\cMC(\tilde{C}, \tilde{u}^*)$.
\qed\end{proof}

\section{Analysis for the asymmetric case}
\label{sec:asym}

In this section, we extend the analysis of the \textit{symmetric} weighted matrix completion problem \eqref{eqn:gmc} to the \textit{asymmetric} weighted matrix completion problem, which is defined as
\begin{align}
\label{eqn:gmc-asym}
    \min_{u\in\R^m, v\in\R^n} {\sum}_{i\in[m], j\in[n]} C_{ij}(u_iv_j - M^*_{ij})^2,
\end{align}
where $M^*\in\R^{m\times n}$ is the ground truth matrix and $C\in\R^{m\times n}$ is the weight matrix. We note that in the asymmetric case, we do not assume that the weight matrix $C$ is symmetric. Similar to the symmetric case, we assume that $M^* = u^*(v^*)^T$ has rank-$1$, where $u^*\in\R^m$ and $v^*\in\R^n$. We denote each instance of problem \eqref{eqn:gmc-asym} as $\cMC(C, u^*, v^*)$, where $C$ is the weight matrix and $u^*(v^*)^T$ is the ground truth matrix. Moreover, since the degenerate instances where $C=0$ or $M^* = 0$ can be easily analyzed separately, we utilize the ``scale-free'' property of problem \eqref{eqn:gmc-asym} and extend the normalization assumption (Assumption \ref{asp:norm}) to the asymmetric case:
\begin{assumption}
Assume that $C\in\Smn_{+,1}$, $u^*\in\Smm_1$ and $v^*\in\Snn_1$, i.e., $\|C\|_1 = \|u^*\|_1 = \|v^*\|_1 = 1$.
\end{assumption}
Define the objective function of problem \eqref{eqn:gmc-asym} as
\[ h(u,v; C, u^*, v^*) := {\sum}_{i\in[m], j\in[n]} C_{ij}(u_iv_j - u^*_iv^*_j)^2. \]
Then, the set of degenerate instances is defined as
\begin{align*}
    \cD_{asym} := \{ &(C, u^*, v^*) ~|~ C\in\Sn_{+,1},u^*\in\Snn_1,v^*\in\Smm_1,\\
    &\exists u\in\R^{m},~v\in\R^n\quad \st\quad h(u, v; C,u^*,v^*) = 0,~ uv^T \neq u^*(v^*)^T \}.
\end{align*}
Using graphical notations, we can establish an exact characterization for the set $\cD_{asym}$. The weighted graph $\mathbb{G}(C,u^*,v^*)=[\mathbb{V}(C,u^*,v^*),\mathbb{E}(C,u^*,v^*),\mathbb{W}(C,u^*,v^*)]$ is defined by
\begin{align*}
    &\mathbb{V}(C,u^*,v^*) := [m+n],~\\
    &\mathbb{E}(C,u^*,v^*) := \left\{ \{i,j+m\}~|~ C_{ij} > 0, i\in[m],j\in[n] \right\},\\
    &[\mathbb{W}(C,u^*,v^*)]_{i,j+m} := C_{ij},\quad \forall i\in[m],j\in[n]\quad\st~\{i,j+m\}\in\mathbb{E}(C,u^*,v^*).
\end{align*}
%
To include the information of $u^*$ and $v^*$, we define
\begin{align*} 
\mathcal{I}_1^u(C,u^*,v^*) &:= \{i\in[m] ~|~ u_{i}^* \neq 0\},\quad \mathcal{I}_0^u(C,u^*,v^*) := [m]\backslash \mathcal{I}_1^u(C,u^*,v^*),\\
\mathcal{I}_1^v(C,u^*,v^*) &:= \{j+m ~|~ j\in[m], v_{j}^* \neq 0\},\\
\mathcal{I}_0^v(C,u^*,v^*) &:= \{m+1,\dots,m+n\}\backslash \mathcal{I}_1^v(C,u^*,v^*),\\
\mathcal{I}^u_{00}(C,u^*,v^*) &\\
&\hspace{-3em}:= \{ i\in\mathcal{I}_0^u(C,u^*,v^*) ~|~ \{ i,j+m \}\notin\mathbb{E}(C,u^*,v^*) ,~ \forall j\in\mathcal{I}^v_1(C,u^*,v^*) \},\\
\mathcal{I}^v_{00}(C,u^*,v^*) &\\
&\hspace{-4em}:= \{ j+m\in\mathcal{I}_0^v(C,u^*,v^*) ~|~ \{ i,j+m \}\notin\mathbb{E}(C,u^*,v^*) ,~ \forall i\in\mathcal{I}^u_1(C,u^*,v^*) \}.
\end{align*}
The sub-graph $\mathbb{G}_1(C,u^*,v^*)$ is induced by $\cI_1^u(C,u^*,v^*)\cup\cI_1^v(C,u^*,v^*)$. The following theorem provides necessary and sufficient conditions for instances in $\cD_{asym}$ and $\overline{\cD}_{asym}$.
\begin{theorem}\label{thm:unique-asym}
Given $C\in\Smn_{+,1}$, $u^*\in\Smm_1$ and $v^*\in\Snn_1$, it holds that $(C,u^*,v^*)$ does not belong to $\cD_{asym}$ if and only if
\begin{enumerate}
	\item $\mathbb{G}_1(C,u^*,v^*)$ is connected;
	\item $\mathcal{I}^u_{00}(C,u^*,v^*) = \emptyset$ and $\mathcal{I}^v_{00}(C,u^*,v^*) = \emptyset$.
\end{enumerate}
Moreover, the following relation holds:
\begin{align*} 
&\overline{\cD}_{asym} \\
=& \{ (C,u^*,v^*) ~|~C\in\Smn_{+,1},u^*\in\Smm_1,v^*\in\Snn_1, \mathbb{G}_1(C,u^*,v^*) \text{ is disconnected} \}\\
\cup& \{ (C,u^*,v^*) ~|~ C\in\Smn_{+,1},u^*\in\Smm_1,v^*\in\Snn_1, \cI_{00}^u(C,u^*,v^*) \text{ is not empty} \}\\
\cup& \{ (C,u^*,v^*) ~|~ C\in\Smn_{+,1},u^*\in\Smm_1,v^*\in\Snn_1, \cI_{00}^v(C,u^*,v^*) \text{ is not empty} \}.
\end{align*}
\end{theorem}
The proof of Theorem \ref{thm:unique-asym} is based on a slight modification of the proof of Theorem \ref{thm:unique} and therefore, we omit the proof here. Similarly, the proofs of all subsequent theorems in this section follow directly from those of the symmetric case and are omitted for brevity. The new complexity metric for the asymmetric problem is given by
\begin{align*}
    &\mathbb{D}_\alpha^{asym}(C,u^*,v^*)\\
    \nonumber:= &\left[ \inf_{(\tilde{C},\tilde{u}^*,\tilde{v}^*)\in\cD_{asym}} \alpha \|C - \tilde{C}\|_1 + (1-\alpha) (\|u^* - \tilde{u}^*\|_1 + \|v^* - \tilde{v}^*\|_1) \right]^{-1}\\
    \nonumber= &\left[ \min_{(\tilde{C},\tilde{u}^*,\tilde{v}^*)\in\overline{\cD}_{asym}} \alpha \|C - \tilde{C}\|_1 + (1-\alpha) (\|u^* - \tilde{u}^*\|_1 + \|v^* - \tilde{v}^*\|_1) \right]^{-1}.
\end{align*}

\subsection{Connection to existing results}

Now, we derive upper bounds on the complexity metric under several different existing conditions. We first develop an upper bound on the complexity metric under the RIP condition, which is stated in the following theorem.
\begin{theorem}\label{thm:rip-1-asym}
Suppose that $\delta\in[0,1)$ is a constant and the instance $\cMC(C,u^*,v^*)$ satisfies the $\delta$-RIP$_{2,2}$ condition. Then, it holds that
\[ \mathbb{D}_\alpha^{asym}(C,u^*,v^*) \leq \frac{mn(1+\delta) - 2\delta}{2\alpha(1-\delta)}. \]
The maximum complexity is attained by the instance $\cMC(C^\delta,u^\delta,v^\delta)$, where
\begin{align*} 
C_{11}^\delta &= \frac{1-\delta}{(1+\delta)mn - 2\delta};\quad C_{ij}^\delta = \frac{1+\delta}{(1+\delta)mn - 2\delta},~\forall (i,j)\in[m]\times[n]\backslash\{(1,1)\};\\
u_1^\delta &= 1;\quad u_i^\delta = 0,\quad \forall i\geq 2,\quad v_1^\delta = 1;\quad v_j^\delta = 0,\quad \forall j\geq 2.  
\end{align*}
\end{theorem}
We note that the upper bound in Theorem \ref{thm:rip-1-asym} is $O(\min\{m,n\})$ larger than the smallest possible complexity, which is $O(\max\{m,n\})$. Following the same path as in the symmetric case, we improve the upper bound using the incoherence information. We first give the definition of the incoherence in the asymmetric case.
\begin{definition}[\cite{jain2013low}]
Given constants $\mu_1\in[1,m]$ and $\mu_2\in[1,n]$, the ground truth matrix $M^*\in\R^{m\times n}$ is said to be \textbf{$(\mu_1,\mu_2)$-incoherent} if
\begin{align*}
\| (e_i^m)^T U^* \|_F \leq \sqrt{ {\mu_1 r}/{m} },\quad \forall i\in[m],\quad \| (e_j^n)^T V^* \|_F \leq \sqrt{ {\mu_2 r}/{n} },\quad \forall j\in[n], 
\end{align*}
where $U^*\Sigma^* (V^*)^T$ is the truncated SVD of $M^*$, $e_i^m$ is the $i$-th standard basis of $\R^m$ and $e_j^n$ is the $j$-th standard basis of $\Rn$. Moreover, the ground truth matrix $M^*\in\R^{m\times n}$ is said to be \textbf{$\mu$-incoherent} if it is $(\mu_1,\mu_2)$-incoherent with some $\mu_1,\mu_2 \leq \mu$.
\end{definition}
As a counterpart of Theorem \ref{thm:rip-2}, the upper bound can be improved to $O[\mu\max\{m, n\}]$.  
\begin{theorem}\label{thm:rip-2-asym}
Suppose that the instance $\cMC(C,u^*,v^*)$ satisfies the $\delta$-RIP$_{2,2}$ condition and $u^*(v^*)^T$ is $(\mu_1,\mu_2)$-incoherent. Then, it holds that
\begin{align*} 
&\mathbb{D}_\alpha^{asym}(C,u^*,v^*) \leq \max\left\{\frac{\max\{\rho_1,\rho_2\}mn(1+\delta)}{4\alpha(1-\delta)}, \frac{1}{2(1-\alpha)} \right\}\\
&\hspace{16em} \times \min\left\{ \left(1 - \max\{\rho_1,\rho_2\}\right)^{-1}, 3 \right\},
\end{align*}
where $\rho_1 := \mu_1 / m$ and $\rho_2 := \mu_2 / n$. Moreover, suppose that the instance $\cMC(C,u^*,v^*)$ satisfies the $\delta$-RIP$_{2,2}$ condition and $u^*(v^*)^T$ is  $\mu$-incoherent. Then, it holds that
\begin{align*} 
&\mathbb{D}_\alpha^{asym}(C,u^*,v^*) \leq \max\left\{\frac{\max\{m,n\}(1+\delta)}{4\alpha(1-\delta)}, \frac{1}{2(1-\alpha)\mu} \right\}\\
&\hspace{16em} \times \min\left\{ \left(\frac{1}{\mu} - \frac{1}{\min\{m, n\}}\right)_+^{-1}, 3\mu \right\},
\end{align*}
where we define $x_+ := \max\{x,0\}$ and $1 / 0 = +\infty$.
\end{theorem}
If we choose $1-\alpha = \Theta (n^{-1})$, the complexity can be upper-bounded by
\[ \mathbb{D}^{asym}_\alpha(C,u^*,v^*) = O\left( \mu \max\{m,n\} \cdot  \frac{1+\delta}{1-\delta} \right). \]
In the case when $1-\delta=\Theta(1)$ and $\mu = O(1)$, the upper bound is on the same order (i.e., $O(\max\{m,n\})$) as the minimum possible complexity.

Next, we consider the case when components of $M^*$ are observed under the Bernoulli model with parameter $p$.
\begin{theorem}\label{thm:incoh-asym}
Given $\mu\in[1,n]$ and $p\in(0,1]$, suppose that the weight matrix $C$ obeys the Bernoulli model with the parameter $p$ and that $u^*$ has incoherence $\mu$. If $\eta>2$ is a constant and the sampling rate satisfies
\[ p \geq  \min\left\{ 1,  \frac{(m+n)[16(1+\eta\mu)\log(mn) + 16]}{mn} \right\}, \]
then it holds with probability at least $1 - O[(mn)^{-\eta/2 + 1}]$ that
\begin{align*} 
&\mathbb{D}^{asym}_\alpha(C,u^*,v^*) \leq \max \left\{ \frac{3\max\{m,n\}}{4\alpha}, \frac{1}{2(1-\alpha)\mu} \right\}\\
&\hspace{16em} \times \min\left\{ \left(\frac{1}{\mu} - \frac{1}{\min\{m, n\}}\right)_+^{-1}, 3\mu \right\}.
\end{align*}
%
\end{theorem}
In the case when $1-\alpha = \Theta(n^{-1})$ and $\mu = O(1)$, the upper bound is on the order of $O(\max\{m,n\})$, which is also the same as the minimum possible complexity.

\subsection{Theoretical results}

Now, we extend the theoretical results in Section \ref{sec:theory} to the asymmetric case. We first prove that if the complexity metric is on the order of $O(\max\{m,n\})$, there does not exist spurious second-order critical point. This result is established in the case when we choose $\alpha = \alpha^*_{asym}$, where $\alpha^*_{asym}$ is the minimizer of the minimum possible complexity metric:
\[ \mathbb{D}_\alpha^{min, asym} := \min_{C\in\Smn_{+,1}, u^*\in\Smm_1, v^*\in\Snn_1 }~ \mathbb{D}^{asym}_\alpha(C,u^*,v^*). \]
The following theorem provides a characterization of the complexity metric when $\alpha=\alpha^*_{asym}$.
\begin{theorem}
It holds that
\[ \alpha^*_{asym} = 1 - \frac{1}{\max\{m,n\} +1},\quad \mathbb{D}_{\alpha^*_{asym}}^{min, asym} = \frac{\max\{m,n\}}{2\alpha^*_{asym}}. \]
Moreover, the complexity metric $\mathbb{D}_{\alpha^*_{asym}}^{asym}(C,u^*,v^*)$ is equal to $\mathbb{D}_{\alpha^*_{asym}}^{min, asym}$ if and only if
\begin{align*}
    C_{ij} = \frac{1}{mn},\quad \forall i\in[m],~j\in[n],\quad u_i^* = \frac1m,\quad\forall i\in[m],\quad v_j^* = \frac1n,\quad \forall j\in[n].
\end{align*}
\end{theorem}
The next theorem states that the optimization landscape is benign when the complexity is close to $\mathbb{D}_{\alpha^*_{asym}}^{min, asym}$.
\begin{theorem}
Suppose that $\alpha = \alpha^*_{asym}$. Then, there exists a constant $\delta > 1/2$ such that for every instance $\cMC(C,u^*,v^*)$ satisfying
\[ \mathbb{D}^{asym}_{\alpha^*_{asym}}(C,u^*,v^*) \leq {\delta \max\{m,n\}}/{\alpha^*_{asym}}, \]
the instance $\cMC(C,u^*,v^*)$ does not have any SSCPs.
\end{theorem}

Next, we consider instances with a large complexity. We note that the landscape of problem \eqref{eqn:gmc-asym} is ``scale-invariant''. Namely, if $(u,v)$ is a stationary point of problem \eqref{eqn:gmc-asym}, the scaled point $(c_1 u, c_1^{-1} v)$ is also a stationary point of problem \eqref{eqn:gmc-asym} for all constants $c_1 \neq 0$. To deal with this problem, consider a regularized version of problem \eqref{eqn:gmc-asym}:
\begin{align}
\label{eqn:gmc-asym-reg}
    \min_{u\in\R^m, v\in\R^n} {\sum}_{i\in[m], j\in[n]} C_{ij}(u_iv_j - M^*_{ij})^2 + \lambda (u^Tu - v^Tv)^2,
\end{align}
where $\lambda > 0$ is the regularization parameter. We denote instances of problem \eqref{eqn:gmc-asym-reg} as $\cMC_{reg}(C,u^*,v^*)$. It is proved in \cite{zhu2018global} that problems \eqref{eqn:gmc-asym} and \eqref{eqn:gmc-asym-reg} are equivalent in the sense that they have the same set of local minima under scaling; see \cite{zhang2021general} for a more detailed discussion. We note that adding the regularizer to problem \eqref{eqn:gmc-asym} will not affect the existence of SSCPs, and we consider problem \eqref{eqn:gmc-asym-reg} since it is desirable to construct degenerate instances with a positive definite Hessian matrix at global minima. Similar to the symmetric case, we define the following subset of $\cD_{asym}$:
\begin{align*}
    \cSD_{asym} := \{ (C,u^*,v^*)\in\cD_{asym} ~|~ &\mathbb{G}_1(C,u^*,v^*)\text{ is disconnected},\\
    &\mathcal{I}_{00}^u(C,u^*,v^*) = \mathcal{I}_{00}^v(C,u^*,v^*) = \emptyset\}.
\end{align*}
The following theorem proves that the Hessian matrix is positive definite at global solutions for instances in $\cSD_{asym}$.
\begin{theorem}\label{thm:large-1-asym}
Suppose that $(C,u^*,v^*)\in\cSD_{asym}$. Then, the Hessian matrix of the objective function of problem \eqref{eqn:gmc-asym-reg} is positive definite at all global solutions of the instance $\cMC(C,u^*,v^*)$.
\end{theorem}
The next step is to consider a closed subset of $\cSD_{asym}$, which is defined as
\begin{align*} 
    \mathcal{SD}_{asym,\epsilon} := \big\{ (C,u^*,v^*) \in\mathcal{SD}_{asym} ~|~  &C_{ij} \in \{0\} \cup [\epsilon,1],\quad \forall i\in[m],~ j\in[n],\\
    &\hspace{-10em} |u_i^*| \in \{0\} \cup [\epsilon,1],\quad \forall i\in[m],\quad |v_j^*| \in \{0\} \cup [\epsilon,1],\quad \forall j\in[n] \big\}.
\end{align*}
Define the alternative complexity metric as
\begin{align*}
    &\mathbb{D}_{\alpha,\epsilon}^{asym}(C,u^*,v^*)\\
    := &\left[ \min_{(\tilde{C},\tilde{u}^*,\tilde{v}^*)\in\cSD_{asym,\epsilon}} \alpha \|C - \tilde{C}\|_1 + (1-\alpha)(\|u^* - \tilde{u}^*\|_1 + \|v^* - \tilde{v}^*\|_1) \right]^{-1}.
\end{align*}
The new metric $\mathbb{D}_{\alpha,\epsilon}^{asym}$ is a lower bound on the original metric $\mathbb{D}_{\alpha}^{asym}$. The following theorem provides a sufficient condition on the existence of spurious local minima for problems \eqref{eqn:gmc-asym} and \eqref{eqn:gmc-asym-reg}.
\begin{theorem}\label{thm:global-asym}
Suppose that $\epsilon > 0 $ is a constant. Then, there exists a large constant $\Delta(\epsilon)>0$ such that for every instance $\cMC(C,u^*,v^*)$ satisfying
\[ \mathbb{D}_{\alpha,\epsilon}^{asym}(C,u^*,v^*) \geq \Delta(\epsilon), \]
both instances $\cMC(C,u^*,v^*)$ and $\cMC_{reg}(C,u^*,v^*)$ have spurious local minima.
\end{theorem}

\end{appendix}

\begin{acknowledgements}
This work was supported by grants from ARO, AFOSR, ONR and NSF.
\end{acknowledgements}

%
%

\end{document}